\documentclass[11pt]{amsart}

\usepackage{amssymb}
\usepackage[cmtip,all]{xy}
\usepackage{verbatim}
\usepackage{upref}
\usepackage{url}
\usepackage[normalem]{ulem}
\usepackage[usenames,dvipsnames]{color}
\usepackage{tikz}
\usetikzlibrary{arrows}

\allowdisplaybreaks

\newtheorem{thm}{Theorem}[section]
\newtheorem{lem}[thm]{Lemma}
\newtheorem{cor}[thm]{Corollary}
\newtheorem{prop}[thm]{Proposition}

\theoremstyle{definition} 
\newtheorem{defn}[thm]{Definition}
\newtheorem{ex}[thm]{Example}

\newtheorem{rem}[thm]{Remark}

\numberwithin{equation}{section}
\renewcommand{\theenumi}{\roman{enumi}}

\allowdisplaybreaks
\newcommand{\secref}[1]{Section~\textup{\ref{#1}}}

\newcommand{\thmref}[1]{Theorem~\textup{\ref{#1}}}
\newcommand{\corref}[1]{Corollary~\textup{\ref{#1}}}
\newcommand{\lemref}[1]{Lemma~\textup{\ref{#1}}}
\newcommand{\propref}[1]{Proposition~\textup{\ref{#1}}}

\newcommand{\defnref}[1]{Definition~\textup{\ref{#1}}}

\newcommand{\exref}[1]{Example~\textup{\ref{#1}}}

\newcommand{\KK}{\mathcal K}

\newcommand{\OO}{\mathcal O}
\newcommand{\TT}{\mathcal T}
\newcommand{\CC}{\mathcal C}
\newcommand{\DD}{\mathcal D}
\newcommand{\EE}{\mathcal E}

\newcommand{\Z}{\mathbb Z}
\newcommand{\C}{\mathbb C}

\renewcommand{\bar}{\overline}

\newcommand{\inv}{^{-1}}
\newcommand{\<}{\langle}
\renewcommand{\>}{\rangle}

\newcommand{\Chi}{\raisebox{2pt}{\ensuremath{\chi}}}
\renewcommand{\epsilon}{\varepsilon}
\newcommand{\case}{& \text{if }}

\newcommand{\minus}{\setminus}

\renewcommand{\:}{\colon}
\renewcommand{\subset}{\subseteq}

\renewcommand{\)}{\textup)}

\DeclareMathOperator*{\spn}{span}
\DeclareMathOperator*{\clspn}{\overline{\spn}}

\newcommand{\arabicnum}{\renewcommand{\theenumi}{\arabic{enumi}}}

\newcommand{\exelist}{\arabicnum\renewcommand{\labelenumi}{(\textbf{\exelabel} \theenumi)}}
\newcommand{\exelabel}{\textbf{Exe}}
\newcommand{\exe}[1]{\textup{(\exelabel\ {\ref{#1}})}}

\newcommand{\eplabel}{\textbf{EP}}

\newcommand{\eplistprime}\arabicnum{\renewcommand{\labelenumi}{(\textbf{\eplabel} \theenumi')}}

\newcommand{\midtext}[1]{\quad\text{#1}\quad}
\newcommand{\righttext}[1]{\quad\text{#1 }}

\newcommand{\yphi}{Y^\varphi}

\newcommand{\oy}{\OO_{\yphi}}

\newcommand{\talg}[1]{\TT(#1)}
\newcommand{\oalg}[1]{\OO(#1)}

\newcommand{\tsp}[1]{\talg #1}
\newcommand{\osp}[1]{\oalg #1}

\begin{document}
\title[Categories, Zappa-Sz{\' e}p products and Exel-Pardo algebras]{On finitely aligned left cancellative small categories, Zappa-Sz{\' e}p products and Exel-Pardo algebras}
\author{Erik B\'edos}
\address{Institute of Mathematics, University of Oslo, PB 1053 Blindern, 0316 Oslo, Norway}
\email{bedos@math.uio.no}
\author{S. Kaliszewski}
\address{School of Mathematical and Statistical Sciences, Arizona State University, Tempe, AZ 85287}
\email{kaliszewski@asu.edu}
\author{John Quigg}
\address{School of Mathematical and Statistical Sciences, Arizona State University, Tempe, AZ 85287}
\email{quigg@asu.edu}
\author{Jack Spielberg}
\address{School of Mathematical and Statistical Sciences, Arizona State University, Tempe, AZ 85287}
\email{spielberg@asu.edu}
\date{December 6, 2018}

\subjclass[2000]{46L05, 46L55}
\keywords{Groups, graphs, self-similarity, category of paths, left cancellative small categories, Zappa-Sz{\'e}p products, Toeplitz algebras, Cuntz-Krieger algebras}

\begin{abstract}
We consider Toeplitz and Cuntz-Krieger $C^*$-algebras associated with finitely aligned left cancellative small categories. We pay special attention to the case where such a category arises as the Zappa-Sz{\'e}p product of a category and a group linked by 
a one-cocycle. As our main application, we obtain a new approach to Exel-Pardo algebras in the case of row-finite graphs. We also present some other ways of constructing $C^*$-algebras from left cancellative small categories and discuss their relationship.
\end{abstract}

\maketitle

\section{Introduction}\label{intro}

By a category of paths we will mean a left cancellative small category with no (nontrivial) inverses. As shown in \cite{Spi11},  one may attach to each finitely aligned category of paths a Toeplitz
$C^*$-algebra and a Cuntz-Krieger $C^*$-algebra, which generalize previously known constructions for graphs, higher-rank graphs and quasi-lattice ordered groups.

The assumption that a category of paths has no inverses is quite restrictive and it is desirable to relax it. This has already been done in several  recent works on $C^*$-algebras associated to left cancellative monoids, see e.g.~\cite{LiSemigroup, norling, Li:nuclear, BRRW, starling15, BS:bdary, BLS:rightLCM, BPRRW, BLS:rightLCM2, AD:weakcont}.
Our primary motivation for  considering more general left cancellative small categories has been our desire to provide a new approach to $C^*$-algebras associated with Exel--Pardo systems \cite{EP, bkqexelpardo}. Our starting point is the observation that if $(E, G, \varphi)$ is an Exel--Pardo system, consisting of an action of a group $G$ on a directed graph $E$ equipped with a 1-cocycle $\varphi\: E^1\times G \to G$
satisfying certain compatibility conditions (see Section \ref{prelim}), and $E^*$ denotes the category of finite paths of $E$, then we may form the Zappa-Sz{\'e}p product $E^*\rtimes^\varphi G$ of the associated system $(E^*, G, \varphi)$. This gives us a left cancellative small category containing nontrivial inverses (unless if $G$ is trivial), and although it is possible to find a certain subcategory of $E^*\rtimes^\varphi G$ which is a category of paths (cf.~Proposition \ref{subpath}), we believe it is better to work with $E^*\rtimes^\varphi G$. In fact,
we may as well consider what we call a category system $(\CC, G, \varphi)$, where $\CC$ is a left cancellative small category, and form its Zappa-Sz{\'e}p product $\CC\rtimes^\varphi G$, as we do in Section  \ref{act cop}.
The most tractable situation is when  $\mathcal{C}$ is \emph{finitely aligned}. This notion
is defined in a similar way as for categories of paths \cite{Spi11}. 
An example of a finitely aligned (even singly aligned) left cancellative small category is provided by the Zappa-Sz{\'e}p product $E^*\rtimes^\varphi G$ arising from an Exel--Pardo system  $(E, G, \varphi)$. 
We also mention that singly aligned left cancellative monoids have often been called \emph{right LCMs} in the recent literature.

In Section \ref{C-alg} we consider the Toeplitz algebra $\mathcal{T}(\mathcal{C})$ and the Cuntz-Krieger algebra 
$\mathcal{O}(\mathcal{C})$ of a finitely aligned left cancellative small category $\CC$. They may both be described as universal $C^*$-algebras generated by families of partial isometries indexed by $\CC$ satisfying certain natural relations, $\oalg\CC$ being a quotient of $\talg\CC$. Equivalently,  $\talg\CC$ is the universal $C^*$-algebra for representations of $\CC$ in $C^*$-algebras, while $\oalg\CC$ is universal for covariant representations of $\CC$. A definition of  $\mathcal{T}(\mathcal{C})$ and $\mathcal{O}(\mathcal{C})$ as groupoid $C^*$-algebras, valid also in the case where
 $\CC$ is not finitely aligned, will be discussed in an article \cite{jackpaths2} by the fourth author.
It should be noted that a different definition has been proposed earlier by Exel in \cite{ExeSemigroupoid} (see also \cite{exelcomb}), where he associates a Cuntz-Krieger like  $C^*$-algebra to any semigroupoid. As a left cancellative small category $\CC$ is an example of a semigroupoid, we compare in Section \ref{C-alg} our approach with Exel's when $\CC$ is finitely aligned. In particular, we explain why our definition of a representation of $\CC$ in a $C^*$-algebra involves an extra condition and give an example showing that this condition does not follow from Exel's conditions.  After adapting Exel's notion of \emph{tightness} for representations of semigroupoids to our setting, we verify (using arguments from \cite{ExeSemigroupoid}) that this notion is equivalent to our notion of covariance for nondegenerate representations of $\CC$.

In Section \ref{ZS algebras} we consider a category system $(\CC, G, \varphi)$. By a representation $(T, U)$ of  $(\CC, G, \varphi)$ in a $C^*$-algebra $B$ we mean a nondegenerate representation $T\:\CC \to B$ and  a unitary homomorphism $U\: G\to M(B)$ satisfying  the condition 
 $U_gT_\alpha=T_{g\alpha}U_{\varphi(g,\alpha)}$ for all $g\in G, \alpha\in\CC. $
Our main result (Theorem \ref{integrated form}) says that if $\CC$ is finitely aligned, and $\DD$ denotes the Zappa-Sz{\'e}p product $\CC\rtimes^\varphi G$, then there is a natural bijective correspondence between representations $(T, U)$ of 
 $(\CC, G, \varphi)$ and nondegenerate representations $S=T\times U$ of $\DD$, having the property that $S$ is covariant if and only if $T$ is too (cf.~Corollary \ref{integrated CP}). This means that 
 $\talg\DD$ (resp. $\oalg\DD$) 
 can loosely be described as a kind of crossed product of $\talg\CC$ (resp.~$\oalg\CC$) by $(G, \varphi)$. This can be stated more precisely by using the concept of $C^*$-blend recently introduced in \cite{exelblend} (see Remarks \ref{TTDD-prop} and \ref{OODD-prop}).

When $(E, G, \varphi)$ is an Exel--Pardo system, we apply our results to the category system $(E^*, G, \varphi)$ and show in Section \ref{EP algebras} that the Toeplitz algebra $\mathcal{T}(E^*\rtimes^\varphi G) $ 
is  isomorphic to the Toeplitz $C^*$-algebra $\mathcal{T}(E, G, \varphi)$ associated to the $C^*$-correspondence $Y^\varphi$ constructed from $(E, G, \varphi)$ in \cite{bkqexelpardo} (see also \cite{EP}). Assuming that $E$ is row-finite, we also show that the Cuntz-Krieger algebra $ \mathcal{O}(E^*\rtimes^\varphi G)$ is isomorphic to the Cuntz-Pimsner $C^*$-algebra $\mathcal{O}(E, G, \varphi)$ associated to $Y^\varphi$.

It appears that one may associate other $C^*$-algebras to a 
left cancellative small category $\mathcal{C}$. Indeed, in Section \ref{final}, we introduce the \emph{regular Toeplitz algebra}
$\TT_\ell(\CC)$ generated by the regular representation of $\CC$ on $\ell^2(\CC)$,  the $C^*$-algebra $C^*({\rm ZM}(\mathcal{C}))$ associated to the so-called \emph{zigzag inverse semigroup} ${\rm ZM}(\mathcal{C})$ of $\mathcal{C}$, and the \emph{Li $C^*$-algebra} $C_{\rm Li}^*(\mathcal{C})$, analogous to the full $C^*$-algebra associated by Li \cite{LiSemigroup} to a left cancellative monoid. These $C^*$-algebras are connected by canonical surjective homomorphisms 
\[ C^*_{\rm Li}(\mathcal{C}) \to C^*({\rm ZM}(\mathcal{C})) \to  \TT_\ell(\CC).\]
 When $\CC$ is a finitely aligned,  there  also exist canonical surjective homomorphisms such that  
\[ C^*({\rm ZM}(\mathcal{C})) \to\mathcal{T}(\mathcal{C})\to \TT_\ell(\CC).\]  
Adapting some results of Donsig and Milan in \cite{DonMil} about categories of paths to our situation, we also note that one may identify $\TT(\CC)$ with the $C^*$-algebra 
that is universal for \emph{finitely join-preserving representations} of  ${\rm ZM}(\mathcal{C})$. Moreover, when $\CC$ is finitely aligned, $\mathcal{O}(\mathcal{C})$ corresponds to the \emph{tight $C^*$-algebra} of ${\rm ZM}(\mathcal{C})$ introduced by Exel \cite{exelcomb}, which is universal for \emph{tight representations} of ${\rm ZM}(\mathcal{C})$.

\section{Preliminaries}\label{prelim}

We recall a few definitions and conventions from \cite{Spi11} and \cite{bkqexelpardo}. 
See also \cite{jackpaths2}.

\subsection*{Small categories}
If $\CC$ is a small category, we refer to the objects in $\CC$ as \emph{vertices}, and we write $\CC^0$ for the set of vertices.
We also
use juxtaposition to indicate composition of morphisms,
and frequently (but not always) identify the vertices with the identity morphisms.
Thus, a \emph{small category} may be defined as a set $\CC$,
a subset $\CC^0\subset \CC$,
two maps $r,s\:\CC\to\CC^0$,
and a partially-defined multiplication
\[
(\alpha,\beta)\mapsto \alpha\beta,
\]
defined if and only if $s(\alpha)=r(\beta)$,
such that
if $s(\alpha)=r(\beta)$ and $s(\beta)=r(\gamma)$ then
\begin{enumerate}
\item
$r(\alpha\beta)=r(\alpha)$
and
$s(\alpha\beta)=s(\beta)$;

\item
$\alpha(\beta\gamma)=(\alpha\beta)\gamma$;

\item
$r(v)=s(v)=v$ for all $v\in \CC^0$;

\item $r(\alpha)\alpha=\alpha s(\alpha)=\alpha$.
\end{enumerate}
A \emph{subcategory} of  $\CC$ is a subset $\EE \subset \CC$ that is closed under $r$, $s$, and composition. It becomes a small category in the obvious way by setting $\EE^0=\CC^0 \cap \EE$.
 
We assume that $\CC^0$ is nonempty and sometimes write
\[
\CC^{(2)}=\{(\alpha,\beta)\in \CC\times \CC:s(\alpha)=r(\beta)\}
\]
for the set of composable pairs.
Note that
$\CC^{(2)} = \CC\times \CC$
if $\CC^0$ is a singleton (in which case $\CC$ is a monoid).

If $S,T\subset\CC$, we write
\[
ST=\{\alpha\beta:\alpha\in S,\beta\in T,s(\alpha)=r(\beta)\},
\]
and similarly for finitely many subsets $S_1,\dots,S_n$.
If $S=\{\alpha\}$ is a singleton we write $\alpha T=\{\alpha\}T$,
and similarly for $S\alpha$.
In particular, for  $v,w$ in $\CC^0$,
we have
$r\inv(v)=v\CC$,
$s\inv(w)=\CC w$,
and
$r\inv(v)\cap s\inv(w)=v\CC w$. 
Note that an element $\alpha\in \CC$ is a vertex if and only if
$\alpha\beta=\beta$ for all $\beta\in s(\alpha) \CC$ and
$\beta\alpha=\beta$ for all $\beta\in \CC r(\alpha)$.

We will often (but not always)
use the convention that when writing $\alpha\beta$ for $\alpha, \beta\in \CC$, we tacitly assume that  the product $\alpha\beta$ is defined.

If $\alpha \in \CC$
then $\alpha$ is called \emph{invertible} if there exists (a necessarily unique) $\beta \in \CC$ such that $\alpha\beta $ and $\beta\alpha$ both belong to $\CC^0$, in which case we often write $\beta = \alpha^{-1}$. Clearly, any vertex in $\CC$ is invertible. We say that the small category $\CC$ has \emph{no inverses} if the set of invertible elements in $\CC$ coincides with $\CC^0$.

We say a small category $\CC$ has \emph{left cancellation}, or is \emph{left cancellative}, if for any $\alpha,\beta,\gamma\in \CC$ such that $s(\alpha) = r(\beta) = r(\gamma)$ we have
\[
\alpha\beta=\alpha\gamma\midtext{implies}\beta=\gamma.
\]
Right cancellation is defined in a similar way.
Note that if $\CC$ is left cancellative then $\alpha\in \CC$ is a vertex if and only if
there exists $\beta\in \CC r(\alpha)$ such that $\beta\alpha=\beta$.
Note also that when $\CC$ is left cancellative, 
$\alpha \in \CC$ is invertible if and only if there exists (a necessarily unique) $\beta \in \CC$ such that $\alpha\beta $ belongs to $\CC^0$; in particular, 
$\CC$ has no inverses if and only if
\[
\alpha\beta\in\CC^0\midtext{implies}\alpha,\beta\in\CC^0.
\]
A \emph{category of paths} is a left cancellative small category $\CC$ with no inverses.

\begin{rem}
The definition of category of paths in \cite{Spi11} also requires right cancellation,
but this property is 
only used in a few places in \cite{Spi11}, which will not affect anything we do here.
\end{rem}

Throughout the remainder of this section, $\CC$ denotes a left cancellative small category.

We
define an equivalence relation on $\CC$ by $\alpha\sim \beta$ when there is an invertible $\gamma\in\CC$ such that $\beta=\alpha\gamma$.

We say $\CC$ is \emph{finitely aligned} if for all $\alpha,\beta\in\CC$ there is a finite 
(possibly empty)
subset $F\subset \CC$  
such that
\[
\alpha\CC\cap \beta\CC=F\CC=\bigcup_{\gamma\in F}\gamma\CC,
\]
and \emph{singly aligned} if 
for all $\alpha,\beta\in\CC$, either $\alpha\CC\cap \beta\CC = \varnothing$ or there is $\gamma\in\CC$ such that
\[
\alpha\CC\cap \beta\CC=\gamma\CC.
\]

If $\CC$ is singly aligned and we have $\alpha\CC\cap \beta\CC=\gamma\CC$ as above, how unique is the $\gamma$?
Suppose
\[
\gamma\CC=\zeta\CC,
\]
so that there are $\lambda,\mu\in\CC$ such that
\[
\gamma=\zeta\lambda\midtext{and}\zeta=\gamma\mu.
\]
Then
$\gamma=\gamma\mu\lambda$,
so by left cancellativity $\mu\lambda=s(\gamma)$,
and similarly $\lambda\mu=s(\zeta)$.
Thus $\mu=\lambda\inv$.
In particular, we have $\gamma\sim\zeta$,
and so $\gamma$ is unique up to equivalence.
Note that this argument generalizes \cite[Lemma~2.2]{BRRW} to our context. Note also that if $\gamma\sim\zeta$, then one readily checks that $\gamma\CC=\zeta\CC$; hence we have
$\gamma\sim\zeta$ if and only if  $\gamma\CC=\zeta\CC$.

The finitely aligned case is only slightly more complicated:
suppose
\begin{equation}\label{FL}
\bigcup_{\gamma\in F}\gamma\CC=\bigcup_{\zeta\in L}\zeta\CC
\end{equation}
for some nonempty subsets $F, L \subseteq \CC$.
Then for all $\gamma\in F$ there are $\zeta\in L,\lambda\in\CC$ such that $\gamma=\zeta\lambda$,
and then there are $\gamma'\in F,\lambda'\in\CC$ such that $\zeta=\gamma'\lambda'$,
so that
\[
\gamma=\gamma'\lambda'\lambda\in \gamma'\CC.
\]
Now, if $F$ is finite,
we can replace it by a subset if necessary so that
for all distinct $\gamma,\gamma'\in F$ we have 
$\gamma\notin \gamma'\CC$,
in which case we say $F$ is 
\emph{independent}.
Similarly we can assume that $L$ is independent.
Then we conclude that
if $F$ and $L$ 
are both
independent and satisfy \eqref{FL}, then
for all $\gamma\in F$ there is $\zeta\in L$ such that $\gamma\sim \zeta$,
and symmetrically for all $\zeta\in L$ there is $\gamma\in F$ such that $\zeta\sim \gamma$.
When this happens we say that $F$ is \emph{unique up to equivalence}.  We will also consider the empty set as independent and unique up to equivalence.

\begin{rem}
\cite[Lemma~3.2]{Spi11} shows that when $\CC$ is a finitely aligned category of paths
(so that $\CC$ not only is left cancellative, but also has no inverses),
there is a unique independent subset $F$ as above.
But as we have seen, 
in the general left cancellative case $F$ is only unique up to equivalence.
\end{rem}

When $\CC$ is finitely aligned, by induction we see that for every finite subset $F\subset \CC$ there is a finite independent set $L$,
that is unique up to equivalence,
such that
\[
\bigcap_{\alpha\in F}\alpha\CC=\bigcup_{\gamma\in L}\gamma\CC,
\]
and we will write $\bigvee F$ for any such choice of independent finite $L$,
keeping in mind that this is only determined up to equivalence.
Thus, if $L\neq \varnothing$ and for every $\gamma\in L$ we choose $\gamma'\sim \gamma$,
and let $L'=\{\gamma':\gamma\in L\}$,
then we also have $\bigvee F=L'$.
When $F=\{\alpha,\beta\}$ we write
$\alpha\vee \beta=L$.
Note that this convention is slightly different from that of \cite{jackpaths2}, where $\bigvee F$ denotes the set of all elements that are equivalent to an element of the set we denote by $\bigvee F$.

If 
$v\in \CC^0$
and 
$F\subset v\CC$,
then $F$
is \emph{exhaustive} at $v$ if
for every $\alpha\in v\CC$ there is $\beta\in F$ with $\alpha\CC\cap\beta\CC\ne\varnothing$.

Later we will need the following elementary result:

\begin{lem}\label{exhaust inv}
Let $\CC$ be a left cancellative small category, let $v\in \CC^0$, and let $F\subset v\CC$ be nonempty.
For each $\alpha\in F$ let $\beta_\alpha\sim \alpha$,
and let
\[
F'=\{\beta_\alpha:\alpha\in F\}.
\]
Then $F$ is exhaustive at $v$ if and only if $F'$ is.
\end{lem}

\begin{proof}
By symmetry, it suffices to show that if $F$ is exhaustive at $v$ then so is $F'$.
Let $\beta\in v\CC$.
Since $F$ is exhaustive at $v$,
we can choose $\alpha\in F$ such that $\alpha \CC\cap \beta \CC\ne \varnothing$.
Since $\beta_\alpha\sim \alpha$, we have $\beta_\alpha \CC=\alpha \CC$.
Thus $\beta_\alpha \CC\cap \beta \CC\ne \varnothing$.
This shows that $F'$ is exhaustive at $v$.
\end{proof}

\subsection*{Cocycles}

Let $G$ denote a discrete group (with identity $1$) and $S$ a set.
We write $G\curvearrowright S$ to mean that $G$ acts on $S$ by permutations,
and we write the action as
\[
(g,x)\mapsto gx\:G\times S\to S.
\]
A \emph{cocycle} for an action $G\curvearrowright S$ is a function $\varphi\:G\times S\to G$
satisfying the \emph{cocycle identity}
\begin{equation}\label{cocy-id}
\varphi(gh,x)=\varphi(g,hx)\varphi(h,x)\righttext{for all}g,h\in G,x\in S.
\end{equation}
Plugging in $h=1$ we get $\varphi(g,x)1= \varphi(g,x)\varphi(1,x)$, so  
\begin{equation}\label{cocy-norm}
\varphi(1,x) = 1 \righttext{for all}x\in S.\footnote{Note that this only requires the left cancellative property of $G$; cf.~Remark~\ref{monoid}.}
\end{equation}
One may also consider more general cocycles, taking their values in another group than $G$, but we won't need these in the present work. The cocycle identity is exactly what is needed for $G$ to act on $S\times G$ via
\[
g(x,h)=\bigl(gx,\varphi(g,x)h\bigr)\righttext{for}g,h\in G,x\in S.
\]

\subsection*{Graph cocycles}
We say $G$ acts on a directed graph $E=(E^0,E^1,r,s)$, written $G\curvearrowright E$,
if $G$ acts on the vertex set $E^0$ and the edge set $E^1$
by graph automorphisms, i.e., $G\curvearrowright E^0$, $G\curvearrowright E^1$ and
\[
r(ge)=gr(e)\midtext{and}s(ge)=gs(e)\righttext{for all}g\in G,e\in E^1.
\]

\begin{defn}
A \emph{graph cocycle} for $G\curvearrowright E$ is
a cocycle $\varphi$ for the action of $G$ on the edge set $E^1$ such that
\begin{equation}\label{source-inv}
\varphi(g,e)s(e)=gs(e)\righttext{for all}g\in G,e\in E^1,
\end{equation}
and we call $(E,G,\varphi)$ an \emph{Exel-Pardo system}.
\end{defn}

Note that in \cite{EP} Exel and Pardo impose the stronger condition \[\varphi(g,e)v=gv\righttext{for all} g\in G, e\in E^1, v\in E^0.\]
As we tried to make clear in \cite{bkqexelpardo}, our weaker condition on sources allows for greater flexibility. We note that in \cite{LRRW2}, where self-similar actions of groupoids on the path spaces of finite directed graphs are considered, this equivariance property of the source map is not necessarily satisfied (or even meaningful).

\section{$C^*$-algebras}\label{C-alg}

Let $\CC$ be a finitely aligned left cancellative small category.
There are various ways to associate $C^*$-algebras to $\CC$. 
We follow here the approach
developed by the fourth author for categories of paths in \cite{Spi11},
then generalized to left cancellative small categories in \cite{jackpaths2},
and 
consider the Toeplitz algebra $\talg \CC$ and the Cuntz-Krieger algebra $\oalg \CC$. 
We then compare these two $C^*$-algebras with those we get by using Exel's approach for semigroupoids in \cite{ExeSemigroupoid}.
In section \ref{final}, we discuss other constructions related to the work of Li \cite{LiSemigroup} on semigroup $C^*$-algebras and the work of Donsig-Milan \cite{DonMil} on inverse semigroups and categories of paths.

We will approach the $C^*$-algebras via ``universal representations''.

\begin{defn}\label{spi rep}
A \emph{representation} of a finitely aligned left cancellative small category $\CC$ in a $C^*$-algebra $B$ is a mapping $T\:\CC\to B$
satisfying the axioms in \cite[Theorem~6.3]{Spi11}:
for all $\alpha,\beta\in\CC$,
\begin{enumerate}
\arabicnum
\item 
\label{spi 1}
$T_\alpha^* T_\alpha=T_{s(\alpha)}$;

\item 
\label{spi 2}
$T_\alpha T_\beta=T_{\alpha\beta}$ if $s(\alpha)=r(\beta)$;

\item 
\label{spi 3}
$T_\alpha T_\alpha^* T_\beta T_\beta^*=\bigvee_{\gamma\in \alpha\vee \beta}T_\gamma T_\gamma^*$,
\end{enumerate}
and
a representation $T$ is \emph{covariant} if it satisfies
one additional axiom:
\begin{enumerate}
\setcounter{enumi}{3}
\arabicnum
\item 
\label{spi 4}
$T_v=\bigvee_{\alpha\in F}T_\alpha T_\alpha^*$
for every $v\in \CC^0$ and every finite exhaustive set $F$ at $v$.
\end{enumerate}
When $T$ is a representation of $\CC$ in $B$, we will let $C^*(T)$ denote the $C^*$-subalgebra of $B$ generated by the range of $T$.
\end{defn}
\begin{rem}\label{remSpi} 
Note that (\ref{spi 1}) and (\ref{spi 2}) imply that $T_v$ is a projection in $B$ for every $v \in \CC^0$ and that  $T_\alpha$ is a partial isometry in $B$ for every $\alpha \in \CC$. Hence $T_\alpha T_\alpha^*$ is a projection in $B$ for every $\alpha \in \CC$. Concerning condition (\ref{spi 3}), the join $\bigvee_{\gamma\in \alpha\vee \beta}T_\gamma T_\gamma^*$ is a priori defined as a projection in $B^{**}$. 
The same comment applies also to the  join $\bigvee_{\alpha\in F}T_\alpha T_\alpha^*$ in condition (\ref{spi 4}). 
We also note that condition (\ref{spi 3}) is not ambiguous, as one immediately sees by using Lemma \ref{inverse}.
Finally, the reader should be aware  that by convention the join over an empty index set is defined to be zero. Thus condition (\ref{spi 3}) says in particular that if $v, w \in \CC^0$ and $v\neq w$, then $T_v T_w = 0$, i.e., the projections $T_v$ and $T_v$ are orthogonal to each other.
\end{rem}
\begin{rem}\label{remSpi2} It is also worth mentioning that condition (\ref{spi 3}) does not follow from conditions (\ref{spi 1}), (\ref{spi 2}) and (\ref{spi 4}).
We will illustrate this in Example \ref{not(3)}.
\end{rem}
\begin{lem}\label{inverse}
Let $\CC$ be a finitely aligned left cancellative small category,
and let 
$T\:\CC\to B$ satisfy conditions $($\ref{spi 1}$)$ and $($\ref{spi 2}$)$ in \defnref{spi rep}.
If $\gamma\in\CC$ is invertible then
$T_{\gamma\inv}=T_\gamma^*$ and
$T_\gamma T_\gamma^*=T_{r(\gamma)}$.
Moreover, if $\alpha \sim \alpha'$ in $\CC$, then $T_\alpha T_\alpha^*= T_{\alpha' }T_{\alpha'}^*$.
\end{lem}
\begin{proof}
If $\gamma\in\CC$ is invertible, then
\begin{align*}
T_{\gamma\inv}
&=T_{s(\gamma) \gamma\inv}
=T_{s(\gamma)}T_{\gamma\inv}
=T_\gamma^*T_\gamma T_{\gamma\inv}\\
&=T_\gamma^*T_{\gamma\gamma\inv}
=T_\gamma^*T_{r(\gamma)}
=\big(T_{r(\gamma)}T_\gamma\big)^*\\
&=T_\gamma^* & & 
\end{align*}
and we get
\[
T_\gamma T_\gamma^*=T_\gamma T_{\gamma\inv}=T_{\gamma\gamma\inv}=T_{r(\gamma)}.
\]
Moreover, if $\alpha \in \CC$ and
$\gamma\in s(\alpha)\CC$ is invertible, then
\begin{align*}
T_{\alpha\gamma}T_{\alpha\gamma}^*
&=T_\alpha T_\gamma T_\gamma^* T_\alpha^*
\\&=T_\alpha T_{r(\gamma)} T_\alpha^*
\\&=T_\alpha T_{s(\alpha)} T_\alpha^*
\\&=T_\alpha T_\alpha^*,
\end{align*}
and  the second assertion follows.
\end{proof}
The next lemma will be useful later.
\begin{lem}\label{cp preserved}
Let $\CC$ be a finitely aligned left cancellative small category,
and let $T$ be a representation of $\CC$.
Let $v\in \CC^0$, and let $F$ be a finite subset of $v\CC$.
For each $\alpha\in F$ let $\beta_\alpha\sim \alpha$,
and put
\[
F'=\{\beta_\alpha:\alpha\in F\}.
\]
Then
\[
\bigvee_{\alpha\in F}T_\alpha T_\alpha^*=\bigvee_{\beta\in F'}T_\beta T_\beta^*.
\]
\end{lem}

\begin{proof}
This is an obvious consequence of \lemref{inverse}.
\end{proof}

\begin{defn}\label{spi toep}
A \emph{Toeplitz algebra} of a finitely aligned left cancellative small category $\CC$ is a pair
$(\talg \CC,t)$,
where $\talg \CC$ is a $C^*$-algebra and
$t$ is a representation of $\CC$ in $\talg \CC$
having the universal property that
for every representation $T$ of $\CC$ in a $C^*$-algebra $B$
there is a unique homomorphism $\phi_T\:\talg \CC\to B$ such that
\[
T_\alpha=\phi_T(t_\alpha)\righttext{for all}\alpha\in \CC.
\]
\end{defn}

One readily checks that
$(\talg \CC,t)$ is unique up to isomorphism in the sense that
if $(B,T)$ is any Toeplitz algebra of $\CC$ then
$\phi_T\:\talg \CC\to B$ is an isomorphism.
Thus we commit the usual abuse of referring to ``the'' Toeplitz algebra,
and also we note that $\talg \CC=C^*(t)$.

\begin{defn}\label{spi cp}
A \emph{Cuntz-Krieger algebra} of a finitely aligned left cancellative small category $\CC$ is a pair
$(\oalg \CC,\tilde t\,)$,
where $\oalg \CC$ is a $C^*$-algebra and
$\tilde t$ is a covariant representation of $\CC$ in $\oalg \CC$
having the universal property that
for every covariant representation $T$ of $\CC$ in a $C^*$-algebra $B$
there is a unique homomorphism $\psi_T\:\oalg \CC\to B$ such that
\[
T_\alpha=\psi_T({\tilde t}_\alpha)\righttext{for all}\alpha\in \CC.
\]
\end{defn}
Similarly to the Toeplitz case,
$(\oalg \CC,\tilde t\,)$ is unique up to isomorphism,
and $\oalg \CC=C^*(\,\tilde t\,)$. Note that since $\tilde t$ is a representation of $\CC$ in $\OO(\CC)$, the associated homomorphism 
$\phi_{\,\tilde t}\: \TT(\CC)\to \OO(\CC)$ satisfies 
${\tilde t}_\alpha=\phi_{\,\tilde t}(t_\alpha)$ for all $\alpha\in \CC.$ It follows that $\phi_{\,\tilde t}$ is surjective and we have 
$\phi_T = \psi_T\circ \phi_{\,\tilde t}$ for every covariant representation $T$ of $\CC$.

\begin{rem} \label{rem groupoid approach}
In \cite{jackpaths2}
the approach to $\tsp \CC$ and $\osp \CC$ is via certain groupoids $G$ and $G|_{\partial G}$, respectively.
When $\CC$ is finitely aligned 
\cite[Theorems~9.8 and 10.15]{jackpaths2}
show that $\tsp \CC$ is characterized by the universal properties (\ref{spi 1})--(\ref{spi 3}),
and that $\osp \CC$ is characterized by the universal properties (\ref{spi 1})--(\ref{spi 4}),
in Definition~\ref{spi rep}.
 (Moreover, it is shown in \cite{jackpaths2} that the hypothesis of amenability in \cite[Theorem~8.2]{Spi11} is unnecessary.)
A similar groupoid approach may  be followed for any left cancellative small category, cf.~\cite{jackpaths2}.  We will sketch an alternative approach 
in Remark \ref{existenceTT}.
\end{rem}

Exel works in somewhat greater generality,
namely he starts with a \emph{semigroupoid},
which is not quite a small category because it is not assumed to have identity morphisms.
However, in \cite[Section~7]{ExeSemigroupoid} he considers the special case of small categories.
We will always have a small category, so we will interpret Exel's definitions and results in that context.
Exel studies a version of $\OO(\CC)$,
and in particular when $\CC$ is a row-finite higher-rank graph $\Lambda$ with no sources
Exel recovers the familiar higher-rank-graph algebra $C^*(\Lambda)$.
On the other hand,
Exel does not actually investigate a version of $\TT(\CC)$, but he does at least hint at its definition in the paragraph following \cite[Proposition~4.7]{ExeSemigroupoid}.
Exel defines what he calls representations of $\CC$,
and his ``tight representations'' satisfy an extra property that 
we will recall in Definition \ref{exel-tight}.
In the following theorem,
\exe{exe 1}--\exe{exe 5} constitute Exel's definition of a representation of $\CC$
\cite[Definition~4.1]{ExeSemigroupoid}\footnote{\,where in \exe{exe 2} and \exe{exe 5} we have taken into account that $\CC$ is a small category and not just a semigroupoid};
after the proof we will explain our motivation for adding
the last property \exe{exe 6}.
We should point out that
our numbering does not quite match Exel's because he lists the requirement \exe{exe 3} without a number.

\begin{thm}\label{spiexe}
Let $\CC$ be a finitely aligned left cancellative small category,
let $B$ be a $C^*$-algebra,
and let $T\:\CC\to B$.
Then $T$ is a representation of $\CC$ in the sense of \defnref{spi rep} if and only if
it satisfies the following conditions:
\begin{enumerate}
\exelist
\item \label{exe 1}
$T_\alpha$ is a partial isometry for every $\alpha\in\CC$;

\item \label{exe 2}
for all $\alpha,\beta\in\CC$,
$T_\alpha T_\beta=
\begin{cases}
T_{\alpha\beta}
\quad \text{if} \ s(\alpha)=r(\beta),
\\
\ 0
\quad \ \ \text{otherwise};
\end{cases}$
\\

\item \label{exe 3}
the family of initial projections $\{T_\alpha^* T_\alpha:\alpha\in\CC\}$ commutes,
as does the family of final projections $\{T_\alpha T_\alpha^*:\alpha\in\CC\}$;

\item \label{exe 4}
$T_\alpha T_\alpha^*T_\beta T_\beta^*=0$ if $\alpha \CC\cap \beta \CC=\varnothing$;

\item \label{exe 5}
$T_\alpha^* T_\alpha\ge T_\beta T_\beta^*$
if $s(\alpha)=r(\beta)$;
\item \label{exe 6}
$T_\alpha T_\alpha^* T_\beta T_\beta^*=\bigvee_{\gamma\in \alpha\vee \beta}T_\gamma T_\gamma^*$
for all $\alpha,\beta\in \CC$.
\end{enumerate}
\end{thm}

\begin{proof}
First assume that $T$ is a representation of $\CC$ in $B$. 
Then \exe{exe 1} is satisfied, as we already pointed out in Remark \ref{remSpi}.  
\exe{exe 2} is satisfied because if $s(\alpha)=r(\beta)$ then $T_\alpha T_\beta=T_{\alpha\beta}$  according to condition (\ref{spi 2}) in Definition \ref{spi rep}, while if $s(\alpha)\ne r(\beta)$ then $T_{s(\alpha)} T_{r(\beta)}=0$,  cf.~Remark \ref{remSpi}, so we get
\[ 
(T_\alpha T_\beta)^*T_\alpha T_\beta = T_\beta^* T_{s(\alpha)}T_\beta  = T_\beta^* T_{s(\alpha)}T_{r(\beta)}T_\beta= 0,
\]
hence $T_\alpha T_\beta =0$.       
For \exe{exe 3},
note that for $\alpha, \beta \in \CC$, we have
 \[
 T_\alpha^* T_\alpha T_\beta^*T_\beta=T_{s(\alpha)}T_{s(\beta)}
 = \begin{cases}
T_{s(\alpha)}
\quad \text{if} \ s(\alpha)=s(\beta),
\\
\ 0
\quad \quad \ \text{otherwise},
\end{cases}
\]
from which it readily follows that 
the initial projections $T_\alpha^* T_\alpha$
commute. 
On the other hand, condition (\ref{spi 3}) in Definition \ref{spi rep} clearly implies that
the range projections $T_\alpha T_\alpha^*$ commute.
\exe{exe 4}
follows from condition (\ref{spi 3}) in Definition \ref{spi rep},
because if $\alpha\CC\cap\beta\CC=\varnothing$ then $\alpha\vee \beta=\varnothing$.
To see that \exe{exe 5}
is satisfied we observe that if $s(\alpha)=r(\beta)$ then 
\[
T_\alpha^* T_\alpha T_\beta T_\beta^* = T_{s(\alpha)}T_\beta T_\beta^*= T_{r(\beta)}T_\beta T_\beta^* = T_\beta T_\beta^*.
\] 
Finally, \exe{exe 6} is exactly condition (\ref{spi 3}) in \defnref{spi rep}.

Conversely, assume that $T$ satisfies the conditions \exe{exe 1}--\exe{exe 6}.
In \cite[Proposition~7.1]{ExeSemigroupoid} Exel points out
some consequences of \exe{exe 1}--\exe{exe 5}
in the case that $\CC$ is a small category. One of these is precisely 
condition (\ref{spi 1}) in Definition \ref{spi rep}. Indeed, after observing that $T_v$ is a projection for each $v\in \CC^0$, he uses \exe{exe 2} and \exe{exe 5} to obtain that
\[
T_\alpha^*T_\alpha = T_\alpha^*T_{\alpha s(\alpha)}= T_\alpha^*T_{\alpha} T_{s(\alpha)} = T_{s(\alpha)}
\]  
for each $\alpha \in \CC$. Next, it is clear that condition (\ref{spi 2}) in Definition \ref{spi rep} follows from \exe{exe 2}.
Finally, condition (\ref{spi 3}) in Definition \ref{spi rep} is exactly \exe{exe 6}.
\end{proof}

\begin{ex}\label{counterexample}
As we mentioned before the statement of \thmref{spiexe}, the condition \exe{exe 6} was not included in Exel's \cite[Definition~4.1]{ExeSemigroupoid}.
Here we show that this sixth condition is necessary.  We consider the left cancellative small category $\CC$ (actually a 2-graph) with the following graph, and identification $\alpha \gamma = \beta \delta$:
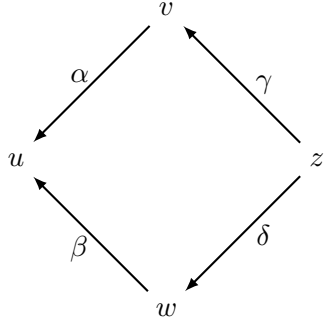
\[
\begin{tikzpicture}[scale=2]

\node (0_0) at (0,0) [circle] {$u$};
\node (2_0) at (2,0) [circle] {$z$};
\node (1_1) at (1,1) [circle] {$v$};
\node (1_m1) at (1,-1) [circle] {$w$};

\draw[-latex,thick] (1_1) -- (0_0) node[pos=0.5, inner sep=0.5pt, anchor=south east] {$\alpha$};
\draw[-latex,thick] (1_m1) -- (0_0) node[pos=0.5, inner sep=0.5pt, anchor=north east] {$\beta$};

\draw[-latex,thick] (2_0) -- (1_1) node[pos=0.4, inner sep=0.5pt, anchor=south west] {$\gamma$};

\draw[-latex,thick] (2_0) -- (1_m1) node[pos=0.4, inner sep=0.5pt, anchor=north west] {$\delta$};

\end{tikzpicture}
\]
We will give a representation of $\CC$ satisfying \exe{exe 1} -- \exe{exe 5} but not \exe{exe 6}.  Let $K$ be a fixed Hilbert space.  ($K$ may be taken to have dimension one; we hope that our notation will make the construction easier to understand.)  Let
\[
H_u = \bigoplus_{i=1}^4 H_{u,i}, \quad
H_v = \bigoplus_{i=1}^3 H_{v,i}, \quad
H_w = \bigoplus_{i=1}^3 H_{w,i},
\]
where $H_{u,i}$, $H_{v,i}$, $H_{w,i}$, and $H_z$ are isomorphic to $K$.  We define a representation $T \: \CC \to B(H)$, where $H = H_u \oplus H_v \oplus H_w \oplus H_z$.  For this we need only define partial isometries $T_\alpha$, $T_\beta$, $T_\gamma$, $T_\delta$ with the above properties.  For a subspace $M \subseteq H$ let $P_M$ denote the projection of $H$ onto $M$.
\begin{align*}
T_\alpha^* T_\alpha &= P_{H_v},\ T_\alpha T_\alpha^* = P_{H_{u,1} \oplus H_{u,2} \oplus H_{u,3}} \\
T_\alpha(H_{v,1}) &= H_{u,1},\ T_\alpha(H_{v,2}) = H_{u,2},\ T_\alpha(H_{v,3}) = H_{u,3} \\
T_\gamma^* T_\gamma &= P_{H_z},\ T_\gamma T_\gamma^* = P_{H_{v,1}} \\
T_\delta^* T_\delta &= P_{H_z},\ T_\delta T_\delta^* = P_{H_{w,1}} \\
T_\beta^* T_\beta &= P_{H_w},\ T_\beta T_\beta^* = P_{H_{u,1} \oplus H_{u,2} \oplus H_{u,4}} \\
T_\beta|_{H_{w,1}} &= T_\alpha T_\gamma T_\delta^*|_{H_{w,1}} \\
T_\beta(H_{w,2}) &= H_{u,2},\ T_\beta(H_{w,3}) = H_{u,4}.
\end{align*}
It is straightforward to verify \exe{exe 1} -- \exe{exe 5}.  However,
\begin{align*}
T_\alpha T_\alpha^* H_u &= H_{u,1} \oplus H_{u,2} \oplus H_{u,3} \\
T_\beta T_\beta^* H_u &= H_{u,1} \oplus H_{u,2} \oplus H_{u,4} \\
T_\alpha T_\alpha^* H_u \cap T_\beta T_\beta^* H_u &= H_{u,1} \oplus H_{u,2} \\
T_{\alpha \vee \beta} T_{\alpha \vee \beta}^* H_u &= T_{\alpha \gamma} T_{\alpha \gamma}^* H_u 
= H_{u,1}.
\end{align*}
Therefore \exe{exe 6} does not hold for this representation.

We comment on the motivation for choosing to use \exe{exe 1} -- \exe{exe 6}, instead of just \exe{exe 1} -- \exe{exe 5} as Exel does.  First we give an ad hoc reason.  In the case of a higher rank graph (e.g.\ in the above example), the generally accepted definition of the Toeplitz $C^*$-algebra requires that the last relation in the definition of the higher rank graph $C^*$-algebra (\cite[Definition 2.5(iv)]{raesimsyee}) be relaxed from equality to a weak inequality (see \cite[Definition 4.1(iv)]{raesims}).  As mentioned in \cite[Theorem 5.11]{Spi11}, it is \exe{exe 6} (i.e., condition (\ref{spi 3}) in \defnref{spi rep}) that corresponds to this ``Toeplitz-Cuntz-Krieger'' relation.  Moreover, the results of \cite[Section 5]{Spi11} show that condition \exe{exe 6} is necessary in order that representations of $\CC$ reflect the basic Boolean ring structure corresponding to composition in $\CC$.
\end{ex}

\begin{ex}\label{not(3)} One might ask whether a representation satisfying \exe{exe 1} -- \exe{exe 5} and which is tight, as in Definition \ref{exel-tight}, will automatically also satisfy \exe{exe 6}.  In fact, Exel shows in \cite[Theorem 8.7]{ExeSemigroupoid} that if $\CC$ is a row-finite higher-rank graph with no sources, then this is the case.  However even for $\CC$ a finitely aligned higher-rank graph, \exe{exe 6} does not follow from the other relations (and tightness).  This can be seen using \cite[Example A.3]{raesimsyee}, which we reproduce in a flattened version here:
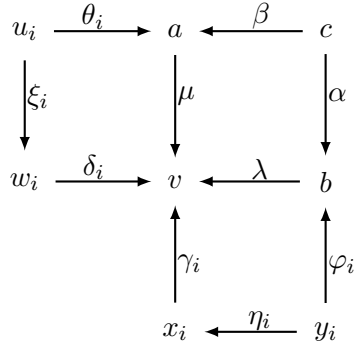
\[
\begin{tikzpicture}[scale=2]

\node (1_0) at (1,0) [circle] {$x_i$};
\node (2_0) at (2,0) [circle] {$y_i$};
\node (0_1) at (0,1) [circle] {$w_i$};
\node (1_1) at (1,1) [circle] {$v$};
\node (2_1) at (2,1) [circle] {$b$};
\node (0_2) at (0,2) [circle] {$u_i$};
\node (1_2) at (1,2) [circle] {$a$};
\node (2_2) at (2,2) [circle] {$c$};

\draw[-latex,thick] (2_0) -- (1_0) node[pos=0.4, inner sep=0.5pt, anchor=south] {$\eta_i$};
\draw[-latex,thick] (2_1) -- (1_1) node[pos=0.4, inner sep=0.5pt, anchor=south] {$\lambda$};
\draw[-latex,thick] (2_2) -- (1_2) node[pos=0.4, inner sep=0.5pt, anchor=south] {$\beta$};
\draw[-latex,thick] (0_1) -- (1_1) node[pos=0.4, inner sep=0.5pt, anchor=south] {$\delta_i$};
\draw[-latex,thick] (0_2) -- (1_2) node[pos=0.4, inner sep=0.5pt, anchor=south] {$\theta_i$};

\draw[-latex,thick] (0_2) -- (0_1) node[pos=0.4, inner sep=0.5pt, anchor=west] {$\xi_i$};
\draw[-latex,thick] (1_2) -- (1_1) node[pos=0.4, inner sep=0.5pt, anchor=west] {$\mu$};
\draw[-latex,thick] (1_0) -- (1_1) node[pos=0.4, inner sep=0.5pt, anchor=west] {$\gamma_i$};
\draw[-latex,thick] (2_0) -- (2_1) node[pos=0.4, inner sep=0.5pt, anchor=west] {$\varphi_i$};
\draw[-latex,thick] (2_2) -- (2_1) node[pos=0.4, inner sep=0.5pt, anchor=west] {$\alpha$};

\end{tikzpicture}
\]
The index $i$ varies through the positive integers, all vertices $\{u_i,w_i,x_i,y_i : i \ge 1\} \cup \{v,a,b,c\}$ are distinct, and the identifications are $\lambda \alpha = \mu \beta$, $\lambda \varphi_i = \gamma_i \eta_i$ for $i \ge 1$, and $\mu \theta_i = \delta_i \xi_i$ for $i \ge 1$.  (The 2-graph structure is obtained by letting the horizontal edges have degree $(1,0)$, and letting the vertical edges have degree $(0,1)$.)  The key point is that every nontrivial finite exhaustive set at $v$ (i.e., which does not contain $v$ itself) must contain both $\lambda$ and $\mu$, while at $a$ and at $b$ there are no nontrivial finite exhaustive sets.  Then it is possible that the projections $T_a$ and $T_b$ are strictly larger than the (strong operator) sum of the range projections of the partial isometries corresponding to edges with range at $a$ or $b$.  In this case, it is possible for $T_\lambda T_\lambda^*$ and $T_\mu T_\mu^*$ to both dominate a common image of the differences.  We give an explicit example of such a representation $T$.  We note that $\lambda \vee \mu = \lambda \alpha = \mu \beta$.

Let $K$ be a fixed Hilbert space.  We define Hilbert spaces at the vertices of $\CC$.
\begin{align*}
H_v &= \bigoplus_{i \in \Z} H_{v,i} \oplus H_v' \\
H_b &= \bigoplus_{i \ge 0} H_{b,i} \oplus H_b' \\
H_a &= \bigoplus_{i \ge 0} H_{a,i} \oplus H_a',
\end{align*}
where all the Hilbert spaces on the right hand sides are isomorphic to $K$.  Moreover, let $H_c$, $H_{u_i}$, $H_{w_i}$, $H_{x_i}$, $H_{y_i}$ be Hilbert spaces isomorphic to $K$.  Next we define the representation on edges of $\CC$.  Since the initial space must equal the Hilbert space at the source, it is enough to describe the final space; we need specify the partial isometry explicitly only where it is necessary to obey the commutation relations.
\begin{align*}
T_\lambda(H_{b,i}) &= H_{v,i},\ i \ge 0, \\
T_\lambda(H_b') &= H_v', \\
T_\mu(H_{a,i}) &= H_{v,-i},\ i \ge 0, \\
T_\mu(H_a') &= H_v', \\
T_\alpha(H_c) &= H_{b,0}, \\
T_\beta &= T_\mu^* 
T_\lambda T_\alpha,
\\
T_{\varphi_i}(H_{y_i}) &= H_{b,i},\ i \ge 1, \\
T_{\eta_i}(H_{y_i}) &= H_{x_i},\ i \ge 1, \\
T_{\gamma_i} &= T_\lambda T_{\varphi_i} T_{\eta_i}^*, \\
T_{\theta_i}(H_{u_i}) &= H_{a,i},\ i \ge 1, \\
T_{\xi_i} (H_{u_i}) &= H_{w_i},\ i \ge 1, \\
T_{\delta_i} &= T_\mu T_{\theta_i} T_{\xi_i}^*.
\end{align*}
It is straightforward to check that $T$ satisfies \exe{exe 1} -- \exe{exe 5} and is tight.  (We note that by \cite[theorem 7.4 (ii)]{ExeSemigroupoid}, 
cf.\  Theorem \ref{spiexe 2},
tightness follows from covariance, which is easy to check in this case.)  However $T_\lambda T_\lambda^* T_\mu T_\mu^*$ equals the projection onto $H_{v,0} \oplus H_v'$, while $T_{\lambda \vee \mu} T_{\lambda \vee \mu}^*$ equals the projection onto $H_{v,0}$.  Therefore $T$ does not satisfy \exe{exe 6}.

We also note that $T$ is an example of a representation of $\CC$ satisfying conditions (\ref{spi 1}), (\ref{spi 2}) and (\ref{spi 4}) in \defnref{spi rep}, but not (\ref{spi 3}), cf.\ Remark \ref{remSpi2}. Indeed, it is clear from the proof of Theorem \ref{spiexe} that \exe{exe 1} -- \exe{exe 5} imply that conditions (\ref{spi 1}) and (\ref{spi 2}) in \defnref{spi rep} hold. Moreover, $T$ is covariant, that is, it satisfies (\ref{spi 4}) in \defnref{spi rep}.
As $T$ does not satisfy \exe{exe 6}, it does not satisfy (\ref{spi 3}) in \defnref{spi rep}. 
\end{ex}

Our next aim is to prove a ``Cuntz-Krieger version'' of \thmref{spiexe}.
But we will first introduce a notion of nondegeneracy for representations of $\CC$.

\begin{defn}\label{nd}
Let $\CC$ be a finitely aligned left cancellative small category. A representation $T$ of $\CC$ in a $C^*$-algebra $B$ is \emph{nondegenerate} if
the series $\sum_{v\in \CC^0}T_v$ converges strictly to 1 in $M(B)$.
\end{defn}

\begin{prop}\label{rep-nd}
Let $\CC$ be a finitely aligned left cancellative small category, and let $T\:\CC\to B$ be a representation. Consider $T$ as a representation of $\CC$ in $C^*(T)$. Then $T$ is nondegenerate. 
\end{prop}

\begin{proof}
It follows from  
the proof of
\cite[Proposition~6.7]{Spi11}
that
\[
 C^*(T) =\clspn\{T_\alpha T_\beta^* q:\alpha,\beta\in \CC, q\in P\},
\]
where $P$ is the set of finite products of range projections of the $T_\alpha$.
Since 
the finite partial sums of the series $\sum_{v\in \CC^0}T_v$ are projections,
it suffices to show that for any 
generator
\[
a=T_\alpha T_\beta^* q
\]
the series
\[
\sum_{v\in \CC^0}T_va
\]
converges in norm to $a$.
Note that for all $v\in \CC^0$, 
\[
T_vT_\alpha=\begin{cases}T_\alpha\case v=r(\alpha)\\0\case v\ne r(\alpha).\end{cases}
\]
Thus the series $\sum_{v\in \CC^0}T_va$ has only 
one
nonzero term, and its sum is $a$.
\end{proof}

Since $C^*(t) = \talg \CC$ and $C^*(\, \tilde t\,) = \OO(\CC)$ we get:
\begin{cor}\label{univ nd}
Let $\CC$ be a finitely aligned left cancellative small category. The universal representations $t\:\CC\to \talg \CC$ and $\tilde t\:\CC\to \oalg \CC$ are nondegenerate in the sense of \defnref{nd}.
\end{cor}

In light of Proposition \ref{rep-nd}, we will often restrict our attention to nondegenerate representations of $\CC$ in the sequel. As the following lemma shows, this just means that we will work with nondegenerate  homomorphisms of the associated $C^*$-algebras.

\begin{lem}\label{nd phi_T}
Let $T\:\CC\to B$ be a representation,
and let $\phi_T\:\talg \CC\to B$ be the associated homomorphism.
Then $T$ is nondegenerate in the sense of \defnref{nd}
if and only if $\phi_T$ is nondegenerate in the usual sense that
\[
\clspn\{\phi_T(a)b:a\in \talg \CC,b\in B\}=B.
\]
Moreover, if $T$ is covariant, and  $\psi_T\:\oalg \CC\to B$ denote the associated homomorphism,
then $T$ is nondegenerate in the sense of \defnref{nd}
if and only if $\psi_T$ is nondegenerate in the usual sense.
\end{lem}

\begin{proof}
First, if $T$ is nondegenerate in the sense of \defnref{nd},
then for any $b\in B$
\[
\clspn\{
\phi_T(t_v)b:v\in \CC^0\}=B,
\]
so $\phi_T$ is nondegenerate.
Conversely, suppose $\phi_T$ is nondegenerate.
Then $\phi_T$ extends uniquely to a strictly continuous unital homomorphism, still denoted by $\phi_T$, from $M(\talg \CC)$ to $M(B)$.
Thus the series
\[
\sum_{v\in \CC^0}T_v=\sum_{v\in \CC^0}\phi_T(t_v)
\]
converges strictly to $1_{M(B)}$ by \propref{univ nd}. The proof of the final statement is similar and left to the reader.
\end{proof}

In \cite[Definition 4.5]{ExeSemigroupoid},
Exel defines a notion of tightness for  representations of semigroupoids in unital $C^*$-algebras. His definition adapts   to our context as follows.

\begin{defn}\label{exel-tight} Let  $\CC$ be left cancellative finitely aligned small category. If $L\subset \CC$ then a subset $H\subset L$ is called a \emph{covering} of $L$ if for every $\alpha\in L$ there is $\beta\in H$ such that $\alpha\CC\cap\beta\CC\ne\varnothing$.
Next, for finite subsets $F,K\subset \CC$, set
\[
\CC^{F,K}=\left(\bigcap_{\beta\in F}s(\beta)\CC\right)\cap \left(\bigcap_{\gamma\in K}\CC\minus s(\gamma)\CC\right).
\]
Then a representation $T\:\CC\to B$ in a $C^*$-algebra $B$ is said to be \emph{tight} 
if for every pair of finite subsets $F,K\subset \CC$
and for every finite covering $H$ of $\CC^{F,K}$ we have
\begin{equation}\label{tight}
\bigvee_{\alpha\in H}T_\alpha T_\alpha^*
=\prod_{\beta\in F}T_\beta^*T_\beta \prod_{\gamma\in K}(1-T_\gamma^*T_\gamma),
\end{equation}
where $1$ denotes the unit in $M(B)$. 
\end{defn}

For nondegenerate representations, this notion is equivalent to covariance. 

\begin{thm}\label{spiexe 2} 
Let $\CC$ be a finitely aligned left cancellative small category,
and let $T\:\CC\to B$ be a representation of $\CC$ in a $C^*$-algebra $B$.
If $T$ is tight, then it is covariant. On the other hand, if $T$ is nondegenerate and covariant, then
it is tight. 
\end{thm}

\begin{proof}
The proof, which is essentially due to Exel,  is an adaption of the proofs \cite[Proposition~7.3]{ExeSemigroupoid} and \cite[Proposition~7.4]{ExeSemigroupoid} to our context.   
 
1) Assume that $T$ is tight. Let $v\in \CC^0$ and let $H\subset \CC$ be finite and exhaustive at $v$. Setting $F={v}$ and $K = \varnothing$, it is clear that $H$ is a finite covering of $\CC^{F, K}$. Thus we get
\[ \bigvee_{\alpha \in H} T_\alpha T_\alpha^* = 
\prod_{\beta\in F}T_\beta^*T_\beta \, \prod_{\gamma\in K}(1-T_\gamma^*T_\gamma) 
=  T_v \,.\]
Hence $T$ is covariant. 

2) Assume that $T$ is nondegenerate and covariant. 
Let $F,K\subset \CC$ be finite and 
let  $H$ be a finite covering of $\CC^{F,K}$. We will show that equation (\ref{tight}) holds, thus proving that $T$ is tight.

a) Assume that $F\neq \varnothing$. We consider two subcases.
 
i)  Suppose that for all $v\in \CC^0$ we have $F\not \subset  \CC v$. Then we have $\bigcap_{\beta\in F}s(\beta)\CC = \varnothing$, so  $\CC^{F,K} = \varnothing$ and equation (\ref{tight}) amounts to  
\begin{equation} \label{tight-2}\prod_{\beta\in F}T_\beta^*T_\beta \prod_{\gamma\in K}(1-T_\gamma^*T_\gamma) = 0.
\end{equation}
But, as we can then pick $\beta_1, \beta_2 \in F$ such that $s(\beta_1)\neq s(\beta_2)$, we get that $\prod_{\beta\in F}T_\beta^*T_\beta = \prod_{\beta\in F}T_{s(\beta)}= 0$, so equation (\ref{tight}) is satisfied.

ii) Suppose that $F\subset  \CC v \ \text{for some}\  v\in \CC^0$. Then  $\bigcap_{\beta\in F}s(\beta)\CC = v\CC$, so
\[
\CC^{F,K}=v\CC\cap \left(\bigcap_{\gamma\in K}\CC\minus s(\gamma)\CC\right) 
= v\CC\cap \left(\CC \minus \bigcup_{\gamma\in K} s(\gamma)\CC\right) .
\]

Assume first that $v \in \{s(\gamma): \gamma \in K\}$. Then we have $\CC^{F,K} = \varnothing$, so 
equation (\ref{tight}) reduces to equation (\ref{tight-2}). Now, we have $T_v (1-T_{s(\gamma)}) = 0$ for at least one $\gamma \in K$. Thus
\[\prod_{\beta\in F}T_\beta^*T_\beta \, \prod_{\gamma\in K}(1-T_\gamma^*T_\gamma) 
=  T_v  \, \prod_{\gamma\in K}(1-T_{s(\gamma)})= 0. \] 
Hence equation (\ref{tight}) holds in this case.

Next, assume that $v \not \in \{s(\gamma): \gamma \in K\}$. Then we get $\CC^{F,K} =v\CC$. Now $H$ is a finite covering of $\CC^{F,K}= v\CC$, which means that $H$ is exhaustive at $v$. Since $T$ is assumed to be covariant,  we have 
\[
T_v=\bigvee_{\alpha\in H}T_\alpha T_\alpha^*.
\]
As $\prod_{\beta\in F}T_\beta^*T_\beta = T_v$ and $T_v (1 - T_{s(\gamma)}) = T_v$ for all $\gamma \in K$, we get 
\[\prod_{\beta\in F}T_\beta^*T_\beta \, \prod_{\gamma\in K}(1-T_\gamma^*T_\gamma) 
= T_v \, \prod_{\gamma\in K}(1-T_{s(\gamma)}) = T_v = \bigvee_{\alpha\in H}T_\alpha T_\alpha^*.
\]
Thus equation (\ref{tight}) is satisfied in this case too.

b) Assume  that $F =\varnothing$. Then we have \[\CC^{\varnothing,K} = \bigcap_{\gamma\in K}\CC\minus s(\gamma)\CC = \CC \minus \bigcup_{\gamma\in K} s(\gamma)\CC.\] 
We consider two subcases. 

i) Suppose $\CC^{\varnothing,K} =\varnothing$. In other words, we have $\CC = \bigcup_{\gamma\in K} s(\gamma)\CC$. This implies that $\CC^0 = \{ s(\gamma) : \gamma \in K\}$, hence that $\CC^0$ is finite.  Since $T$ is nondegenerate by assumption, we have $\sum_{v\in \CC^0} T_v = 1$, which gives that  $\prod_{\gamma \in K} (1-T^*_\gamma T_\gamma) = 0$. It follows  readily that equation (\ref{tight}) is satisfied.  

ii) Suppose  $\CC^{\varnothing,K} \neq \varnothing$. Setting $V= \CC^0 \minus \{s(\gamma): \gamma \in K\}$
we have $\CC^{\varnothing,K} = \bigcup_{v \in V} v\CC$. Since $H$ is a finite covering of $\CC^{\varnothing,K}$,  $V$ has to be finite, so $\CC^0 = V \cup \{s(\gamma): \gamma \in K\}$ is finite.
Since $T$ is nondegenerate we have $\sum_{v\in \CC^0} T_v = 1$, so we get
\[ \prod_{\gamma\in K}(1-T_\gamma^*T_\gamma) =  \prod_{\gamma\in K}(1-T_{s(\gamma)})
= 1 - \sum_{\gamma \in K} T_{s(\gamma)} = \sum_{v\in V} T_v.
\] 
Since $H \subset \CC^{\varnothing,K}= \bigcup_{v \in V} v\CC$, we have $H = \bigcup_{v \in V} H_v$ (disjoint union), where $H_v := H \cap \,v\CC$ for each $v \in V$. Moreover, since $H$ is a finite covering of $\CC^{\varnothing,K}$, each $H_v$ is a finite covering of $v\CC$, that is, each $H_v$ is exhaustive at $v \in V$. Using that $T$ is covariant we get 
 \[\bigvee_{\alpha \in H} T_\alpha T^*_\alpha = \bigvee_{v \in V} \Big( \bigvee_{\alpha \in H_v} T_\alpha T^*_\alpha \Big)=  \bigvee_{v \in V} T_v =  \sum_{v\in V} T_v =  \prod_{\gamma \in K} (1-T^*_\gamma T_\gamma).\]
 Thus equation (\ref{tight}) is satisfied in this case too.
\end{proof}

Using Theorem \ref{spiexe 2} we readily get:
\begin{cor}
Let $\CC$ be a finitely aligned left cancellative small category. The universal covariant representation $\tilde t\: \CC \to \oalg \CC$ is tight. Moreover, $(\oalg \CC, \tilde t\,)$ is universal for tight representations of $\CC$ in $C^*$-algebras.
\end{cor}

\begin{rem}
In \cite[Definition~4.6]{ExeSemigroupoid}, Exel insists that
the universal tight representation of $\CC$ take values in a unital $C^*$-algebra.
We suspect that the only reason
for this
is so that he can extend representations of $\CC$ to the unitization $\tilde \CC$,
which is a device he introduces in order to deal with the lack of identity morphisms.
Since we deal exclusively with small categories,
we can safely
ignore Exel's requirement
of a unit.
\end{rem}

\section{Cocycles and categories of paths}\label{act cop}

Let $\CC$ be a small category, and assume that a group $G$ acts on the set $\CC$ by permutations
in such a way that
\begin{equation}\label{act on CC}
r(g\alpha)=gr(\alpha)\midtext{and}s(g\alpha)=gs(\alpha)\righttext{for all}g\in G,\alpha\in \CC.
\end{equation}
Observe that we do \emph{not} assume that $G$ acts by automorphisms of the category --- in fact, we will typically not want this!
Note that for all $g\in G$ and $ v\in \CC^0$ we have $gv=gr(v) = r(gv) \in \CC^0$.
It follows easily that $G\curvearrowright \CC^0$ by restriction.

\begin{defn} \label{category cocycle}
With the above notation,
if $\varphi\:G\times \CC\to G$ is a cocycle for the action of $G$ on $\CC$ as a set,
we call $\varphi$ a \emph{category cocycle} for this action if
for all $g\in G$, $v\in \CC^0$, and $(\alpha,\beta) \in \CC^{(2)}$
we also have
\begin{enumerate}
\item \label{coc 1}
$\varphi(g,v)=g$;

\item \label{coc 2}
$\varphi(g,\alpha)s(\alpha)=gs(\alpha)$;

\item \label{coc 3}
$g(\alpha\beta)=(g\alpha)(\varphi(g,\alpha)\beta)$;\,\footnote{\,Using property (\ref{coc 2}), one easily sees that $(g\alpha, \varphi(g,\alpha)\beta) \in \CC^{(2)}$.}

\item \label{coc 4}
$\varphi(g,\alpha\beta)=\varphi\bigl(\varphi(g,\alpha),\beta\bigr)$,
\end{enumerate}
and we call $(\CC,G,\varphi)$ a \emph{category system}. In the case where $\CC$ is a category of paths, we call $(\CC,G,\varphi)$ a \emph{path system}.
\end{defn}

\begin{rem}
In \eqref{act on CC},
the first condition is crucial,
but the second could be dropped without altering our results in the sequel,
provided that condition (ii) in Definition \ref{category cocycle} is replaced by the following condition: 
\begin{enumerate}
\item[(ii')]
$\varphi(g,\alpha)s(\alpha)=s(g\alpha)$;
\end{enumerate}
This would add somewhat more flexibility to the theory.
We thank a referee for this observation.
\end{rem}

\begin{ex}\label{trivialcocy}
A rather trivial way to get a category cocycle for an action $G\curvearrowright \CC$ is to define $\varphi(g,\alpha)=g$ for all $g\in G,\alpha\in\CC$,
and this is the only case where we are guaranteed that
$G$
acts
on $\CC$ by automorphisms of the category.
\end{ex}

\begin{ex}\label{EP ex}
Suppose $(E,G,\varphi)$ is an Exel-Pardo system.
For $n \geq 2$, we let $E^n$ denote the set of all paths in $E$ of length $n$, and let $E^* = \bigcup_{n \geq 0}  E^n$ denote the set of all finite paths in $E$. We will consider $E^*$ as a category of paths, as defined in \cite{Spi11}, composition being defined by concatenation of paths whenever it makes sense.
 
In \cite{EP}, Exel and Pardo consider 
countable groups
and 
finite graphs without sources.
However, their proof of \cite[Proposition 2.4]{EP} shows that without any restriction on $G$ and $E$, the action of $G$ on $E$ 
extends 
to an action 
of $G$ on the category of paths $E^*$,
and also that the cocycle $\varphi\:G\times E^1\to G$
extends uniquely to a category cocycle on $E^*$, also denoted by $\varphi$.

Their construction may 
roughly be described as follows:
First, define $\varphi$ on $G\times E^0$ by (i).
Next, as guided by (iii) and (iv), inductively define the action of $G$ on $E^{n+1}$ and the map $\varphi$ on $G\times E^{n+1}$ for each $n\geq 1$:
for each $g\in G,e\in E^1,$ and $\alpha\in E^n$ such that $e\alpha \in E^{n+1},$ set
\[
g(e\alpha):=(ge)(\varphi(g,e)\alpha), 
\quad
\varphi(g,e\alpha):=\varphi(\varphi(g,e),\alpha).
\]
Thus we obtain a path system $(E^*, G, \varphi)$. Note that since we only assume that equation (\ref{source-inv}) holds, while Exel and Pardo require that $\varphi(g, e)  v = gv$ holds
 for all $g\in G, e \in E^1$, and $ v \in E^0$, property (ii) above is weaker than the corresponding one in \cite{EP}.  
\end{ex}

\begin{rem}
Category systems have also recently been considered by H.~Li and D.~Yang \cite{li-yang} in the case where $\CC$ is a higher-rank graph; in their terminology, a category cocycle is called a restriction map. 
\end{rem} 

To a category system $(\CC,G,\varphi)$  one may associate a small category that we will call its Zappa-Sz{\'e}p product.  Generalizing earlier works of Zappa, Sz{\'e}p (and others) in the case where $\CC$ is a group, such a product has been introduced and studied by Brin in the context of more general multiplicative structures, such as monoids and categories, see~\cite{Brin}. For the 
convenience
of the reader, we give below the details of this construction for a category system.
\begin{prop}\label{product}
Let $(\CC,G,\varphi)$ be a category system.
Put $\DD=\CC\times G$
and $\DD^0=\CC^0\times \{1\}$,
and define $r,s\:\DD\to\DD^0$ by
\[
r(\alpha,g)=\bigl(r(\alpha),1\bigr)
\midtext{and}
s(\alpha,g)=\bigl(g\inv s(\alpha),1\bigr).
\]
For $(\alpha,g),(\beta,h)\in \DD$ with $s(\alpha,g)=r(\beta,h)$, 
define
\begin{equation}\label{zz-prod}
(\alpha,g)(\beta,h)=\bigl(\alpha(g\beta),\varphi(g,\beta)h\bigr).
\end{equation}
Then
$\DD$ is a small category.
Moreover, if $\CC$ is left cancellative then so is $\DD$.
\end{prop}

\begin{proof} Note first that if $s(\alpha,g)=r(\beta,h)$ then $g^{-1}s(\alpha) = r(\beta)$, hence $r(g\beta)= gr(\beta) = g(g^{-1}s(\alpha))=s(\alpha)$, so the expression
on the right hand side of (\ref{zz-prod}) is well-defined.
Next, let $(\alpha, g), (\beta, h), (\gamma,k) \in \DD$ with $s(\alpha,g)=r(\beta,h)$ and $s(\beta,h)=r(\gamma,k)$. Then
\begin{align*}
r\bigl((\alpha,g)(\beta,h)\bigr)
&=\bigl(r\bigl(\alpha(g\beta)\bigr),1\bigr)
=\bigl(r(\alpha),1\bigr)
=r(\alpha,g),
\end{align*}
\begin{align*}
s\bigl((\alpha,g)(\beta,h)\bigr)
&=\bigl(\bigl(\varphi(g,\beta)h\bigr)\inv s\bigl(\alpha(g\beta)\bigr),1\bigr)
\\&=\bigl(h\inv \varphi(g,\beta)\inv s(g\beta),1\bigr)
\\&=\bigl(h\inv \varphi(g,\beta)\inv gs(\beta),1\bigr)
\\&=\bigl(h\inv s(\beta),1\bigr)
\\&=s(\beta,h),
\end{align*}
and
\begin{align*}
(\alpha,g)\bigl((\beta,h)(\gamma,k)\bigr)
&=(\alpha,g)\bigl(\beta(h\gamma),\varphi(h,\gamma)k\bigr)
\\&=\bigl(\alpha\bigl(g(\beta(h\gamma))\bigr),\varphi\bigl(g,\beta(h\gamma)\bigr)\varphi(h,\gamma)k\bigr)
\\&=\bigl(\alpha(g\beta)\bigl(\varphi(g,\beta)(h\gamma)\bigr),\varphi\bigl(\varphi(g,\beta),h\gamma\bigr)\varphi(h,\gamma)k\bigr)
\\&=\bigl(\alpha(g\beta)\bigl(\varphi(g,\beta)h\bigr)\gamma,\varphi\bigl(\varphi(g,\beta)h,\gamma\bigr)k\bigr)
\\&=\bigl(\alpha(g\beta),\varphi(g,\beta)h\bigr)(\gamma,k)
\\&=\bigl((\alpha,g)(\beta,h)\bigr)(\gamma,k).
\end{align*}
Moreover, for any $(v,1)\in \DD^0$ and $(\alpha, g) \in \DD$ , we have
\[
r(v,1)=\bigl(r(v),1\bigr)=(v,1),
\]
\[
s(v,1)=\bigl(1\inv s(v),1\bigr)=\bigl(s(v),1\bigr)=(v,1),
\]
\begin{align*}
r(\alpha,g)(\alpha,g)
&=\bigl(r(\alpha),1\bigr)(\alpha,g)
\\&=\bigl(r(\alpha)(1\alpha),\varphi(1,\alpha)g\bigr)
\\&=\bigl(r(\alpha)\alpha,1g\bigr)
\\&=(\alpha,g),
\end{align*}
and
\begin{align*}
(\alpha,g)s(\alpha,g)
&=(\alpha,g)\bigl(g\inv s(\alpha),1\bigr)
\\&=\bigl(\alpha\bigl(g(g\inv s(\alpha))\bigr),\varphi\bigl(g,g\inv s(\alpha)\bigr)1\bigr)
\\&=\bigl(\alpha s(\alpha),\varphi(g,s(g\inv \alpha))\bigr)
\\&=(\alpha,g).
\end{align*}

Thus
$\DD$ is a small category.
Assume now that  $\CC$ is left cancellative, and 
suppose that 
\[
(\alpha,g)(\beta,h)=(\alpha,g)(\gamma,k).
\]
Then
$\alpha(g\beta)=\alpha(g\gamma)$,
so $g\beta=g\gamma$ since $\CC$ is left cancellative,
and hence $\beta=\gamma$.
Then we also have
\[
\varphi(g,\beta)h=\varphi(g,\gamma)k=\varphi(g,\beta)k,
\]
so $h=k$.
Therefore $(\beta,h)=(\gamma,k)$. Hence,  $\DD$ is left cancellative.
\end{proof}

\begin{defn}\label{ZS}
If $(\CC,G,\varphi)$ is a category system 
we will denote the small category $\DD$ defined above by $\CC\rtimes^\varphi G$,
and call it the \emph{Zappa-Sz\'ep product} of  $(\CC,G,\varphi)$.
\end{defn}

\begin{rem} \label{groupoid} 
If $(\CC, G, \varphi)$ is a category system and $\CC$ is a groupoid (so every element of $\CC$ is invertible), then one easily verifies that $\CC\rtimes^\varphi G$ is a groupoid. Moreover, if $\CC$ is a group, then $\CC\rtimes^\varphi G$ is also a group, as in the original construction of the Zappa-Sz{\' e}p product. We note that the Zappa-Sz{\' e}p product of topological groupoids has recently been studied in \cite{BPRRW}.
\end{rem}

\begin{rem} \label{monoid} In our definition of a category system $(\CC, G, \varphi)$, one may instead assume that $G$ is a 
monoid  acting on the set $\CC$  by permutations. For each $g\in G$ let 
$\sigma_g$ denote the associated permutation of $\CC$, so $\sigma_g(\alpha)=g\alpha$  for all $\alpha\in \CC$, and set $g^{-1}\alpha:= \sigma_g^{-1}(\alpha)$ for all $\alpha \in \CC$. The monoid $G$ still acts on $\CC^0$ by restriction\,\footnote{ \,Indeed, let $g\in G, v\in \CC^0$. Then, as for a group, we get  $gv \in \CC^0$. Moreover, $gr(g^{-1}v) = r(gg^{-1}v)= r(v) = v$, so $g^{-1}v = r(g^{-1}v) \in \CC^0$. It follows that $\sigma_g$ restricts to a permutation of $\CC^0$.}, so the definition of $\DD=\CC\rtimes^\varphi G$ in Proposition~\ref{product} continues to make sense. It should be clear from our proof of this proposition that we again get a small category, which is left cancellative if  $\CC$ and $G$ 
are both
left cancellative.

If $\CC$ is a left cancellative monoid, considered as a small category with vertex set consisting of the identity element, and $G$ is a left cancellative monoid, the resulting Zappa-Sz{\'e}p product $\CC\rtimes^\varphi G$ will also be a left cancellative monoid. We refer to \cite{BRRW} for many interesting examples illustrating this special situation. 

In the sequel, we will only consider category systems where $G$ is a group.
\end{rem}

\begin{rem}
Even if $\CC$ has no inverses, it is possible for $\DD= \CC\rtimes^\varphi G$ to have inverses other than vertices,
because if $v\in \CC^0$ and $g\ne 1$ then
$(v,g)\notin \DD^0$, but it is not difficult to check that
\begin{align*}
(v,g)\bigl(g\inv v, g\inv\bigr) &=  (v,1) \in \DD^0
\midtext{and}
\\
\bigl(g\inv v, g\inv\bigr)(v,g) &= (g\inv v,1)\in\DD^0,
\end{align*}
so $(v,g)$ is invertible in $\DD$.
 
However, we can find a large subcategory of $\DD$ with no inverses, at least when $(\CC,G,\varphi)$ is a path system:
\end{rem}

\begin{prop}\label{subpath}
Let $(\CC,G,\varphi)$ be a path system.
Define $\DD$ as above.
Then the subset
\[
\EE:=\{(\alpha,g)\in\DD:\alpha\notin \CC^0\text{ or }g=1\}
\]
is a subcategory of $\DD$ 
that
is a category of paths.
\end{prop}

\begin{proof}
First note that $\DD^0\subset \EE$.
Suppose $(\alpha,g),(\beta,h)\in\EE$ and $s(\alpha,g)=r(\beta,h)$.
Then
\[
(\alpha,g)(\beta,h)=\bigl(\alpha(g\beta),\varphi(g,\beta)h\bigr).
\]

There are two cases to consider:

Case 1.
$\alpha\notin \CC^0$ or $\beta\notin \CC^0$.
In the case $\beta\notin \CC^0$ we also have $g\beta\notin \CC^0$.
Thus $\alpha(g\beta)\notin\CC^0$ since $\CC$ has no inverses.
Thus $(\alpha,g)(\beta,h)\in\EE$ in this case.

Case 2.
$\alpha,\beta\in\CC^0$.
Then $g=h=1$,
so
\[
\varphi(g,\beta)=\varphi(1,\beta)=1,
\]
and hence
$(\alpha,g)(\beta,h)\in\EE$ in this case also. 

Note also that if we assume that $(\alpha,g)(\beta,h)\in \DD^0$, then we can conclude that $\alpha, \beta \in \CC^0$ and $g=h=1$.

Thus 
$\EE$ is a subcategory of $\DD$ with $\EE^0 = \DD^0$, 
and $\EE$
is left cancellative (using Proposition~\ref{product}) and has no inverses.
\end{proof}

\begin{defn}
If $(\CC,G,\varphi)$ is a path system 
we will denote the category of paths $\EE$ defined above by $\CC\rtimes_0^\varphi G$,
and call it the \emph{restricted Zappa-Sz\'ep product} of  $(\CC,G,\varphi)$.
\end{defn}

Let $(\CC, G, \varphi)$ be a category system with $\CC$ left cancellative. Arguing as in the proof of Proposition~\ref{subpath}, one sees that $(\alpha, g)$ is invertible in $\DD=\CC\rtimes^\varphi G$ if and only if $\alpha$ is invertible in $\CC$.
It follows that for $(\alpha, g), (\beta, h)\in \DD$, we have
\[
(\alpha,g)\sim (\beta,h)\midtext{if and only if}\alpha\sim\beta.
\]

\begin{prop}\label{finite preserved}
Let $(\CC,G,\varphi)$ be a category system with $\CC$ left cancellative.  Write $\DD = \CC \rtimes^\varphi G$.  Let $\alpha$, $\beta \in \CC$ and $g$, $h \in G$.
\begin{enumerate}
\item $(\alpha,g) \DD \cap (\beta,h) \DD = (\alpha \CC \cap \beta \CC) \times G$.
\item If $\CC$ is finitely aligned, then so is $\DD$, and $(\alpha,g) \vee (\beta,h) = (\alpha \vee \beta) \times \{1\}$. \(In particular, if $\CC$ is singly aligned then so is $\DD$.\)
\end{enumerate}
\end{prop}

\begin{proof}
(i) $\subseteq$:  
Assume $z \in (\alpha,g) \DD \cap (\beta,h) \DD$.  Then \[z = (\alpha,g) (\gamma,k) = (\beta,h) (\delta,\ell)\] for some $(\gamma, k), (\delta, \ell) \in \DD$. Thus $z=(\alpha(g \gamma), \varphi(g ,\gamma)k) = (\beta(h \delta), \varphi(h,\delta) \ell)$, hence in particular $\alpha(g \gamma) = \beta(h \delta) \in \alpha \CC \cap \beta \CC$.  Therefore $z \in (\alpha \CC \cap \beta \CC) \times G$. 

$\supseteq$:  
Assume $z \in (\alpha \CC \cap \beta \CC) \times G$. Then $z = (\varepsilon,m)$ where $\varepsilon \in \alpha \CC \cap \beta \CC$ and $m \in G$.
Write 
$\varepsilon=\alpha\lambda$ with $\lambda \in \CC$.
Since $(\alpha,g)=(\alpha,1)(s(\alpha),g)$
and $(s(\alpha),g)$ is invertible, we have $(\alpha,1)\sim (\alpha,g)$. So we get
\[
z=(\alpha\lambda,m)=(\alpha,1)(\lambda,m)\in (\alpha,1)\DD=(\alpha,g)\DD,
\]
and similarly $z\in (\beta,h)\DD$.

(ii) Suppose that $\CC$ is finitely aligned.  Let $(\alpha,g)$, $(\beta,h) \in \DD$.
One easily checks that $\gamma\CC \times G = (\gamma, 1)\DD$ for every $\gamma \in \CC$.
Using (i) we get
\begin{align*}
(\alpha,g) \DD \cap (\beta,h) \DD &= (\alpha \CC \cap \beta \CC) \times G = \big(\bigcup_{\gamma\in\, \alpha\vee \beta}\gamma\CC\big)\times G \\ 
&= \bigcup_{\gamma\in\, \alpha\vee \beta}\big(\gamma\CC \times G\big) = \bigcup_{\gamma\in \,\alpha\vee \beta}(\gamma, 1)\DD\\
&= \bigcup_{
\zeta\,\in \,(\alpha\vee \beta)\times \{1\}}
\zeta\DD.
\end{align*}
Since $\alpha\vee\beta$ is finite and independent, so is $(\alpha\vee \beta)\times \{1\}$. Hence $\DD$ is finitely aligned and $(\alpha,g) \vee (\beta,h) = (\alpha \vee \beta) \times \{1\}$.
\end{proof}

\begin{cor}
Let $(E, G, \varphi)$ be an Exel-Pardo system. Then the resulting Zappa-Sz\'ep product $E^*\rtimes^\varphi G$ is a singly aligned left cancellative small category.
\end{cor}
\begin{proof} 
Since $E^*$ is a singly aligned category of paths this follows from Proposition~\ref{product} and Proposition~\ref{finite preserved}.
\end{proof}

Later we will need the following elementary lemma:

\begin{lem}\label{exhaust preserved}
Let $(\CC,G,\varphi)$ be a category system with $\CC$ left cancellative, and assume that $\CC$ is finitely aligned.
Let $\DD=\CC\rtimes^\varphi G$ be the Zappa-Sz\'ep product.
Take any vertex $(v,1)$ of $\DD$,
and let $F\subset (v,1)\DD$.
Put
\[
H=\{\alpha \in v\CC:\text{there is $g\in G$ such that $(\alpha,g)\in F$}\}.
\]
Then $F$ is exhaustive at $(v,1)$ if and only if $H$ is exhaustive at $v$.
\end{lem}

\begin{proof}
Since $(\alpha, 1) \sim (\alpha, g)$ for all $\alpha\in \CC$ and $g\in G$,  we can use \lemref{exhaust inv} to assume without loss of generality that
\[
F=H\times \{1\}.
\]
First assume that $F$ is exhaustive at $(v,1)$.
To see that $H$ is exhaustive at $v$, let $\beta\in v\CC$.
Then $(\beta,1)\in (v,1)\DD$,
so we can choose $(\alpha,1)\in F$ such that $(\alpha,1)\DD\cap (\beta,1)\DD\ne \varnothing$.
Since $(\alpha,1)\DD=\alpha\CC\times G$ and $(\beta,1)\DD=\beta\CC\times G$,
we have $\alpha\CC\cap \beta\CC\ne \varnothing$.
Since $\alpha\in H$, we have shown that $H$ is exhaustive at $v$.

Conversely, assume that $H$ is exhaustive at $v$.
To see that $F$ is exhaustive at $(v,1)$, let $(\beta,g)\in (v,1)\DD$.
Then $\beta\in v\CC$, so we can choose $\alpha\in H$ such that $\alpha\CC\cap \beta\CC\ne \varnothing$.
Since $(\beta,g)\sim (\beta,1)$,
we have
\begin{align*}
(\alpha,1)\DD\cap (\beta,g)\DD
&=(\alpha,1)\DD\cap (\beta,1)\DD
\\&=(\alpha\CC\times G)\cap (\beta\CC\times G)
\\&=(\alpha\CC\cap \beta\CC)\times G
\\&\ne \varnothing.
\end{align*}
Since $(\alpha,1)\in F$, we have shown that $F$ is exhaustive at $(v,1)$.
\end{proof}

\begin{rem}
Suppose that $(\CC, G, \varphi)$ is a path system, and let $\EE = \CC \rtimes^\varphi_0 G$ be its restricted  Zappa-Sz\'ep product.
In this situation, finite alignment need not pass from $\CC$ to $\EE$.  

For example, let $\alpha \in \CC \setminus \CC^0$ and $g_1$, $g_2 \in G$.  Let's consider the set $(\alpha,g_1) \EE \cap (\alpha,g_2) \EE$.  A typical element has the form $(\alpha,g_1) (\gamma,k) = (\alpha, g_2) (\delta,\ell)$, for which $\alpha(g_1 \gamma) = \alpha (g_2 \delta)$ and $\varphi(g_1,\gamma)k = \varphi(g_2,\delta)\ell$, and where at least one of $\gamma$ and $\delta$ is not in $\CC^0$.  By left cancellation in $\CC$ we have $g_1 \gamma = g_2 \delta$, and hence $\delta = g_2^{-1}g_1 \gamma$.  Then we must have $\gamma$, $\delta \not\in \CC^0$, and $\ell = \varphi(g_2, g_2^{-1}g_1 \gamma)^{-1} \varphi(g_1, \gamma) k = \varphi(g_2^{-1}, g_1 \gamma) \varphi(g_1,\gamma) k = \varphi(g_2^{-1} g_1, \gamma) k$.  Thus $k \in G$ can be arbitrary, and we have that $(\alpha,g_1) \EE \cap (\alpha,g_2) \EE = (\alpha \CC \setminus \{\alpha\})\times G$.  If $\alpha \CC \setminus \{ \alpha \}$ has infinitely many independent elements, e.g. if $\CC$ is the path category of a non row-finite graph, then $(\alpha,g_1) \EE \cap (\alpha,g_2) \EE$ might not be a finite union of 
of sets of the form $\epsilon \EE$ with $\epsilon \in \EE$.
\end{rem}

\section{Representations of Zappa-Sz\'ep products}\label{ZS algebras}

Let $\CC$ be a finitely aligned left cancellative small category,
and  let $(\CC,G,\varphi)$ be a category system, as in \secref{act cop}, so that
$G$ is a discrete group acting by permutations on the set $\CC$
and $\varphi\:G\times \CC\to \CC$ is a category cocycle in the sense of
Definition \ref{category cocycle}.
Let $\DD=\CC\rtimes^\varphi G$ be the Zappa-Sz\'ep product as in \defnref{ZS}.
Then $\DD$ is a left cancellative
small category by \propref{product},
and moreover is finitely aligned by \propref{finite preserved}.
In this section we present a semidirect-product-type description of the $C^*$-algebras $\talg \DD$ and $\oalg \DD$.

\begin{defn}\label{covariant}
A \emph{representation} of the category system $(\CC,G,\varphi)$ in a $C^*$-algebra $B$
is a pair $(T,U)$,
where $T\:\CC\to B$ is a nondegenerate representation
and $U\:G\to M(B)$ is a unitary representation such that
\[
U_gT_\alpha=T_{g\alpha}U_{\varphi(g,\alpha)}\righttext{for all}g\in G,\alpha\in\CC.
\]
\end{defn}

\begin{thm}\label{integrated form}
There is a bijective correspondence $(T,U) \mapsto R$ between representations $(T,U)$ of $(\CC,G,\varphi)$ in $B$ and nondegenerate representations $R$ of $\DD$ in $B$, given as follows:
\begin{enumerate}
\item If $(T,U)$ is a representation of $(\CC,G,\varphi)$ in $B$,
then $R\:\DD\to B$ is defined by
\[
R_{(\alpha,g)}=T_\alpha U_g.
\]

\item If $R\:\DD\to B$ is a nondegenerate representation,
then $(T,U)$ is defined by
\begin{align*}
T_\alpha&=R_{(\alpha,1)}\\
U_g&=\sum_{v\in \CC^0}R_{(v,g)},
\end{align*}
where the series converges strictly in $M(B)$.
\end{enumerate}
Note: Later we will  denote by $T\times U$ the nondegenerate representation $R$ of $\DD$  associated to a representation $(T,U)$ of $(\CC,G,\varphi)$. 
\end{thm}
\begin{proof}
Given a representation $(T,U)$ of $(\CC,G,\varphi)$ in $B$,
define $R\:\DD\to B$ by
\[
R_{(\alpha,g)}=T_\alpha U_g.
\]
We need to show that $R$ is a nondegenerate representation of $\DD$.

First we show that
$R$ satisfies (\ref{spi 1}) in \defnref{spi rep}:
for $(\alpha,g)\in\DD$ we have
\begin{align*}
R_{(\alpha,g)}^*R_{(\alpha,g)}
&=(T_\alpha U_g)^*T_\alpha U_g
\\&=U_{g\inv}T_\alpha^*T_\alpha U_g
\\&=U_{g\inv}T_{s(\alpha)}U_g
\\&=T_{g\inv s(\alpha)}U_{\varphi(g\inv,s(\alpha))}U_g
\\&=T_{g\inv s(\alpha)}U_{g\inv}U_g
\\&=T_{g\inv s(\alpha)}
\\&=R_{(g\inv s(\alpha),1)}
\\&=R_{s(\alpha,g)}.
\end{align*}

Next we show that
$R$ satisfies \ref{spi 2} in \defnref{spi rep}:
for $(\alpha,g),(\beta,h)\in\DD$,
if $s(\alpha,g)=r(\beta,h)$ then
$s(\alpha)=r(g\beta)$, so
\begin{align*}
R_{(\alpha,g)}R_{(\beta,h)}
&=T_\alpha U_g T_\beta U_h
\\&=T_\alpha T_{g\beta} U_{\varphi(g,\beta)}U_h
\\&=T_{\alpha (g\beta)}U_{\varphi(g,\beta)h}
\\&=R_{(\alpha (g\beta),\varphi(g,\beta)h)}
\\&=R_{(\alpha,g)(\beta,h)}.
\end{align*}

Now we show that
$R$ satisfies \ref{spi 3} in \defnref{spi rep}:
for $(\alpha,g),(\beta,h)\in \DD$,
\begin{align*}
R_{(\alpha,g)}R_{(\alpha,g)}^*R_{(\beta,h)}R_{(\beta,h)}^*
&=T_\alpha U_g U_g\inv T_\alpha^* T_\beta U_h U_{h\inv} T_\beta^*
\\&=T_\alpha T_\alpha^* T_\beta T_\beta^*
\\&=\bigvee_{\gamma\in \alpha\vee\beta}T_\gamma T_\gamma^*
\\&=\bigvee_{(\gamma,1)\in (\alpha\vee\beta)\times \{1\}}T_\gamma U_1U_1^*T_\gamma^*
\\&\overset{*}=\bigvee_{(\gamma,1)\in (\alpha,g)\vee (\beta,h)}R_{(\gamma,1)}R_{(\gamma,1)}^*,
\end{align*}
where the equality at $*$ follows from \propref{finite preserved}.
Thus $R$ is a representation of $\DD$ in $B$.

We will now show that $R$ is nondegenerate.
Since the finite partial sums of the series $\sum_{(v,1)\in \DD^0}R_{(v,1)}$ are projections,
and since $T$ is nondegenerate,
it suffices to show that
for
elements of the form
\[
a=\sum_{w\in F}T_wb,
\]
where $F\subset \CC^0$ is finite and $b\in B$,
the series
\[
\sum_{(v,1)\in \DD^0}R_{(v,1)}a
\]
converges in norm to $a$.
First note that for all $v,w\in\CC^0$,
\[
R_{(v,1)}T_w=T_vT_w=\begin{cases}T_w\case v=w\\0\case v\ne w,\end{cases}
\]
and so
\[
R_{(v,1)}a=\begin{cases}T_vb\case v\in F\\0\case v\notin F.\end{cases}
\]
Thus the series
\[
\sum_{(v,1)\in \DD^0}R_{(v,1)}a
\]
only has finitely many nonzero terms, and the sum is $a$.

Conversely, suppose that $R\:\DD\to B$ is a nondegenerate representation.
Define $T\:\CC\to B$ and $U\:G\to M(B)$ by
\begin{align*}
T_\alpha&=R_{(\alpha,1)},\\
U_g&=\sum_{v\in \CC^0}R_{(v,g)}.
\end{align*}
We will show that for
$g\in G$ the sum
\[
U_g=\sum_{v\in \CC^0}R_{(v,g)}
\]
converges strictly to a unitary element of $M(B)$,
and the map $g\mapsto U_g$ is a homomorphism.
We will first show that for each $g\in G$ the set of finite partial sums of the series $\sum_{v\in \CC^0}R_{(v,g)}$ consists of partial isometries, hence is norm bounded.
Since each term $R_{(v,g)}$ is a partial isometry, it suffices to show that the range projections
of the family $\{R_{(v,g)}:v\in \CC^0\}$
are mutually orthogonal and also the domain projections are mutually orthogonal.
Let $v,w\in \CC^0$ with $v\ne w$.
For the range projections,
\[
R_{(v,g)}R_{(v,g)}^*R_{(w,g)}R_{(w,g)}^*=\bigvee_{(\gamma, k) \in (v,g)\vee (w,g)} R_{(\gamma, k)}\ = 0
\]
since $(v,g)\vee (w,g) = (v\vee w)\times\{1\} = \varnothing$.

For the domain projections,
\begin{align*}
R_{(v,g)}^*R_{(v,g)}R_{(w,g)}^*R_{(w,g)}
&=R_{s(v,g)}R_{s(w,g)}
=R_{(v,1)}R_{(w,1)}
=0
\end{align*}
since $(v,1)$ and $(w,1)$ are distinct elements of $\DD^0$.

Next we will
show that for any
element of the form
\[
a=\sum_{w\in F}R_{(v,1)}b,
\]
where $F\subset \CC^0$ is finite and $b\in B$,
the two series
\[
\sum_{v\in \CC^0}R_{(v,g)}a
\midtext{and}
\sum_{v\in \CC^0}R_{(v,g)}^*a
\]
converge in norm.
First note that for all $v,w\in \CC^0$,
\begin{align*}
R_{(v,g)}R_{(w,1)}
&=R_{(v,g)(w,1)}
=R_{(v(gw),\varphi(g,w)}
=\begin{cases}R_{(gw,g)}\case v=gw\\0\case v\ne gw.\end{cases}
\end{align*}
Thus with $a=\sum_{w\in F}R_{(w,1)}b$ as above,
only finitely many terms of the series
\[
\sum_{v\in \CC^0}R_{(v,1)}a
\]
are nonzero, so the series converges in norm.
Similarly,
\[
R_{(v,g)}^*R_{(w,1)}
=(R_{(w,1)}R_{(v,g)})^*
=\begin{cases}R_{(v,g)}^*\case w=v\\0\case w\ne v,\end{cases}
\]
so
with $a$ as above
only finitely many terms of the series
\[
\sum_{v\in \CC^0}R_{(v,1)}^*a
\]
are nonzero, 
and hence
the series converges in norm.

Thus we are now assured that for $g\in G$ the series
\[
\sum_{v\in \CC^0}R_{(v,g)}
\]
converges strictly in $M(B)$,
and its sum $U_g$ is then a partial isometry.

For convenient reference, we record the following computations:
for all $g\in G,(\alpha,h)\in \DD$,
\begin{align*}
U_gR_{(\alpha,h)}
&=\sum_{v\in \CC^0}R_{(v,g)}R_{(\alpha,h)}
\\&=R_{(g\alpha,\varphi(g,\alpha)h)}\righttext{(need $v=r(g\alpha)$)}
\\
R_{(\alpha,h)}U_g
&=\sum_{v\in \CC^0}R_{(\alpha,h)}R_{(v,g)}
\\&=R_{(\alpha,\varphi(h,h\inv s(\alpha))g)}\righttext{(need $v=h\inv s(\alpha)$)}
\\&=R_{(\alpha,hg)}\righttext{(since $h\inv s(\alpha)\in \CC^0$).}
\end{align*}
Note that in the first of these, in the particular case that $(\alpha,h)=(v,1)\in \DD^0$ we have
\[U_gR_{(v,1)}=R_{(gv,g)}
\]
since $\varphi(g,v)=g$ because $v\in \CC^0$.

Next we show that $g\mapsto U_g$ is a homomorphism.
Let $g,h\in G$.
It suffices to check
that the equation
$U_gU_h=U_{gh}$
holds in $M(B)$
after multiplying on the right by an element of the form $R_{(v,1)}b$ for $v\in \CC^0$ and $b\in B$:
\begin{align*}
U_gU_hR_{(v,1)}b
&=U_gR_{(hv,h)}b
=R_{(ghv,gh)}b
=U_{gh}R_{(v,1)}b.
\end{align*}
Since
the representation $R\:\CC\to B$ is nondegenerate,
we have
\begin{align*}
U_1=\sum_{v\in \CC^0}R_{(v,1)}=\sum_{(v,1)\in \DD^0}R_{(v,1)}=1.
\end{align*}
It follows that
each $U_g$ is an invertible partial isometry, and hence is a unitary.
Thus $U\:G\to M(B)$ is a unitary 
homomorphism.

Next we show that $T\:\CC\to B$ is a nondegenerate representation.
For \ref{spi 1} in \defnref{spi rep},
if $\alpha\in\CC$ then
\begin{align*}
T_\alpha^*T_\alpha
&=R_{(\alpha,1)}^*R_{(\alpha,1)}
\\&=R_{s(\alpha,1)}
\\&=R_{(s(\alpha),1)}
\\&=T_{s(\alpha)}.
\end{align*}
For \ref{spi 2} in \defnref{spi rep},
if $\alpha,\beta\in\CC$ with $s(\alpha)=r(\beta)$ then
\[
s(\alpha,1)=(s(\alpha),1)=(r(\beta),1)=r(\beta,1),
\]
so
\begin{align*}
T_\alpha T_\beta
&=R_{(\alpha,1)}R_{(\beta,1)}
\\&=R_{(\alpha,1)(\beta,1)}
\\&=R_{(\alpha\beta,1)}
\\&=T_{\alpha\beta}.
\end{align*}
For \ref{spi 3} in \defnref{spi rep},
if $\alpha,\beta\in\CC$ then
\begin{align*}
T_\alpha T_\alpha^* T_\beta T_\beta^*
&=R_{(\alpha,1)}R_{(\alpha,1)}^*R_{(\beta,1)}R_{(\beta,1)}^*
\\&=\bigvee_{(\gamma,h)\in (\alpha,1)\vee (\beta,1)}R_{(\gamma,h)}R_{(\gamma,h)}^*
\\&=\bigvee_{(\gamma,1)\in (\alpha\vee \beta)\times \{1\}}R_{(\gamma,1)}R_{(\gamma,1)}^*
\\&=\bigvee_{\gamma\in \alpha\vee\beta}T_\gamma T_\gamma^*.
\end{align*}
Next we show that $T$ is nondegenerate,
i.e., that the series $\sum_{v\in \CC^0}T_v$ of projections converges strictly to 1 in $M(B)$.
But this follows immediately from nondegeneracy of $R$, since
\[
\DD^0=\CC^0\times \{1\}.
\]

Thus we have a pair $(T,U)$,
where $T\:\CC\to B$ is a nondegenerate representation and $U\:G\to M(B)$ is a unitary representation.
To complete the verification that $(T,U)$ is a representation of $(\CC,G,\varphi)$,
let $g\in G$ and $\alpha\in\CC$. Then
\begin{align*}
U_gT_\alpha
&=U_gR_{(\alpha,1)}
\\&=R_{(g\alpha,\varphi(g,\alpha))}
\\&=R_{(g\alpha,1)}U_{\varphi(g,\alpha)}
\\&=T_{g\alpha}U_{\varphi(g,\alpha)}.
\end{align*}

Now we need to verify that we have a bijective correspondence between representations of $(\CC,G,\varphi)$ and nondegenerate representations of $\DD$.
Starting with a representation $(T,U)$ of $(\CC,G,\varphi)$ in $B$,
and letting $R$ be the associated nondegenerate representation of $\DD$ in $B$,
and then in turn letting $(T',U')$ be the associated representation of $(\CC,G,\varphi)$,
for $\alpha\in \CC$ we have
\begin{align*}
T'_\alpha
&=R_{(\alpha,1)}
=T_\alpha U_1
=T_\alpha,
\end{align*}
and for $g\in G$ we have
\begin{align*}
U'_g
&=\sum_{v\in \CC^0}R_{(v,g)}
\\&=\sum_{v\in \CC^0}T_vU_g
\\&=1_{M(B)}U_g
\\&=U_g.
\end{align*}
Finally, starting with a nondegenerate representation $R$ of $\DD$ in $B$,
and letting $(T,U)$ be the associated representation of $(\CC,G,\varphi)$,
and then in turn letting $R'$ be the associated nondegenerate representation of $\DD$,
for $(\alpha,g)\in\DD$ we have
\[
R'_{(\alpha,g)}
=T_\alpha U_g
=R_{(\alpha,1)}U_g
=R_{(\alpha,g)}.
\qedhere
\]
\end{proof}
\begin{rem}\label{TTDD-prop}
Consider a category system  $(\CC,G,\varphi)$ with $\CC$ finitely aligned and $\DD = \CC\rtimes^\varphi G$. 
Let $r^\DD\:\DD \to \talg\DD$ denote the universal representation of $\DD$ and let $(t^\DD, u^\DD)$ denote the representation of $(\CC,G,\varphi)$ in $\talg\DD$ satisfying $r^\DD = t^\DD \times u^\DD$, whose existence is guaranteed by Theorem \ref{integrated form}. Thus we have $r^\DD_{(\alpha, g)} =  t^\DD_\alpha u^\DD_g$ for all $(\alpha, g)\in \DD$.

a) We first remark that the triple $(\TT(\DD), t^\DD, u^\DD)$ is universal for representations of $(\CC,G,\varphi)$ in $C^*$-algebras: 
Suppose  $(T,U)$ is a representation of $(\CC,G,\varphi)$ in a $C^*$-algebra $B$. According to Theorem \ref{integrated form}, $T\times U$ is a nondegenerate representation of $\DD$ in $B$, so Lemma \ref{nd phi_T} gives that the associated homomorphism $\phi_{T\times U}\:\talg\DD \to B$ is nondegenerate. Moreover, we have
\begin{equation}\label{phi cross}
\phi_{T\times U}\circ t^\DD= T \righttext{and} \quad \phi_{T\times U}\circ u^\DD= U,
\end{equation}
and $\phi_{T\times U}$ is uniquely determined by these properties.
Indeed, 
for $\alpha \in \CC$ and $g\in G$ we have
\[\phi_{T\times U} \circ t^\DD_\alpha = \phi_{T\times U}(r^\DD_{(\alpha, 1)}) = (T\times U)_{(\alpha,1)}= T_\alpha \quad \text{and}\]
\begin{align*} 
\phi_{T\times U} \circ u_g^\DD &= \phi_{T\times U}\Big(\sum_{v\in \CC^0}r^\DD_{(v,g)}\Big) = \sum_{v\in \CC^0}\phi_{T\times U}(r^\DD_{(v,g)})\\
&=  \sum_{v\in \CC^0} (T\times U)_{(v,g)} =  \sum_{v\in \CC^0} T_v U_g = 1_{M(B)}U_g = U_g.
\end{align*}
Moreover, if $\phi\: \talg\DD \to B$ is also a homomorphism satisfying (\ref{phi cross}), then it is clear that $\phi$ agrees with  $\phi_{T\times U}$ on the range of $r^\DD$. As this range generates $\TT(\DD)$ as a $C^*$-algebra, we get that $\phi =\phi_{T\times U}$.

b) Next, we remark that $\TT(\DD)$ may be described as a $C^*$-blend of $\talg \CC$ and $C^*(G)$ in the sense of \cite{exelblend}:
Set $i := \phi_{t^\DD}\: \TT(\CC) \to \TT(\DD)$, and let $j\: C^*(G)\to M(\TT(\DD))$ denote the integrated form of the unitary representation $u^\DD$ of $G$ in $M(\TT(\DD))$. Both $i$ and $j$ are homomorphisms. Moreover,  let $i\ast j $ denote the linear map defined on the algebraic tensor product $\TT(\CC) \odot C^*(G)$, taking values in $M(\TT(\DD))$ and satisfying 
\[ (i\ast j) (a\otimes b) = i(a) j(b)\]
for all $a \in \TT(\CC)$ and $b \in C^*(G)$. Then the range of $i\ast j$ is contained and dense in $\TT(\DD)$. In fact, this range is equal to $\TT(\DD)$.

Indeed,  letting $g \mapsto \iota_g$ denote the canonical embedding of $G$ in $C^*(G)$,  we have 
\[ (i\ast j)(t_\alpha \otimes \iota_g) = \phi_{t^\DD}(t_\alpha) j(\iota_g) = t^\DD_\alpha u^\DD_g = r^\DD_{(\alpha, g)}\]
for all $\alpha \in \CC, g\in G$. It immediately follows that the range of $i\ast j$ is equal to $C^*(r^\DD)= \TT(\DD)$. This shows that the quintuple \[(\TT(\CC), C^*(G), i , j, \TT(\DD))\] is a $C^*$-blend in the sense of \cite{exelblend}.

c) Assume $\varphi(g,\alpha)=g$ for all $g\in G,\alpha\in\CC$, so $G\curvearrowright \CC$ by automorphisms of the category (cf.~Example \ref{trivialcocy}). We continue to use the notation introduced in a) and b). It follows readily from the universality of $(\talg \CC, t)$ 
that there exists an action $\beta$ 
of $G$ on $\talg \CC$ 
such that  
\[\beta_g(t_\alpha) = t_{g\alpha} 
\] for all $g \in G, \alpha \in \CC$.  
Moreover, Theorem \ref{integrated form} gives that $\TT(\DD) $ is isomorphic to the full $C^*$-crossed product $\talg \CC \rtimes_\beta G$. This can be formulated more precisely by saying that the triple $(\TT(\DD), i , u^\DD)$ is the full crossed product in the sense of \cite{rae:full} of the $C^*$-dynamical system $(\TT(\CC), G, \beta)$.
Indeed, note first that for all $\alpha \in \CC, g\in G$ we have 
\[i(\beta_g(t_\alpha)) = i (t_{g\alpha}) = t^\DD_{g\alpha} = u^\DD_g\, t^\DD_\alpha \,(u^\DD_g)^* = u^\DD_g\, i( t_\alpha)\, (u^\DD_g)^*.\]
This clearly implies that $(i, u^\DD)$ is a covariant homomorphism of the system $(\TT(\CC), G, \beta)$ in $M(\TT(\DD))$. Moreover, we have 
\[i(t_\alpha)j (\iota_g) = t^\DD_\alpha \, u^\DD_g =  r^\DD_{(\alpha, g)}\] for all  $\alpha \in \CC, g\in G$. Thus we see that the  span of $i(\TT(\CC)) j(C^*(G))$ is dense in $\TT(\DD)$.
Finally, if $(\pi, U) $ is a covariant representation of $(\TT(\CC), G, \beta)$ into $M(B)$, we can then let $T $ be the representation of $\CC$ given by $T= \pi\circ t $. Then $(T, U)$ is a representation of $(\TT(\CC), G, \varphi)$ in $M(D)$, and if we set $\phi_{\pi, U}:=\phi_{T\times U}\:\TT(\DD) \to M(D)$, we have 
\[ (\phi_{\pi, U} \circ i) \circ t = \phi_{T\times U} \circ t^\DD = T =  \pi\circ t, \]
hence $\phi_{\pi, U} \circ i = \pi$, and 
\[\phi_{\pi, U} \circ u^\DD = \phi_{T\times U}\circ u^\DD = U.\]
In particular, it follows that $i \times u^\DD$ is an isomorphism from $\talg \CC \rtimes_\beta G$ onto $\TT(\DD)$. 

\end{rem}

\begin{defn}\label{int cp}
A representation $(T,U)$ of a category system $(\CC,G,\varphi)$ in a $C^*$-algebra $B$ is \emph{covariant} if the representation $T\:\CC\to B$ is covariant in the sense of \defnref{spi rep}.
\end{defn}

\begin{cor}\label{integrated CP}
In \thmref{integrated form}, the representation $(T,U)$ of the category system $(\CC,G,\varphi)$ is covariant in the sense of \defnref{int cp} if and only if the representation $R=T\times U$ of $\DD$ is covariant in the sense of \defnref{spi rep}.
\end{cor}

\begin{proof}
First suppose that $(T,U)$ is covariant.
Let $(v,1)\in \DD^0$, and let $F \subseteq (v,1)\DD$ be a finite exhaustive set at $(v,1)$.
By \lemref{exhaust inv} and \lemref{exhaust preserved}, 
we may assume without loss of generality 
that  $F=H\times \{1\}$ for a finite subset $H\subset v\CC$
which is exhaustive at $v$.
Then by \lemref{cp preserved},
\begin{align*}
\bigvee_{(\alpha,1)\in F} R_{(\alpha,1)}R_{(\alpha,1)}^*
&=\bigvee_{\alpha\in H}T_\alpha T_\alpha^*
\\&=T_v
\\&=R_{(v,1)}
\end{align*}
and we have verified that $R$ is covariant.

Conversely, assume that $R$ is covariant.
Let $v\in \CC^0$, and let $H$ be a finite exhaustive set at $v$.
Let $F=H\times \{1\} \subseteq (v,1)\DD$, which by \lemref{exhaust preserved} is a finite exhaustive set at $(v,1)$.
Thus
\begin{align*}
T_v
&=R_{(v,1)}
\\&=\bigvee_{(\alpha,1)\in F}R_{(\alpha,1)}R_{(\alpha,1)}^*
\\&=\bigvee_{\alpha\in H}T_\alpha T_\alpha^*,
\end{align*}
and we have verified that $T$ is covariant,
and hence $(T,U)$ is covariant.
\end{proof}

\begin{rem} \label{OODD-prop} Consider a category system $(\CC,G,\varphi)$ where $\CC$ is finitely aligned and $\DD= \CC\rtimes^\varphi G$.  
The Cuntz-Krieger algebra $\OO(\DD)$ enjoys properties similar to
those described in  Remark \ref{TTDD-prop}. As these properties may be proven in the same way, now using also Corollary \ref{integrated CP},  we only state these below.

Let ${\tilde r}^\DD\:\DD \to \oalg\DD$ denote the universal covariant representation of $\DD$ and let $({\tilde t}^\DD, {\tilde u}^\DD)$ denote the covariant representation of $(\CC,G,\varphi)$ in $\oalg\DD$ satisfying ${\tilde r}^\DD= {\tilde t}^\DD \times {\tilde u}^\DD$, whose existence is guaranteed by Theorem \ref{integrated form} and Corollary \ref{integrated CP}. 
Then we have:

a) The triple $(\OO(\DD), {\tilde t}^\DD, {\tilde u}^\DD)$ is universal for representations of $(\CC,G,\varphi)$ in $C^*$-algebras. 

b) The quintuple $(\OO(\CC), C^*(G), \tilde i , \tilde j, \OO(\DD))$  is a $C^*$-blend in the sense of \cite{exelblend}, where $\tilde i:= \psi_{{\tilde t}^\DD}\: \OO(\CC) \to \OO(\DD)$ and  $\tilde j\: C^*(G)\to M(\OO(\DD))$ is the integrated form of the unitary representation ${\tilde u}^\DD$ of $G$ in $M(\OO(\DD))$.

c) Assume $\varphi(g,\alpha)=g$ for all $g\in G,\alpha\in\CC$, so $G\curvearrowright \CC$ by automorphisms of $\CC$. Let $\tilde \beta$ denote the action of $G$ on $\oalg \CC$ satisfying 
\[\tilde \beta_g({\tilde t}_\alpha) = {\tilde t}_{g\alpha} 
\] for all $g \in G, \alpha \in \CC$.  Then the triple $(\OO(\DD), \tilde i , {\tilde u}^\DD)$ is the full crossed product in the sense of \cite{rae:full} of the $C^*$-dynamical system $(\OO(\CC), G, \tilde \beta)$. In particular, $\tilde i \times {\tilde u}^\DD$ is an isomorphism from $\oalg \CC \rtimes_{\tilde\beta} G$ onto $\OO(\DD)$.
\end{rem}

In connection with Remark \ref{OODD-prop} c), we note that  certain actions  of $\mathbb{Z}^l$ by automorphisms on a finitely aligned $k$-graph $\Lambda$ are considered in \cite{FPS}. The authors construct a $(k+l)$-graph $\Lambda \rtimes \mathbb{Z}^l$ and show that the Cuntz-Krieger algebra of $\Lambda \rtimes \mathbb{Z}^l$ is isomorphic to the full crossed product of the Cuntz-Krieger algebra of $\Lambda$ by the induced action of $\mathbb{Z}^l$. Since a higher-rank graph is a category of paths we may regard the result in Remark \ref{OODD-prop} c) as a generalization of their result to finitely aligned category systems.

\section{Application to Exel-Pardo systems}\label{EP algebras}

Throughout this section, $(E,G,\varphi)$ denotes a given (discrete) Exel-Pardo system \cite{EP, bkqexelpardo}. 
So $E= (E^0, E^1, r, s)$ is a 
 directed graph \cite{raeburngraph},
$G$ is a discrete group acting on $E$ by automorphisms,
and $\varphi\:G\times E^1\to G$ is a graph cocycle,
i.e., it is a cocycle for the action of $G$ on the set of edges $E^1$ that also satisfies \eqref{source-inv}.

Let $A=c_0(E^0)$,
and let $X$ be the $C^*$-correspondence over $A$ whose Cuntz-Pimsner algebra $\OO_X$ is canonically isomorphic to the graph algebra $C^*(E)$,
as in \cite[Example~8.13]{raeburngraph}.
Let $B=A\rtimes G$ be the crossed product by the action of $G$ on the vertices.
\cite[Section~3]{bkqexelpardo}
introduced a $C^*$-correspondence $\yphi$ over $B$
that is a kind of ``twisted crossed product'' of the graph correspondence $X$ by the action of $G$,
where the cocycle $\varphi$ provides the ``twist''.
The $B$-correspondence $\yphi$
is modelled
on the 
correspondence $M$ of \cite[Section~10]{EP}
(although in \cite{EP} the graph is assumed to be finite and have no sources);
in fact, \cite[Theorem~6.1]{bkqexelpardo} proves that
these correspondences are isomorphic, and hence
the Cuntz-Pimsner algebra $\oy$ is isomorphic to the $C^*$-algebra that Exel and Pardo denoted by $\OO_{G,E}$
(and which we call the \emph{Exel-Pardo algebra}).

Let $E^*$ denote the (finite) path category of $E$.
Then $E^*$ is singly aligned and cancellative,
and has no inverses,
and
in particular
is a category of paths in the sense of \cite{Spi11}.
As we mentioned in \exref{EP ex}, the cocycle extends uniquely to a category cocycle,
which we continue to denote by $\varphi$,
for the action of $G$ on $E^*$, making $(E^*,G,\varphi)$ a category system.

We need to
relate representations of the Exel-Pardo system $(E,G,\varphi)$
to representations of the category system $(E^*,G,\varphi)$.
For this purpose, we want to
extend
\cite[Definition~5.1]{bkqexelpardo},
as we explain below.

First we briefly review the Toeplitz algebra of the directed graph $E$.
For $v\in E^0$ let $\delta_v$ denote the characteristic function of $\{v\}$, so that the commutative $C^*$-algebra $c_0(E^0)$ is generated by the pairwise orthogonal family of projections $\{\delta_v:v\in E^0\}$.
On the other hand, for $e\in E^1$ let $\Chi_e$ denote the characteristic function of $\{e\}$, so that $c_c(E^1)$ is the 
vector space with basis $\{\Chi_e:e\in E^1\}$.
Then $c_c(E^1)$ is a pre-correspondence over $c_0(E^0)$ with operations determined by
\begin{align*}
\Chi_e\delta_v&=\begin{cases}\Chi_e\case v=s(e)\\0\case v\ne s(e)\end{cases}
\\
\delta_v\Chi_e&=\begin{cases}\Chi_e\case v=r(e)\\0\case v\ne r(e)\end{cases}
\\
\<\Chi_e,\Chi_f\>_{c_0(E^0)}&=\begin{cases}\delta_{s(e)}\case e=f\\0\case e\ne f.\end{cases}
\end{align*}
Then the completion of this pre-correspondence is the $C^*$-correspondence $X(E)$ over $c_0(E^0)$.
A \emph{Toeplitz representation}
(or just a \emph{representation}; see \cite{katsuracorrespondence})
of $X(E)$ in a $C^*$-algebra $B$ is a pair $(\psi,\pi)$, where $\psi\:X(E)\to B$ is a linear map and $\pi\:c_0(E^0)$ is a homomorphism such that for $\xi,\eta\in X(E)$ and $a\in c_0(E^0)$,
\begin{align*}
\psi(\xi a)&=\psi(\xi)\pi(a)\\
\psi(a\xi)&=\pi(a)\psi(\xi)\\
\pi(\<\xi,\eta\>_{c_0(E^0)})&=\psi(\xi)^*\psi(\eta).
\end{align*}
A \emph{Toeplitz $E$-family}
(or a \emph{Toeplitz-Cuntz-Krieger $E$-family}; see \cite{simscouniversal})
in $B$ is a pair $(P,S)$, where $P\:E^0\to B$ and $S\:E^1\to B$ are maps such that
$\{P_v\}_{v\in E^0}$ is a family of pairwise orthogonal projections,
$S_e^*S_e=P_{s(e)}$ for all $e\in E^1$,
and $\sum_{r(e)=v}S_eS_e^*\le P_v$ for all $v\in E^0$ 
such that  $r^{-1}(\{v\})$ is a finite nonempty subset of $E^1$. 

The Toeplitz representations $(\psi,\pi)$ of $X(E)$ are in bijective correspondence with the Toeplitz $E$-families $(P,S)$
in $B$ via $P_v=\pi(\delta_v)$ and $S_e=\psi(\Chi_e)$,
and also are in bijective correspondence with
the representations $T$ of $E^*$ in $B$
via $P=T|_{E^0}$ and $S=T|_{E^1}$.
Thus the Toeplitz algebra $\talg{E^*}$ of the path category $E^*$,
which by definition is generated by a universal representation of $E^*$,
can be characterized as
the $C^*$-algebra, often denoted by $\TT_{X(E)}$, generated by a universal Toeplitz representation of the $c_0(E^0)$-correspondence $X(E)$. It can also be characterized as
the $C^*$-algebra, sometimes denoted by $\TT(E)$,
generated by a universal Toeplitz $E$-family.
$\TT_{X(E)}$ is often called the \emph{Toeplitz algebra} of the correspondence $X(E)$,
and $\talg E$ is sometimes called the \emph{Toeplitz-Cuntz-Krieger algebra} of the graph $E$.

According to \defnref{nd},
a representation
$T$ of $E^*$ in $B$ is called nondegenerate if the series $\sum_{v\in E^0}T_v$ converges strictly to $1_{M(B)}$,
which by \lemref{nd phi_T} is equivalent to nondegeneracy of the associated homomorphism $\phi_T\:\TT(E^*)\to B$.

\begin{defn}\label{nd E}
We call a Toeplitz $E$-family $(P,S)$ in $B$ \emph{nondegenerate} if the series $\sum_{v\in E^0}P_v$ converges strictly to $1_{M(B)}$.
\end{defn}

It follows from the above discussion that a representation $T$ of $E^*$ is nondegenerate in the sense of \defnref{nd} if and only if the associated Toeplitz $E$-family $(P,S)$ is nondegenerate in the sense of \defnref{nd E}.

Returning to our Exel-Pardo system $(E,G,\varphi)$,
we extend \cite[Definition~5.1]{bkqexelpardo} as follows:

\begin{defn}\label{rep EP}
A \emph{representation} of 
$(E,G,\varphi)$ in a $C^*$-algebra $B$ is a triple $(P,S,U)$,
where $(P,S)$ is a nondegenerate Toeplitz $E$-family in $B$,
$U\:G\to M(B)$ is a unitary homomorphism,
and for all $g\in G$, $v\in E^0$, and $e\in E^1$ we have
\begin{equation}\label{rep EP conditions}
\begin{split}
U_gP_v&=P_{gv}U_g\\
U_gS_e&=S_{ge}U_{\varphi(g,e)}.
\end{split}
\end{equation}
We define the $C^*$-algebra $C^*(P,S,U)$ associated to $(P,S,U)$ to be the $C^*$-subalgebra of $B$ generated by the set 
\[
\{P_vU_g:v\in E^0,g\in G\}\cup \{S_eU_g:e\in E^1,g\in G\}.
\]
\end{defn}
In \cite[Definition~5.1]{bkqexelpardo} the nondegeneracy of $(P,S)$ is replaced by the condition that $C^*(P,S,U) =B$. 
But we now think that this latter condition is too strong,
as indicated by the following lemma:

\begin{lem}\label{EP iv}
Let
$\DD=E^*\rtimes^\varphi G$ be the Zappa-Sz\'ep product
of the 
category system $(E^*,G,\varphi)$,
and let $R$ be a representation of $\DD$ in a $C^*$-algebra $B$.
Let $(T,U)$ be the representation of $(E^*,G,\varphi)$ associated to $R$,
and let $(P,S)$ be the Toeplitz $E$-family in $B$ associated to the representation $T$ of $E^*$.
Then the associated homomorphism $\phi_R\:\talg \DD\to B$ 
is surjective
if and only if $C^*(P,S,U) =B$.
\end{lem}

\begin{proof}
Note that $\phi_R(\talg \DD)$ is generated as a $C^*$-algebra by
\[
\{R_{(\alpha,g)}:(\alpha,g)\in \DD\},
\]
and that
\[
R_{(\alpha,g)}=T_\alpha U_g.
\]
By definition of the path category $E^*$,
for every $\alpha\in E^*$ the element $T_\alpha$ is a finite product of elements of the form $P_v$ and $S_e$ for $v\in E^0$ and $e\in E^1$.
Thus $\phi_R(\talg \DD)$ coincides with the $C^*$-subalgebra of $B$ generated by
\[
\{P_vU_g:v\in E^0,g\in G\}\cup \{S_eU_g:e\in E^1,g\in G\},
\]
and the lemma follows.
\end{proof}

The
following theorem
is 
analogous to \cite[Theorem~5.2]{bkqexelpardo},
but for representations of $(E,G,\varphi)$
in the sense of \defnref{rep EP}
rather than the sense of \cite[Definition~5.1]{bkqexelpardo}.

\begin{thm}\label{bkqgenrel}
There is a bijective correspondence between
representations $(P,S,U)$ of $(E,G,\varphi)$ in a $C^*$-algebra $B$
and nondegenerate representations $R$ of $\DD$ in $B$,
given as follows:
\begin{enumerate}
\item
If $(P,S,U)$ is a representation of $(E,G,\varphi)$ in $B$,
then the Toeplitz $E$-family $(P,S)$ in $B$
determines a nondegenerate representation $T\:E^*\to B$,
and $(T,U)$ is a representation of the category system $(E^*,G,\varphi)$ in $B$,
which determines a nondegenerate representation $R= T\times U\:\DD\to B$.

\item
If $R\:\DD\to B$ is a nondegenerate representation,
then letting $(T,U)$ be the associated representation of $(E^*,G,\varphi)$ in $B$ as in \thmref{integrated form},
we let $(P,S)$ be the Toeplitz $E$-family in $B$ associated to $T$,
and then $(P,S,U)$ is a representation of $(E,G,\varphi)$ in $B$.
\end{enumerate}
\end{thm}

\begin{proof}
By \defnref{rep EP}, a representation $(P,S,U)$ of $(E,G,\varphi)$ in $B$
is a nondegenerate Toeplitz $E$-family $(P,S)$ in $B$
and a unitary homomorphism
$U\:G\to M(B)$
that interacts with $(P,S)$ as in \eqref{rep EP conditions}.
Letting $T\:E^*\to B$ be the associated representation of the path category $E^*$,
it is routine to check that the conditions \eqref{rep EP conditions} imply that $(T,U)$ is a representation of $(E^*,G,\varphi)$ in the sense of \defnref{covariant},
which by \thmref{integrated form} determines a unique nondegenerate representation $R= T\times U$ of $\DD$ in $B$ such that
\[
R_{(\alpha,g)}=T_\alpha U_g \righttext{for every }(\alpha,g)\in\DD.
\]

Conversely, suppose that $R\:\DD\to B$ is a nondegenerate representation,
let $(T,U)$ be the associated representation of $(E^*,G,\varphi)$ in $B$ as in \thmref{integrated form},
and then let $(P,S)$ be the Toeplitz $E$-family in $B$ associated to the nondegenerate representation $T\:E^*\to B$.
It follows from \defnref{covariant} that $(P,S,U)$ is a representation of $(E,G,\varphi)$ as in \defnref{rep EP}.

Finally, the bijective correspondence between nondegenerate representations $T\:E^*\to B$ and Toeplitz $E$-families $(P,S)$ in $B$
obviously gives a bijective correspondence between representations $(T,U)$ of $(E^*,G,\varphi)$ and representations $(P,S,U)$ of $(E,G,\varphi)$.
Combining this with the bijective correspondence between $(T,U)$ and $R$ finishes the proof.
\end{proof}

As an immediate consequence of Theorem \ref{bkqgenrel} and \cite[Theorem~5.2]{bkqexelpardo}, we get:
\begin{cor}
Let $(p,s,u)$ denote the representation of $(E,G,\varphi)$ associated to the universal representation of $\DD= E^*\rtimes^\varphi G$ in its 
 Toeplitz algebra $\talg \DD$. Then $\talg \DD = C^*(p,s,u)$, and the pair  $\big(\talg \DD, (p,s, u)\big)$ is universal for representations of 
 $(E,G,\varphi)$ in $C^*$-algebras in the following sense: 
 
 Given a representation $(P,S,U)$ of  $(E,G,\varphi)$ in $B$, there exists a unique nondegenerate homomorphism $\phi_{P, S, U}\: \talg \DD \to B$ such that 
 \[ P=\phi_{P, S, U} \circ p\,, \righttext{} S= \phi_{P, S, U} \circ s\,,\righttext{and} \, \, \phi_{P, S, U}\circ u = U.\] 
In particular, $\talg\DD$ is isomorphic to the Toeplitz algebra $\TT_{Y^\varphi}$ of the $C^*$-correspondence $Y^\varphi$ associated to $(E, G, \varphi)$ in \cite{bkqexelpardo}.
\end{cor}

Now we want to pass to the Cuntz-Krieger algebra $\oalg \DD$,
and we pause to briefly review the Cuntz-Krieger algebra of the directed graph $E$, see~\cite{raeburngraph} for more details. 

\emph{We assume from now on that $E$ is row-finite, i.e., $r^{-1}(\{v\})$ is a finite subset of $E^1$ for every vertex $v \in E^0$.}
Let $A=c_0(E^0)$, and let $X=X(E)$ be the graph correspondence, which is a $C^*$-correspondence over $A$, and identify the Toeplitz-Cuntz-Krieger algebra of $E$ with the Toeplitz algebra $\talg{E^*}$ of the path category $E^*$.
A Toeplitz $E$-family $(P,S)$ in a $C^*$-algebra $B$
satisfies the \emph{Cuntz-Krieger condition},
and $(P,S)$ is called a \emph{Cuntz-Krieger $E$-family},
if
for every vertex $v$ that is not a source (i.e., $r^{-1}(\{v\}) \neq \varnothing$)
\[
P_v=\sum_{r(e)=v}S_eS_e^*.
\]
The \emph{Katsura ideal} of $c_0(E^0)$ is
\[
J_X=\clspn\{\delta_v:v\in r(E^1)\},
\]
and is the largest ideal of $c_0(E^0)$ on which the left-module homomorphism
\[
\phi\:A\to \KK(X)
\]
is injective.
(Note that $\phi(A)\subset \KK(X)$ since $E$ is row-finite.)

Every Toeplitz representation $(\psi,\pi)$ of $X$ in $B$
uniquely determines a homomorphism $\psi^{(1)}\:\KK(X)\to B$ by
\[
\psi^{(1)}(\theta_{x,y})=\psi(x)\psi(y)^*,
\]
where $\theta_{x,y}$ is the (generalized) rank-one operator on $X$ given by
\[\theta_{x,y}(z)=x\,\<y,z\>_A\]
for $x, y, z \in X$.
A Toeplitz representation $(\psi,\pi)$ of $X$ in $B$ is \emph{Cuntz-Pimsner covariant} if
\[
\pi(f)=\psi^{(1)}\circ\phi(f)\righttext{for all}f\in J_X.
\]
The Toeplitz representation $(\psi,\pi)$ of $X$ associated to a Toeplitz $E$-family $(P,S)$
is Cuntz-Pimsner covariant if and only if $(P,S)$ satisfies the Cuntz-Krieger condition,
equivalently the associated representation $T$ of the path category $E^*$ is covariant in the sense of \defnref{spi rep}.
Thus the Cuntz-Krieger algebra $\oalg{E^*}$ of the path category $E^*$,
which by definition is generated by a universal covariant representation of $E^*$,
can be characterized as the $C^*$-algebra, denoted by $\OO_X$,
generated by a universal Cuntz-Pimsner covariant representation of the $A$-correspondence $X$,
and can also be characterized as the $C^*$-algebra, often denoted by $C^*(E)$,
generated by a universal Cuntz-Krieger $E$-family.
$\OO_X$ is called the \emph{Cuntz-Pimsner algebra} of the correspondence $X$,
and $C^*(E)$ is usually just called the $C^*$-algebra of the graph $E$, or the \emph{graph algebra} of $E$.

Now we return to our Exel-Pardo system $(E,G,\varphi)$.

\begin{defn}\label{PSU cov}
Assume $E$ is row-finite.
A representation $(P,S,U)$ of $(E,G,\varphi)$ in a $C^*$-algebra $B$ is \emph{covariant}
if $(P,S)$ is a Cuntz-Krieger $E$-family.
\end{defn}

The
following corollary is 
analogous to
 \cite[Corollary~5.4]{bkqexelpardo},
but for representations of $(E,G,\varphi)$
in the sense of \defnref{rep EP}
rather than the sense of \cite[Definition~5.1]{bkqexelpardo}.

\begin{cor}\label{bkqgenrelcp}
Assume $E$ is row-finite.
In \thmref{bkqgenrel}, the representation $R$ of $\DD$
is covariant if and only if
the representation
$(P,S,U)$ of $(E,G,\varphi)$ is covariant.
\end{cor}

\begin{proof}
Let $(T,U)$ be the representation of the category system $(E^*,G,\varphi)$ associated to $R$ as in \thmref{integrated form}.
By \corref{integrated CP}, 
$R$ is covariant
if and only if 
$(T,U)$ is covariant,
meaning, by \defnref{int cp}, that $T$ is covariant,
which we mentioned above is equivalent to
$(P,S)$ being a Cuntz-Krieger $E$-family.
On the other hand, by \defnref{PSU cov}, to say that $(P,S,U)$ is covariant means that $(P,S)$ is a Cuntz-Krieger $E$-family.
Thus the current corollary is just a combination of these results and definitions.
\end{proof}
It follows now readily from Corollary \ref{bkqgenrelcp} and  \cite[Corollary~5.4]{bkqexelpardo} that we have:

\begin{cor} \label{ZZ-OO-EP}
Assume $E$ is row-finite.
Let $(\tilde p,\tilde s, \tilde u)$ denote the covariant representation of the system $(E,G,\varphi)$ associated to the universal covariant representation of $\DD= E^*\rtimes^\varphi G$ in $\oalg \DD$. Then $\oalg \DD = C^*(\tilde p,\tilde s, \tilde u)$, and the pair  $\big(\oalg \DD, (\tilde p, \tilde s, \tilde u)\big)$ is universal for covariant representations of 
 $(E,G,\varphi)$ in $C^*$-algebras in the following sense: 

 Given a covariant representation $(P,S,U)$ of  $(E,G,\varphi)$ in $B$, there exists a unique nondegenerate homomorphism $\psi_{P, S, U}\: \oalg \DD \to B$ such that 
 \[ P=\psi_{P, S, U} \circ \tilde p\,, \righttext{} S= \psi_{P, S, U} \circ \tilde s\,,\righttext{and} \, \, \psi_{P, S, U}\circ \tilde u = U.\] 
In particular, $\oalg\DD$ is isomorphic to the Exel-Pardo algebra $\OO_{Y^\varphi}$ associated to $(E, G, \varphi)$ in \cite{bkqexelpardo}.
 \end{cor}
 
\begin{rem} \label{kats} As observed by Exel and Pardo in \cite[Example 3.4]{EP}, any Katsura algebra \cite{KatKirchberg} can be written as the Exel-Pardo algebra of an Exel-Pardo system of the form $(E, \mathbb{Z}, \varphi)$ with $E$ finite.  Since the class of Katsura algebras coincides with the class of Kirchberg algebras belonging to the UCT class,  it follows from Corollary \ref{ZZ-OO-EP} that any 
Kirchberg algebra in the UCT class is isomorphic to  some   $\mathcal{O}(E^* \rtimes^\varphi \mathbb{Z})$ with $E$ finite.
As $E^* \rtimes^\varphi \mathbb{Z}$ is singly aligned, we can conclude that the class of Cuntz-Krieger algebras associated to singly aligned left cancellative small categories is vast. 
\end{rem}

\begin{rem}
In Corollary \ref{ZZ-OO-EP},
if we try to remove the row-finiteness assumption then we run up against a special case of the Hao-Ng problem 
(see e.g.\ \cite{HaoNg, bkqexelpardo, Katsou}),
as we now explain. First, by the \emph{Hao-Ng problem}, we mean whether, given an action of a locally compact group $G$ on a 
nondegenerate
$C^*$-correspondence $X$ over a $C^*$-algebra $A$,
$\OO_{X\rtimes G}$ is naturally isomorphic to $\OO_X\rtimes G$. 
 (There is also a reduced version of this problem.)
Hao and Ng give a positive answer when $G$ is amenable \cite[Theorem~2.10]{HaoNg},
and in \cite[Theorem~5.5]{bkqexelpardo} the isomorphism is proved if $G$ is discrete and the action of $G$ on $A$ has Exel's Approximation Property.  
Now suppose that we are given an action of a discrete group $G$ on a directed graph $E$.
Then we have an Exel-Pardo system $(E,G,\varphi)$ with trivial cocycle $\varphi$, i.e., $\varphi(e,g)=g$ for all $(e,g)\in E^1\times G$.
Let $X$ be the graph correspondence over $c_0(E^0)$,
so that $C^*(E)\simeq \OO_X$.
Then the crossed product correspondence $X\rtimes G$ is the correspondence $Y^\varphi$ over $c_0(E^0)\rtimes G$, for the trivial cocycle $\varphi$ (see \cite{bkqexelpardo} for the notation),
and so
$\OO_{X\rtimes G}=\OO_{Y^\varphi}$.
On the other hand, the crossed product $\OO_X\rtimes G$ is universal for
covariant representations
$(P,S,U)$ of $(E,G,\varphi)$.
Thus, the conclusion of
Corollary \ref{ZZ-OO-EP}
holds in the case where $\varphi$ is trivial
if and only if the associated special case of the Hao-Ng problem
$\OO_{X\rtimes G}\simeq \OO_X\rtimes G$ has a positive answer.
\end{rem}

\section{Relationships with other constructions} \label{final}

Throughout this section $\CC$ denotes a left cancellative small category.

\subsection*{The regular Toeplitz algebra of $\CC$.}
We introduce in this subsection 
 the \emph{regular representation} $V$ of $\CC$ on $\ell^2(\CC)$ and define
 the \emph{regular Toeplitz algebra} $\TT_\ell(\CC)$ as the $C^*$-subalgebra of $\mathcal{B}(\ell^2(\CC))$ generated by $V$. 
 (See also \cite[Section 11]{jackpaths2}.)
Whenever $X$ is a subset of $\CC$, we identify $\ell^2(X)$ as a closed subspace of $\ell^2(\CC)$ in the canonical way and let $P_X$ denote the orthogonal projection from $\ell^2(\CC)$ onto $\ell^2(X)$. In particular,  we have $\ell^2(\varnothing) = \{0\}$ and $P_\varnothing = 0$. Moreover, we denote by ${\rm id}_X$ the identity map from $X$ into itself. In particular, ${\rm id}_{\varnothing}$ is the empty function having the empty set $\varnothing$ as its domain and its range. 
 
 For each $\alpha \in \CC $ we let $\tau_\alpha\: s(\alpha)\CC \to \alpha\CC$ denote  the (right) shift map given by 
\[\tau_\alpha(\beta) = \alpha\beta \righttext{for all } \beta \in s(\alpha)\CC.\] 
Since $\CC$ is left cancellative, each $\tau_\alpha$ is a bijection from $s(\alpha)\CC$  onto $\alpha\CC$, with inverse  $\sigma_\alpha\:\alpha\CC \to s(\alpha)\CC$ given by 
\[\sigma_\alpha(\alpha\beta)=\beta \righttext{ for all} \beta \in s(\alpha)\CC.\]

\begin{defn} For each $\alpha \in \CC$ let $V_\alpha\: \ell^2(\CC) \to \ell^2(\CC)$ be the linear contraction given by 
\[(V_\alpha\xi)(\gamma) = \
\begin{cases} 
\xi\big(\sigma_{\alpha}(\gamma)\big) & \righttext{if $\gamma \in \alpha\CC$,}\\ 
\hspace{4ex} 0 & \righttext{otherwise.} 
\end{cases}\]
\end{defn}

It is straightforward to check that each $V_\alpha^*$ is given by  
\[(V_\alpha^*\eta)(\gamma) = \
\begin{cases} 
\eta\big(\tau_{\alpha}(\gamma)\big) & \righttext{if $\gamma \in s(\alpha)\CC$,}\\ 
\hspace{4ex} 0 & \righttext{otherwise.} 
\end{cases}\]
It follows that each $V_\alpha$ is a partial isometry with initial space $\ell^2\big(s(\alpha)\CC\big)$ and final space $\ell^2\big(\alpha\CC\big)$.  In other words, we have $V_\alpha^*V_\alpha = P_{s(\alpha)\CC}$ and $V_\alpha V^*_\alpha = P_{\alpha\CC}$. Note also that if $v \in \CC^0$, then $V_v = P_{v\CC}$. Thus we get that $V_\alpha^*V_\alpha= V_{s(\alpha)}$.

Moreover, letting $\{\delta_\gamma\}_{\gamma \in \CC}$ denote the canonical basis of $\ell^2(\CC)$, we have
\[V_\alpha\delta_\gamma = \
\begin{cases} 
\delta_{\alpha\gamma}& \righttext{if $\gamma \in s(\alpha)\CC$,}\\ 
\hspace{1ex} 0 & \righttext{otherwise.} 
\end{cases}\]
So if $\beta \in \CC$ is such that $s(\alpha)=r(\beta)$, we get 
\[V_\alpha V_\beta\delta_\gamma = \
\begin{cases} 
V_\alpha\delta_{\beta\gamma}& \righttext{if $\gamma \in s(\beta)\CC$,}\\ 
\hspace{1ex} 0 & \righttext{otherwise} 
\end{cases}
= \begin{cases} 
\delta_{\alpha\beta\gamma}& \righttext{if $\gamma \in s(\beta)\CC$,}\\ 
\hspace{1ex} 0 & \righttext{otherwise} 
\end{cases}
= V_{\alpha\beta}\delta_\gamma\]
for all $\gamma \in \CC$, so we have  $V_\alpha V_\beta = V_{\alpha\beta}$.

We call $V$ the \emph{regular representation} of $\CC$ in $\ell^2(\CC)$ and define the \emph{regular Toeplitz algebra} $\TT_\ell(\CC)$ as the $C^*$-subalgebra of $\mathcal{B}(\ell^2(\CC))$ generated by the set $\{ V_\alpha : \alpha \in \CC\}$.
 
\begin{prop} \label{phiV} Assume that $\CC$ is finitely aligned. Then the map $V\:\CC \to \TT_\ell(\CC)$ sending each $\alpha$ to $V_\alpha$ is a representation of $\CC$ in $\TT_\ell(\CC)$ and the canonical homomorphism $\phi_V$ from the Toeplitz algebra $\TT(\CC)$ into $\TT_\ell(\CC)$ is onto.
\end{prop}
\begin{proof}
To show the first assertion, it only remains to check that $V$ satisfies  property (\ref{spi 3}). Let $\alpha, \beta \in \CC$. Then we get
 \[
 V_\alpha V_\alpha^* V_\beta V_\beta^*= P_{\alpha\CC} P_{\beta\CC} = P_{\alpha\CC \cap \beta\CC}
 = P_{\,\bigcup_{\gamma\in \alpha \vee \beta} \gamma\CC}= \bigvee_{\gamma\in \alpha \vee \beta}P_{\gamma\CC}
 = \bigvee_{\gamma\in \alpha \vee \beta} V_\gamma V^*_\gamma,
 \]
 as desired. Since the universal family $\{t_\alpha: \alpha\in \CC\}$ in $\TT(\CC)$ generates $\TT(\CC)$ as a $C^*$-algebra and we have $V_\alpha =\phi_V(t_\alpha)$ for each $\alpha \in \CC$, the second assertion is immediate from the definition of $\TT_\ell(\CC)$. 
\end{proof}

\subsection*{The inverse semigroup ${\rm ZM}(\CC)$} 

In this subsection we introduce the inverse semigroup ${\rm ZM}(\CC)$ consisting of zigzag maps on $\CC$, describe its semilattice of idempotents and show that ${\rm ZM}(\CC)$ is isomorphic to an inverse semigroup of partial isometries lying in $\TT_\ell(\CC)$. For an introduction to the theory of inverse semigroups, see for example \cite{lawsonbook}.

We first  recall that the symmetric inverse semigroup $I(\CC)$ consists of all partial bijections of $\CC$.  
The product in $I(\CC)$ is the composition of maps defined on the largest possible domain: if $A, B, C, D$ are subsets of $\CC$, $\varphi$ is bijection from $A$ onto $B$, and $\psi$ is bijection from $C$ onto $D$, then $\psi\varphi \in I(\CC)$ is the bijection from $\varphi^{-1}(B\cap C) $ onto $\psi(B\cap C) $ given by $(\psi\varphi)(\gamma) = \psi\big(\varphi(\gamma)\big)$ for each $\gamma \in \varphi^{-1}(B \cap C)$ when $B \cap C \neq \varnothing$,
 while $\psi\varphi = {\rm id}_{\varnothing}$ otherwise. We will often denote the domain of $\varphi\in I(\CC)$ by ${\rm dom}(\varphi)$ and its range by ${\rm ran}(\varphi)$.
The inverse of a partial bijection $\varphi\:A\to B$ in $I(\CC)$ is of course the inverse map $\varphi^{-1}\:B\to A$.  We then have $\varphi^{-1} \varphi = {\rm id}_A$ and $\varphi \varphi^{-1} = {\rm id}_B$. Thus the idempotent semilattice of $I(\CC)$ is $E(I(\CC))=\{{\rm id}_X \mid X \subseteq \CC\}$.

For each $\alpha \in \CC$  both  $\tau_\alpha$ and its inverse $\sigma_\alpha$ belong to $I(\CC)$. One readily sees that for $\alpha, \beta \in \CC $ we have $\tau_\alpha\tau_\beta = \tau_{\alpha\beta}$ and $\sigma_\beta\sigma_\alpha= \sigma_{\alpha\beta}$ if $s(\alpha) = r(\beta)$, while $\tau_\alpha\tau_\beta = {\rm id}_{\varnothing} = \sigma_\beta\sigma_\alpha$ otherwise.
Following \cite{DonMil} (where Donsig and Milan consider the case where $\CC$ is a category of paths in the sense of \cite{Spi11}), we define ${\rm ZM}(\CC)$ to be the inverse subsemigroup of $I(\CC)$ generated by $\{\tau_\alpha\}_{\alpha\in \CC} \cup \{{\rm id}_{\varnothing}\}$. In other words, ${\rm ZM}(\CC)$ is
the subsemigroup of $I(\CC)$ generated by $\{\tau_\alpha, \sigma_\alpha\}_{\alpha\in \CC} \cup \{{\rm id}_{\varnothing}\}$\footnote{Note that if $\CC^0$ consists of more than one element,  and we pick $v, w \in \CC^0$ with $v\neq w$, then we have $\sigma_v\sigma_w = {\rm id}_{ v\CC}\, {\rm id}_{ w\CC} = {\rm id}_{\varnothing}$, so we don't need to specify that ${\rm id}_{\varnothing}$ is included in ${\rm ZM}(\CC)$. 
}. 
 We may describe ${\rm ZM}(\CC)$ by introducing so-called {\it zigzag maps} on $\CC$ 
 (cf.~\cite{jackpaths2}).
By a \emph{zigzag} in $\CC$ we will mean an even tuple of the form \[\zeta = (\alpha_1, \beta_1, \ldots, \alpha_n,\beta_n),\] where $n\in \mathbb{N}$,  $\alpha_i, \beta_i \in \CC$ and $r(\alpha_i) = r(\beta_i)$ for $i =1, \ldots, n$, and $ s(\beta_i)=s(\alpha_{i+1}) $ for $i =1, \ldots, n-1$. Letting $\mathcal{Z}_\CC$ denote the set of all zigags in $\CC$  we define maps $s$ and $r$ from $\mathcal{Z}_\CC$ into $\CC^0$ by $s(\zeta) = s(\beta_n)$ and $r(\zeta) = s(\alpha_1)$ whenever $\zeta\in\mathcal{Z}_\CC$ is as above; we also define the \emph{reverse} of $\zeta$ as
\[\bar\zeta= (\beta_n,\alpha_n, \ldots, \beta_1, \alpha_1) \in \mathcal{Z}_\CC.\]  
To each $\zeta = (\alpha_1, \beta_1, \ldots, \alpha_n,\beta_n) \in \mathcal{Z}_\CC$, we associate the \emph{zigzag map} $\varphi_\zeta$ in ${\rm ZM}(\CC)$ defined by
\[ \varphi_\zeta= \sigma_{\alpha_1}\, \tau_{\beta_1}\, \cdots\, \sigma_{\alpha_n}\,\tau_{\beta_n}.\]
Clearly, we have $\varphi_{\bar\zeta}= \varphi_\zeta^{-1}$. Thus, ${\rm dom}(\varphi_{\bar\zeta})={\rm ran}(\varphi_\zeta)$ and ${\rm ran}(\varphi_{\bar\zeta})={\rm dom}(\varphi_\zeta)$. Note that if $\zeta, \zeta' \in \mathcal{Z}_\CC$ satisfy that $s(\zeta) = r(\zeta')$ and we let $\zeta  \zeta' \in \mathcal{Z}_\CC$ be defined by concatenation in the obvious way,  we have $\varphi_\zeta \varphi_{\zeta'} = \varphi_{\zeta \zeta'} $. On the other hand, if $s(\zeta) \neq r(\zeta')$, we get
$\varphi_\zeta \varphi_{\zeta'} = {\rm id}_\varnothing$. It follows that the set $\mathcal{S}=\{\varphi_\zeta : \zeta \in \mathcal{Z}_\CC\} \cup \{ {\rm id}_\varnothing\}$ is closed under the product and the inverse operation in $I(\CC)$, hence that it is an inverse subsemigroup of ${\rm ZM}(\CC)$.
Now, if  $\alpha, \beta \in \CC$, then $(\alpha, r(\alpha))$ and $(r(\beta), \beta)$ both belong to $\mathcal{Z}_\CC$ and we have
\[ \varphi_{(\alpha, r(\alpha))}= \sigma_\alpha \,, \, \, \varphi_{(r(\beta), \beta)}= \tau_\beta.\]
It follows that ${\rm ZM}(\CC)$ is contained in $\mathcal{S}$. Thus we can conclude that they are equal, that is, we have:
\begin{equation}\label{ZM}
{\rm ZM}(\CC) = \{\varphi_\zeta : \zeta \in \mathcal{Z}_\CC\} \cup \{ {\rm id}_\varnothing\}.
\end{equation} 

Since $\varphi_{\zeta}^{-1}\varphi_\zeta = {\rm id}_{{\rm dom}(\varphi_\zeta)}$ for every $\zeta \in \mathcal{Z}_\CC$, we get from (\ref{ZM}) that  the idempotent semilattice  of ${\rm ZM}(\CC)$ is given by 
\begin{equation} \label{E-ZM}
E(ZM(\CC))=  \{ {\rm id}_{{\rm dom}(\varphi_\zeta)} : \zeta \in \mathcal{Z}_\CC\big\}\cup \{{\rm id}_\varnothing\}.\
\end{equation}
To describe more precisely the domain and the range of a zigzag map, it will be helpful to introduce some notation. Let 
$\alpha \in \CC$ and $X\subseteq \CC$. Then we set
\begin{align*}\alpha X&:=
\tau_\alpha\big(s(\alpha)\CC \cap X\big),\\
\alpha^{-1} X&:= 
\sigma_\alpha\big(\alpha\CC \cap X\big).
\end{align*}
Note that $\alpha\CC$ has the same meaning as before, and that $\alpha^{-1}\CC = s(\alpha)\CC$.  
Moreover, we have $\alpha\alpha^{-1}X =  \alpha\CC \cap X$ and $\alpha^{-1} \alpha X= s(\alpha)\CC  \cap X$. Let $\alpha, \beta \in \CC$. One checks without difficulty that $\alpha(\beta X) = (\alpha\beta) X$ and $\beta^{-1}(\alpha^{-1}X)=(\alpha\beta)^{-1}X$ whenever $s(\alpha) = r(\beta)$, while we have $\alpha(\beta X) = \beta^{-1}(\alpha^{-1}X) = \varnothing $ otherwise. 
Moreover, we have \[
\alpha\beta^{-1}\CC= \tau_\alpha\big(s(\alpha)\CC \cap \beta^{-1}\CC\big) =  \tau_\alpha\big(s(\alpha)\CC \cap s(\beta)\CC\big).\]
Hence, $\alpha\beta^{-1}\CC = \alpha\CC$ if $s(\alpha)=s(\beta)$, while $\alpha\beta^{-1}\CC=\varnothing$ otherwise.  On the other hand, we have  
  \[\beta^{-1}\alpha\CC = \sigma_\beta (\beta\CC \cap \alpha \CC) ={\rm dom}(\sigma_\alpha\tau_\beta) \]
and
 \[\alpha^{-1}\beta\CC = \sigma_\alpha(\alpha\CC \cap \beta \CC) = {\rm ran}(\sigma_\alpha\tau_\beta)\,.\]
If $(\alpha, \beta) \not\in \mathcal{Z}_\CC$, i.e., $r(\alpha) \not= r(\beta)$, then $\alpha\CC \cap \beta \CC =\varnothing$, so $\sigma_\alpha\tau_\beta = {\rm id}_\varnothing$ and $\beta^{-1}\alpha\CC = \alpha^{-1}\beta\CC =\varnothing$. But if $(\alpha, \beta) \in \mathcal{Z}_\CC$, then $\varphi_{(\alpha, \beta)} = \sigma_\alpha\tau_\beta $, so we get
 \[{\rm dom}(\varphi_{(\alpha, \beta)})= \beta^{-1}\alpha\CC\,\text{,}\,\,{\rm ran}(\varphi_{(\alpha, \beta)})
 = \alpha^{-1}\beta\CC\,.\]

This may be generalized as follows:
\begin{lem} \label{dom-ran}
For $\zeta=(\alpha_1, \beta_1, \cdots, \alpha_n,\beta_n) \in \mathcal{Z}_\CC$ we have 
\begin{align*}
{\rm dom}(\varphi_\zeta) &= \beta_n^{-1}\alpha_n \cdots \beta_1^{-1} \alpha_1\CC
 \quad  \text{ and }\\
{\rm ran}(\varphi_\zeta) &= \alpha_1^{-1} \beta_1 \cdots \alpha_n^{-1}\beta_n\CC.
\end{align*}
\end{lem}
\begin{proof} We have seen that the assertion holds when $n=1$. So assume it holds for some $n \geq1$. We first observe that if 
$\alpha, \beta \in \CC$ satisfy $r(\alpha) = r(\beta)$, so 
$(\alpha, \beta) \in \mathcal{Z}_\CC$, and $X \subseteq \CC$, then we have
\begin{equation}\label{b-1a}
\beta^{-1}\alpha X= \sigma_\beta (\beta\CC \cap \alpha X) = \sigma_\beta\tau_\alpha(\alpha^{-1}\beta\CC\, \cap X)
\end{equation}
Now, let $\zeta=(\alpha_1, \beta_1, \cdots, \alpha_{n+1},\beta_{n+1}) \in \mathcal{Z}_\CC$. Then $\varphi_{\zeta} = \varphi_{\zeta'}\,\varphi_{(\alpha_{n+1}, \beta_{n+1})},$ where $\zeta':=(\alpha_1, \beta_1, \cdots, \alpha_{n},\beta_{n})\in \mathcal{Z}_\CC$. 

Thus, using equation (\ref{b-1a}) with $\alpha = \alpha_{n+1},\, \beta= \beta_{n+1} $ and $X={\rm dom}(\varphi_{\zeta'})$ at the third step, and the induction hypothesis at the final step, we get 
\begin{align*}{\rm dom}(\varphi_\zeta) &= \varphi_{(\alpha_{n+1}, \beta_{n+1})}^{-1}\big( {\rm ran}(\varphi_{(\alpha_{n+1}, \beta_{n+1})}) \cap {\rm dom}(\varphi_{\zeta'})\big)\\
&=\sigma_{\beta_{n+1}}\tau_{\alpha_{n+1}}\big( \alpha^{-1}_{n+1} \beta_{n+1}\CC \,\cap\, {\rm dom}(\varphi_{\zeta'})\big)\\
&=\beta_{n+1}^{-1}\alpha_{n+1}\,{\rm dom}(\varphi_{\zeta'}) \\
&= \beta_{n+1}^{-1}\alpha_{n+1} \beta_n^{-1}\alpha_n \cdots \beta_1^{-1} \alpha_1\CC.
\end{align*}
This implies that ${\rm ran}(\varphi_\zeta) = {\rm dom}(\varphi_{\bar\zeta})= \alpha_1^{-1} \beta_1 \cdots \alpha_n^{-1}\beta_n\alpha_{n+1}^{-1}\beta_{n+1}\CC.$
\end{proof}
 For $\zeta =(\alpha_1, \beta_1, \cdots, \alpha_n,\beta_n) \in \mathcal{Z}_\CC$ we set 
\begin{equation}\label{Az}
A(\zeta) = \beta_n^{-1}\alpha_n \cdots \beta_1^{-1} \alpha_1\CC\,.
\end{equation} 
 As Lemma  \ref{dom-ran} says that $A(\zeta) = {\rm dom}(\varphi_\zeta)$,  we get from (\ref{E-ZM}) the following:
\begin{prop} \label{EZM}
\[E({\rm ZM}(\CC)) =  \{ {\rm id}_{A(\zeta)} : \zeta \in \mathcal{Z}_\CC\big\}\cup \{{\rm id}_\varnothing\}.\]
\end{prop} 
A useful consequence of this proposition is that the family of subsets of $\CC$ given by \[\mathcal{J}(\CC):=\big\{A(\zeta) : \zeta \in \mathcal{Z}_\CC\big\}\cup \{\varnothing\}\]  
is closed under finite intersections. 

Our next aim is to show that ${\rm ZM}(\CC)$ is isomorphic to a certain inverse semigroup of partial isometries in $\TT_\ell(\CC)$. We will use the notation introduced in the previous subsection.

For $\zeta=(\alpha_1, \beta_1,\,\cdots, \alpha_n,\beta_n) \in \mathcal{Z}_\CC$, we  define 
\[V_\zeta = V_{\alpha_1}^* V_{\beta_1} \cdots V_{\alpha_n}^*V_{\beta_n} \in 
\TT_\ell(\CC). \] 
It follows readily that for each $\eta \in \ell^2(\CC)$ we have
\begin{equation}\label{Vz}
V_\zeta\, \eta=\eta\circ\varphi_{\bar\zeta}\, \in \ell^2\big(A\big(\bar\zeta\big)\big) \subseteq \ell^2(\CC).
\end{equation} 
Hence $V_\zeta$ is a partial isometry with initial space $\ell^2\big(A(\zeta)\big)$ and final space $\ell^2\big(A(\bar\zeta)\big)$, i.e., 
$V_\zeta^*V_\zeta = P_{A(\zeta)}$ and $V_\zeta V^*_\zeta = P_{A(\bar\zeta)}$. Moreover, for $\zeta, \zeta' \in \mathcal{Z}_\CC$, we get
\[
V_\zeta V_{\zeta'}=
\begin{cases}
V_{\zeta  \zeta'}
\case s(\zeta)=r(\zeta'),
\\
\hspace{1ex} 0
& \righttext{otherwise.} 
\end{cases}
\]
Note that 
if $\zeta, \zeta' \in \mathcal{Z}_\CC$ satisfy that $\varphi_\zeta =\varphi_{\zeta'}$, then  (\ref{Vz}) implies that $V_\zeta = V_{\zeta'}$. It follows that the map $
\pi_\ell \: {\rm ZM}(\CC)\to \TT_\ell(\CC) $ 
defined  by $\pi_\ell ({\rm id}_\varnothing) = 0$ and
 \[
 \pi_\ell (\varphi_\zeta) = V_\zeta\]  for each $\zeta \in \mathcal{Z}_\CC$ is well-defined. Note also that for  $\zeta, \zeta' \in \mathcal{Z}_\CC$ we have
\begin{align*}
\pi_\ell
(\varphi_\zeta\varphi_{\zeta'}) &=
\begin{cases}
\pi_\ell
(\varphi_{\zeta  \zeta'})
\case s(\zeta)=r(\zeta'),
\\
\hspace{1ex} \pi_r({\rm id}_\varnothing) 
& \righttext{otherwise} 
\end{cases}
= 
\begin{cases}
V_{\zeta  \zeta'}
\case s(\zeta)=r(\zeta'),
\\ 
\hspace{1ex} 0
& \righttext{otherwise}
\end{cases}
\\
&= V_{\zeta}V_{\zeta'} = 
\pi_\ell
(\varphi_\zeta)
\pi_\ell
(\varphi_{\zeta'}),
\end{align*}
and 
\[
\pi_\ell
(\varphi_{\bar\zeta}) =V_{\bar\zeta} = V^*_\zeta = 
\pi_\ell
(\varphi_{\zeta})^*.
\]

Let us make a digression and consider an inverse semigroup $S$.
We recall that a map $\pi$ from $S$ into a $C^*$-algebra $B$ is called a {\it representation of $S$ in $B$} when 
 \ $\pi(st) = \pi(s)\pi(t)$ and $\pi(s^*)=\pi(s)^*$ for all $s, t\in S$; when $S$ has a zero element $0$, we also require that  $\pi(0) = 0$. It is well known that $\pi(S)$ is then an inverse semigroup with respect to the product and the adjoint operation in $B$, which consists of partial isometries in $B$.

Now, as we have checked above, $\pi_\ell$ is a representation of ${\rm ZM}(\CC)$ in $\TT_\ell(\CC)$. 
Moreover, $\pi_\ell$ is injective:

Indeed, assume  $V_\zeta= V_{\zeta'}$ for $\zeta, \zeta' \in \mathcal{Z}_\CC$. Then $P_{A(\zeta)}= P_{A(\zeta')}$, so $A(\zeta)= A(\zeta')$. 
Let $\gamma \in A(\zeta)= A(\zeta')$. From (\ref{Vz}) we get that \[V_\zeta \, \delta_\gamma = \delta_\gamma \circ \varphi_{\bar\zeta} = \delta_{\varphi_\zeta(\gamma)},\]
and, similarly,  $V_{\zeta'} \, \delta_\gamma = \delta_{\varphi_{\zeta'}(\gamma)}$.  It follows that $\delta_{\varphi_\zeta(\gamma)}= \delta_{\varphi_{\zeta'}(\gamma)}$, so $\varphi_\zeta(\gamma)=\varphi_{\zeta'}(\gamma)$. Thus $\varphi_\zeta = \varphi_{\zeta'}$, as desired.

As $\pi_\ell ({\rm ZM}(\CC)) = \{ V_\zeta : \zeta \in \mathcal{Z}_\CC\} \cup \{ 0\}$, we get:
\begin{prop} \label{prop ez in reg}  $\{ V_\zeta : \zeta \in \mathcal{Z}_\CC\} \cup \{ 0\}$ is an inverse semigroup of partial isometries in $\TT_\ell(\CC)$ which is isomorphic to ${\rm ZM}(\CC)$.
\end{prop}

\subsection*{The $C^*$-algebra $C^*({\rm ZM}(\CC))$} 
Let $S$ be an inverse semigroup having a zero element $0$.  
By the {\it full $C^*$-algebra $C^*(S)$} of $S$
we will mean the $C^*$-algebra which is universal for representations of $S$ in $C^*$-algebras. We stress that we only consider zero-preserving representations of $S$ and note that $C^*(S)$ is  often denoted by $C^*_0(S)$ in the literature. 
For completeness we recall how $C^*(S)$ is obtained from the Banach $*$-algebra $\ell^1(S)$ naturally associated to $S$ (see e.g.~\cite{pat:book}).
Letting $C^*\big(\ell^1(S)\big)$ denote the enveloping $C^*$-algebra of $\ell^1(S)$, we identify $\ell^1(S)$ with its canonical copy in $C^*\big(\ell^1(S)\big)$. If $\delta_s$ denote the usual delta function at $s$, we have $\delta_s\delta_0 = \delta_0\delta_s = \delta_0 $ for all $s\in S$, and it follows that $\C\delta_0$ is a closed ideal  of $C^*\big(\ell^1(S)\big)$. So we can set \[C^*(S) = C^*\big(\ell^1(S)\big)/\,\C\delta_0.\]
The map $\iota: S \to C^*(S)$ given by $s \mapsto \delta_s + \C\delta_0$ is then an injective representation of $S$ in $C^*(S)$ and we will often
consider $S$ as embedded in $C^*(S)$ via this map. It is now easy to verify that if $\pi$ is a representation of $S$ in a $C^*$-algebra $B$, then $\pi$ extends uniquely to a homomorphism  $\widetilde\pi$ from $C^*(S)$ into $B$, determined by \[\widetilde{\pi} (f + \C\delta_0) = \sum_{s\in S} f(s) \pi(s)\]
for all  $f\in \ell^1(S)$.

Next we recall that an inverse semigroup $S$ has a partial order (see e.g.~ \cite{lawsonbook} or \cite{pat:book}) given by 
$s\leq t$ if and only if $s=ss^*t$.
For $e,f \in E(S)$, we then have $e \leq f$ if and only if $e = ef$.  

Following  \cite{DonMil}, we will say that a homomorphism $\theta\:S\to S'$ from $S$ into an inverse semigroup $S'$ is \emph{finitely join-preserving} if for every finite subset $C$ of $S$ having a join $\vee C$ in $S$, the join  $\vee \theta(C)$  exists in $S'$ and is equal to $\theta(\vee C)$. As shown in \cite[Proposition 3.2]{DonMil}, this is equivalent to requiring that the restriction of $\theta$ to $E(S)$ is finitely join-preserving.

We will need another concept introduced by Donsig and Milan in \cite{DonMil}, closely related to Exel's notion of tightness  introduced \cite{exelcomb}.
Let $a \in S$. A finite subset $C$ of $S$ is said to be a {\it cover} of $a$ if $c\leq a$ for every $c\in C$ and if for every $s \in S$ with $s\leq a$ there exists some $c \in C$ and a nonzero element $b$ in $S$ such that  $b\leq s$ and $b\leq c$. 
A homomorphism $\theta$ from $S$ into another inverse semigroup $S'$ is then called {\it cover-to-join} if, whenever $C$ is a cover of some  $a\in S$, the join $\vee\theta(C)$ exists in $S'$ and is equal to $\theta(a)$. Note that this property is stronger than requiring that $\theta$ be
finitely join-preserving. If $\pi$ is  a representation of $S$ in a $C^*$-algebra $B$, and we consider it as a homomorphism from $S$ into 
the inverse semigroup $\pi(S)$, then it follows from \cite[Corollary 2.3]{DonMil} that $\pi$  is {\it tight} in Exel's sense if and only if it is cover-to-join.
 
We now consider a left cancellative small category $\CC$.  It will be useful to introduce the following conditions for a family $\{T_\zeta : \zeta \in \mathcal{Z}_\CC \}$ in a $C^*$-algebra $B$
 that were introduced in \cite[Definition 9.1 and Lemma 9.2]{jackpaths2}:
\begin{itemize}
\item[(S1)] $T_\zeta \,T_{\zeta'} = T_{\zeta \zeta'}$ if $s(\zeta)=r(\zeta')$, while $T_\zeta \,T_{\zeta'}= 0$ otherwise, 
\item[(S2)] $T_{\overline{\zeta}} = T_\zeta^*$,
\item[(S3)] $T^*_\zeta \,T_\zeta= \bigvee_{j=1}^n T^*_{\zeta_j}T_{\zeta_j}$ \, if $ A(\zeta)=\bigcup_{j=1}^n A(\zeta_j)$,
\item[(S4)]  $T_\zeta = T^*_\zeta\, T_\zeta$ \, if $\varphi_\zeta = {\rm id}_{A(\zeta)}$,
\item[(S5)] $T_\zeta = T_{\zeta'}$ if $\varphi_\zeta = \varphi_{\zeta'}$.
\end{itemize}
We first note that if (S1) -- (S3) are satisfied, then (S4) is equivalent to (S5) 
(cf.~\cite[Lemma 9.2(iv)]{jackpaths2}).
Next, we note that if (S1), (S2) and (S5) hold, then one readily checks that the map $\pi: {\rm ZM}(\CC)\to B$ defined by 
$\pi({\rm id}_\varnothing)= 0$ and  $\pi(\varphi_\zeta) = T_\zeta$ for each $\zeta \in \mathcal{Z}_\CC$ is a well-defined representation of ${\rm ZM}(\CC)$ in $B$. In particular, this implies that each $T_\zeta$ is a partial isometry in $B$.
Conversely, if $\pi$ is a representation of ${\rm ZM}(\CC)$ in a $C^*$-algebra $B$ and we set $T_\zeta = \pi(\varphi_\zeta)$ for each $\zeta \in \mathcal{Z}_\CC$, then $\{T_\zeta : \zeta \in \mathcal{Z}_\CC \}$ is easily seen to satisfy (S1), (S2) and (S5). 
Thus, setting $t'_\zeta := \iota(\varphi_\zeta)$ for each $\zeta \in \mathcal{Z}_\CC$, we get that $C^*({\rm ZM}(\CC))$ may be described as the universal $C^*$-algebra generated by a family $\{t'_\zeta : \zeta \in \mathcal{Z}_\CC \}$ satisfying the relations (S1), (S2) and (S5).

Next, we characterize when a representation $\pi$ of ${\rm ZM}(\CC)$ in a $C^*$-algebra $B$ is finitely join-preserving.
In view of Proposition \ref{EZM}, we may identify $E({\rm ZM}(\CC))$ as a semilattice with $\mathcal{J}(\CC)= \{A(\zeta) : \zeta \in \mathcal{Z}_\CC\} \cup \{\varnothing\}$, the product in being given by set-intersection and the partial order being given by set-inclusion. Using Donsig and Milan's result on finitely join-preserving homomorphisms, we get that
 $\pi$ is finitely join-preserving if and only if its restriction to $E({\rm ZM}(\CC))$ is finitely join-preserving,  if and only if the following condition holds:

For any $\zeta, \zeta_1, \ldots , \zeta_n \in \mathcal{Z}_\CC$ such that $ A(\zeta)=\bigcup_{j=1}^n A(\zeta_j)$, we have 
\begin{equation}\label{join-p}
\pi({\rm id}_{A(\zeta)}) =  \bigvee_{j=1}^n \pi({\rm id}_{A(\zeta_j)}).
\end{equation}
Since ${\rm id}_{A(\zeta)} = \varphi_{\overline{\zeta}\zeta} = \varphi_{\overline{\zeta}} \varphi_\zeta$, we have that $\pi({\rm id}_{A(\zeta)}) =  \pi(\varphi_\zeta)^* \pi(\varphi_\zeta)$, and similarly for the $\zeta_j$.  Thus Equation \eqref{join-p} is equivalent to (S3).
As an application of this characterization, let us show that  the representation $
\pi_\ell
\: {\rm ZM}(\CC)\to \TT_\ell(\CC)$ is finitely join-preserving: 

Let $\zeta, \zeta_1, \ldots , \zeta_n \in \mathcal{Z}_\CC$ be such that $ A(\zeta)=\bigcup_{j=1}^n A(\zeta_j)$. Then 
\begin{align*}
\pi_\ell ({\rm id}_{A(\zeta)}) &= 
\pi_\ell (\varphi_\zeta^*\varphi_\zeta) = V_\zeta^*V_\zeta = P_{A(\zeta)}\\
&= P_{\, \bigcup_{j=1}^n A(\zeta_j)} = \bigvee_{j=1}^n P_{A(\zeta_j)} 
=\bigvee_{j=1}^n 
\pi_\ell ({\rm id}_{A(\zeta_j)}).
\end{align*}
Our next result, which relies heavily on arguments from \cite{jackpaths2} and \cite{DonMil}, shows that in the finitely aligned case, one may switch from representations of $\CC$ to finitely join-preserving representations of ${\rm ZM}(\CC)$, or vice-versa, whenever convenient.

\begin{thm}\label{FJ}
Assume $\CC$ is finitely aligned. Then for every representation
$T$ of \,$\CC$ in a $C^*$-algebra $B$
there is a unique finitely join-preserving representation $\pi_T$ of ${\rm ZM}(\CC)$ making the following diagram commute:
\[
\xymatrix{
C^*({\rm ZM}(\CC))\, \ar@{-->}[r]^-{\widetilde{\pi_T}}
& B &\talg\CC \ar@{->}[l]_{\phi_T}
\\
{\rm ZM}(\CC)\ar@{^(->}[u]^{} \ar@{-->}[ur]^{\pi_T} 
&\CC \,\ar@{->}[u]^-T \ar[l]^-{\ \tau} \ar@{->}[ur]_t
}
\]

Moreover, the map $T \to \pi_T$ gives a bijection between representations of $\CC$ and finitely join-preserving representations of ${\rm ZM}(\CC)$.
\end{thm}

\begin{proof}  Let $T$ be a representation
of \,$\CC$ in a $C^*$-algebra $B$. By \cite[Theorem 9.7]{jackpaths2}, representations of $\CC$ correspond with representations of $\TT(\CC)$. For $\zeta = (\alpha_1, \beta_1, \cdots, \alpha_n,\beta_n) \in \mathcal{Z}_\CC$,  define $T_\zeta \in B$ by
\[  T_\zeta =T_{\alpha_1}^* T_{\beta_1} \cdots T_{\alpha_n}^*T_{\beta_n}.\]
 By \cite[Theorem 9.4]{jackpaths2},
we get that the family $\{ T_\zeta : \zeta \in \mathcal{Z}_\CC\}$ satisfies the conditions (S1) -- (S4), hence also (S5).
Thus the map $\pi_T\:{\rm ZM}(\CC) \to B$ given by $\pi_T({\rm id}_\varnothing) = 0$ and $\pi_T(\varphi_\zeta) = T_\zeta$ for $\zeta\in \mathcal{Z}_\CC$ 
is a representation of  ${\rm ZM}(\CC)$. 
It follows that
 \[\pi_T ({\rm id}_{A(\zeta)})= \pi_T (\varphi_{\bar\zeta\zeta})=T_{\bar\zeta\zeta}= T_\zeta^*T_\zeta.\]
 So if $ A(\zeta)=\bigcup_{j=1}^n A(\zeta_j)$, then by 
 (S3)
    \[\pi_T ({\rm id}_{A(\zeta)})=  T_\zeta^*\,T_\zeta =  \bigvee_{j=1}^n T^*_{\zeta_j}T_{\zeta_j} =  \bigvee_{j=1}^n \pi_T ({\rm id}_{A(\zeta_j)}).\]
Thus $\pi_T$ is finitely join-preserving (cf.~(\ref{join-p})). Moreover, for each $\alpha \in \CC$, we have \[\pi_T(\tau_\alpha) = \pi_T\big(\varphi_{(r(\alpha), \alpha)}\big)= T_{(r(\alpha), \alpha)} = T^*_{r(\alpha)} T_\alpha=
  T_{r(\alpha)} T_\alpha =   T_{r(\alpha)\alpha} = T_\alpha.\]
  Hence $\pi_T\circ \tau = T$, as desired to make the diagram commute.
  The uniqueness of $\pi_T$ is immediate from this identity.

To prove the last assertion of the theorem, let $\pi$ be a finitely join-preserving representation of ${\rm ZM}(\CC)$ in a $C^*$-algebra $B$. 
For each $\zeta \in \mathcal{Z}_\CC$ set $T_\zeta =\pi(\varphi_{\zeta}) \in B$. Then one readily checks 
the family $\{T_{\zeta} : \zeta\in\mathcal{Z}_\CC\}$ satisfies the relations 
(S1) -- (S4),
hence define a representation of $\TT(\CC)$.  By \cite[Theorem 9.7]{jackpaths2},
it follows that the map $T:\CC\to B$ defined for each $\alpha \in \CC$ by 
\[T_\alpha = \pi(\tau_\alpha) = \pi(\varphi_{(r(\alpha),\alpha)})\]  is a representation of $\CC$ in $B$.  Since $T = \pi \,\circ\, \tau$ we get that $\pi = \pi_{T}$, showing that the map $T\to \pi_T$ is surjective.
\end{proof}

Assume $\CC$ is finitely aligned.
Note that by considering the universal representation $t\:\CC \to \talg\CC$, we get from Theorem \ref{FJ} that there is 
a unique finitely join-preserving representation $\pi_t\:{\rm ZM}(\CC)\to \TT(\CC)$ making the following diagram commute:
 \[
\xymatrix{
C^*({\rm ZM}(\CC)) \,\ar@{->}[r]^-{\widetilde{\pi_t}}
&\TT(\CC) 
\\
{\rm ZM}(\CC)\ar@{->}[ur]_{\pi_t} \ar@{^(->}[u]^{}
& \CC \,\ar@{->}[u]^-t \ar[l]^-\tau
}
\]
We note that $\widetilde{\pi_t}$ is surjective (because  $\TT(\CC) = C^*(t))$.  Moreover, the universal property of $(\talg\CC, t)$ can be reformulated as follows:
\begin{cor}\label{FJ2}
Assume $\CC$ is finitely aligned and let $\pi$ be a finitely join-preserving representation of ${\rm ZM}(\CC)$ in a $C^*$-algebra $B$. Then there is a unique homomorphism $\phi^\pi\:\talg\CC\to B$ 
which makes the following diagram commute:
\[
\xymatrix{
C^*({\rm ZM}(\CC))\ar@{->}[r]^-{\widetilde{\pi_t}} & \talg\CC  \ar@{-->}[d]^{\phi^\pi}_{!}
\\
{\rm ZM}(\CC)  \ar@{^(->}[u]^{} \ar@{->}[ur]^{\pi_t}  \ar@{->}[r]^-{\pi} &B
}
\]
\end{cor}

\begin{proof} By using Theorem \ref{FJ} we may write $\pi= \pi_T$ for a unique representation $T$ of $\CC$ in $B$. We may then set $\phi^\pi:= \phi_T\:\talg\CC\to B$. It is  straightforward to verify that $(\phi^\pi\, \circ \pi_t)\circ \tau = \pi\circ \tau$, and this implies that  $\phi^\pi\, \circ \pi_t = \pi$ since $\tau(\CC)$ generates ${\rm ZM}(\CC)$.
\end{proof}
 
 We also note that if  $\CC$ is finitely aligned, then the representation $V\:\CC\to \TT_\ell(\CC)$ satisfies that $\pi_V = 
 \pi_\ell$ (simply by comparing their definitions). Thus we recover from Theorem \ref{FJ} that the representation $ \pi_\ell
 :{\rm ZM}(\CC)\to \TT_\ell(\CC)$ is finitely join-preserving. 
 We also mention that $ \widetilde{\pi_\ell}$ factors  as follows:
 \[
\widetilde{\pi_\ell}:\xymatrix{
C^*({\rm ZM}(\CC)) \ar@{->}[r]^-{\widetilde{\pi_t}}
&\talg\CC \,\ar@{->}[r]^-{\phi_V}
& \TT_\ell(\CC).
}
\]
Indeed, we have \[(\phi_V\circ \widetilde{\pi_t})\circ \tau = \phi_V\circ t =V= \widetilde{\pi_V}\circ \tau = 
\widetilde{\pi_\ell}
\circ \tau ,\] which implies that 
$\widetilde{\pi_\ell}= \phi_V\circ \widetilde{\pi_t}.$ 

\begin{rem}\label{existenceTT}
Assume  $\CC$ is finitely aligned. 
As mentioned in Remark \ref{rem groupoid approach}
 the existence of $\talg\CC$ (as a certain nontrivial groupoid $C^*$-algebra) 
is established  in \cite{jackpaths2}.  An alternative way to proceed is to introduce 
the closed ideal of $C^*({\rm ZM} (\CC))$ given by
\[\mathcal{J}_{\rm FJ} := \bigcap \, \ker\, \widetilde\pi\,,\]
the intersection being taken over all finitely join-preserving representations $\pi$ of ${\rm ZM}(\CC)$ in $C^*$-algebras. 
It is then straightforward to see that the quotient $C^*$-algebra 
\[C_{\rm FJ}^*({\rm ZM}(\CC)) := C^*({\rm ZM} (\CC)) / \mathcal{J}_{\rm FJ}\] is universal for finitely join-preserving representations of ${\rm ZM}(\CC)$. More precisely, letting $\iota: {\rm ZM} (\CC) \to C^*({\rm ZM} (\CC))$ denote the canonical embedding and $q: C^*({\rm ZM} (\CC)) \to C_{\rm FJ}^*({\rm ZM}(\CC))$ denote the quotient map, then one sees that the pair $C_{\rm FJ}^*({\rm ZM}(\CC))$ and $q\circ \iota$ play the same role as $\talg\CC$ and $\pi_t$ in Corollary \ref{FJ2}. 
So we could have chosen to \emph{define} $\talg\CC$ as $C_{\rm FJ}^*({\rm ZM}(\CC))$, and shown, arguing as in the proof of Theorem \ref{FJ}, that this $C^*$-algebra is universal for representations of $\CC$.
\end{rem}
Recall from \cite[Definition 5.13]{jackpaths2} that the Toeplitz algebra $\TT(\CC)$ is defined for an arbitrary left cancellative small category as a groupoid $C^*$-algebra, and that by \cite[Theorem 9.4]{jackpaths2} it is characterized by (S1) -- (S4).
 Since the definition of $C_{\rm FJ}^*({\rm ZM}(\CC))$ in Remark \ref{existenceTT} also makes sense when $\CC$ is not finitely aligned, it 
 is equivalent to consider the Toeplitz algebra associated to a left cancellative small category $\CC$ as being defined by \[\talg\CC := C_{\rm FJ}^*({\rm ZM}(\CC)).\] 

The next result is essentially due to Donsig and Milan \cite{DonMil}:
\begin{thm}\label{CJ}
Assume that $\CC$ is finitely aligned and let $T$ be a representation of $\CC$ in a $C^*$-algebra $B$. Then 
 $T$  is covariant if and only if
the representation $\pi_T$ of ${\rm ZM}(\CC)$ in $B$ is cover-to-join.
\end{thm}
\begin{proof}
The proof of \cite[Theorem 3.7]{DonMil} adapts almost verbatim
to our more general situation, and we leave it to the reader to check this.
\end{proof}

When $\CC$ is finitely aligned, we may for example consider the universal covariant representation $\tilde t\:\CC \to \OO(\CC)$ and obtain that  
the associated representation $\pi_{\tilde t}\:{\rm ZM}(\CC)\to \OO(\CC)$ is cover-to-join. Moreover, we have:

\begin{cor}\label{C-J} Assume that $\CC$ is finitely aligned  
and let $\pi$ be a cover-to-join $(=$ tight$)$  representation of ${\rm ZM}(\CC)$ in a $C^*$-algebra $B$. Then there is a unique homomorphism $\psi^\pi\:\OO(\CC)\to B$ which makes the following diagram commute:
\[
\xymatrix{
C^*({\rm ZM}(\CC))\ar@{->}[r]^-{\widetilde{\pi_{\tilde t}}} & \OO(\CC)  \ar@{-->}[d]^{\psi^\pi}_{!}
\\
{\rm ZM}(\CC)  \ar@{^(->}[u]^{} \ar@{->}[ur]^{\pi_{\tilde t}}  \ar@{->}[r]^-{\pi} &B
}
\]
\end{cor}

\begin{proof}
By using Theorem \ref{FJ} and Theorem \ref{CJ} we may write $\pi= \pi_T$ for a unique covariant representation $T$ of $\CC$ in $B$. We may then set $\psi^\pi:= \psi_T\:\OO(\CC)\to B$. It is  then straightforward to verify that 
$\psi^\pi\, \circ \pi_{\tilde t} = \pi$.
\end{proof}
The essence of Corollary \ref{C-J} is that  if $\CC$ is finitely aligned, then $\OO(\CC)$ is isomorphic to the {\it tight $C^*$-algebra} \cite{exelcomb, DonMil} associated with the inverse semigroup ${\rm ZM}(\CC)$. 
A similar result holds in the general case - see \cite[Definition 10.8 and Theorem 10.10]{jackpaths2}.

\subsection*{The $C^*$-algebra $C^*_{\rm Li}(\CC)$} \label{subsec li algebras} 
Let $\CC$ be a left cancellative small category.
 We will say that $X\subseteq \CC$ is a \emph{right ideal} of $\CC$ if  $\alpha \beta \in X$ whenever $\alpha \in X$ and $\beta \in s(\alpha)\CC$. Note that we consider $\varnothing$ as a right ideal.
Note also that if $X$ is a right ideal of $\CC$, then it is straightforward to check that $\alpha X$ and $\alpha^{-1} X$ are also right ideals of $\CC$. Thus it follows from (\ref{Az}) that $A(\zeta)$ is a right ideal of $\CC$ for every $\zeta\in \mathcal{Z}_\CC$. In analogy with the terminology used by Li in \cite{LiSemigroup}
for left cancellative monoids, we will call \[\mathcal{J}(\CC) =\{A(\zeta) : \zeta \in \mathcal{Z}_\CC\}\cup \{\varnothing\}\] the family of {\it constructible right ideals} in $\CC$ (and we note that this coincides with Li's definition when $\CC$ is a monoid). As mentioned after Proposition \ref{EZM}, $\mathcal{J}(\CC)$ is closed under finite intersections. Note that for $E= A(\zeta)$ with $\zeta \in \mathcal{Z}_\CC$ all elements of $E$ have a common range, namely 
 $s(\zeta)$. We denote this by $r(E)$.  Li also uses a larger collection of right ideals, which we adapt as $\mathcal{J}^{(\cup)}(\CC)$, the set of all finite unions of elements of $\mathcal{J}(\CC)$.

In the paper quoted above, Li defines what he calls the full $C^*$-algebra associated with a left cancellative monoid, as well as several variations on this definition (see Definitions 2.2, 2.4, 3.2, and the remarks before section 3.1 in {\it op.~cit.}). Since we are dealing with a small category instead of a monoid, we must adapt the relations given in \cite{LiSemigroup} to account for the presence of multiple units.
 Thus we may define the $C^*$-algebra $C^*_{\rm Li}(\CC)$ associated with $\CC$ 
as the universal $C^*$-algebra generated by a family $\{v_\alpha\}_{\alpha \in \CC}$ of partial isometries and a family $\{ p_X\}_{X\in \mathcal{J}(\CC)}$ of projections satisfying the relations
\[  \text{(L1)} \quad v_\alpha^* v_\alpha = p_{s(\alpha)\CC}, \quad \text{(L2)} \quad v_\alpha v_\beta = \begin{cases} v_{\alpha\beta} &\righttext{if }s(\alpha) = r(\beta), \\ \ 0 &\righttext{otherwise,}\end{cases} \]
\[ \text{(L3)} \quad p_\varnothing = 0, \quad \text{(L4)} \quad p_X \, p_Y = p_{X\cap Y}, \quad \text{(L5)} \quad v_\alpha\, p_X v^*_\alpha = p_{\alpha X} \]
for $\alpha, \beta \in \CC$ and $X, Y \in \mathcal{J}(\CC)$.
\begin{lem}\label{ortho proj}
Assume $\{v_\alpha\}_{\alpha \in \CC}$ is a family of partial isometries in a $C^*$-algebra $B$ which satisfies (L2). Then $v_u$ is a projection for every $u \in \CC^0$.  Moreover, $v_\alpha v_\beta^* = 0$ whenever $\alpha, \beta \in \CC$ and $s(\alpha) \neq s(\beta)$.
\end{lem}
\begin{proof} Let  $u \in \CC^0$. Then (L2) gives that $v_u^2 = v_u$. Since $v_u$ is a partial isometry, this implies that $v_u$ is a projection. (This is surely well known, and may be proven as follows. We can assume that $B$ is faithfully represented on a Hilbert space $H$. Since $v_u$ acts as the identity on its final space $N$, we get that $N \subseteq M$, where $M$ denotes the initial space of $v_u$. Now the same argument applies to the partial isometry $v_u^*$. So we also get that $M\subseteq N$, hence that $M=N$. It follows that $v_u$ is the orthogonal projection of $H$ onto $M=N$.) 
Next,  consider $\alpha, \beta \in \CC$. Using (L2) we get 
$v_\alpha = v_{\alpha  s(\alpha)} = v_{\alpha} v_{s(\alpha)}$; using also that $v_{s(\beta)}$ is self-adjoint, we get $v_\beta^* =  (v_{\beta} v_{s(\beta)})^* =  v_{s(\beta)} v_{\beta}^*$. Now, if $s(\alpha) \neq s(\beta)$, then (L2) gives $v_{s(\alpha)}v_{s(\beta)} =0$, and we therefore get
\[v_\alpha v_\beta^* = v_{\alpha} v_{s(\alpha)}v_{s(\beta)} v_{\beta}^* = 0. \qedhere \] 
\end{proof}
Li also defines a variation compatible with unions. Thus we define $C^{*\,(\cup)}_{\rm Li}(\CC)$ as the universal $C^*$-algebra generated by a family $\{v_\alpha\}_{\alpha \in \CC}$ of partial isometries and a family $\{ p_X\}_{X\in \mathcal{J}^{(\cup)}(\CC)}$ of projections satisfying the relations (L1) -- (L3), (L4)$^{(\cup)}$, (L5)$^{(\cup)}$, and (L6), where (L4)$^{(\cup)}$ and (L5)$^{(\cup)}$ are the same as (L4) and (L5) but using ideals from $\mathcal{J}^{(\cup)}(\CC)$, and (L6) is the relation
\[
\text{(L6)} \quad p_{X \cup Y} = p_X \vee p_Y, \text{ for } X,\; Y \in \mathcal{J}^{(\cup)}(\CC).
\]
Since the projections $\{ p_X \}_{X \in \mathcal{J}^{(\cup)}(\CC)}$ commute, by (L4)$^{(\cup)}$, the join in (L6) is given by $p_X \vee p_Y = p_X + p_Y - p_X p_Y$ (as in \cite[Definition 2.4]{LiSemigroup}).

In the case where $\CC$ is a submonoid of a group, Li also defines an algebra more directly related to the definition of constructible ideals.  In fact, the definition does not require in an essential way the ambient group, and we adapt it to general left cancellative small categories as follows.  We define $C^*_{{\rm Li},\,s}(\CC)$ as the universal $C^*$-algebra generated by a family $\{v_\alpha\}_{\alpha \in \CC}$ of partial isometries and a family $\{ p_X\}_{X\in \mathcal{J}(\CC)}$ of projections satisfying the relations (L1) -- (L3) and a new relation: \\

(L7) for $\zeta = (\alpha_1,\beta_1, \ldots, \alpha_n,\beta_n) \in \mathcal{Z}_\CC$, if $\varphi_\zeta = \text{id}_{A(\zeta)}$ then \linebreak $v_{\alpha_1}^* v_{\beta_1} \cdots v_{\alpha_n}^* v_{\beta_n} = p_{A(\zeta)}$. \\

It is interesting to observe that (L7) implies (L1):  for $\alpha \in \CC$, 
$\varphi_{(\alpha,\alpha)} = \text{id}_{s(\alpha) \CC}$,
 hence by (L7) we have $v_\alpha^* v_\alpha = p_{s(\alpha)\CC}$.  (However, (L1) does not appear to be a consequence of (L2) -- (L5).)  We also note that if $\CC$ is a submonoid of a group, then $C^*_{{\rm Li},s}(\CC)$ reverts to the analogous algebra in \cite{LiSemigroup}.

Further, Li mentions, but does not explicitly define, a fourth algebra, which we will find it convenient to have: $C^{* \, (\cup)}_{{\rm Li},s}(\CC)$ is the universal $C^*$-algebra generated by a family $\{v_\alpha\}_{\alpha \in \CC}$ of partial isometries and a family $\{ p_X\}_{X\in \mathcal{J}^{(\cup)}(\CC)}$ of projections satisfying the relations (L1) -- (L3), (L6), and (L7).

Since the $^{(\cup)}$-versions are obtained from the other two by adding relation (L6), it is clear that there are surjections $\pi^{(\cup)}\: C^*_{\rm Li}(\CC) \to C^{*\, (\cup)}_{\rm Li}(\CC)$ and $\rho^{(\cup)}\: C^*_{{\rm Li},s}(\CC) \to C^{*\, (\cup)}_{{\rm Li},s}(\CC)$ taking generators to generators.  The following is an analog of \cite[Lemma 3.3]{LiSemigroup}.

\begin{lem} \label{lem pi s surjective}

There is a surjective homomorphism $\pi_s\: C^*_{\rm Li}(\CC) \to C^*_{{\rm Li},s}(\CC)$ carrying generators to generators. 

\end{lem}

\begin{proof}
To show this we assume that (L1) -- (L3) and (L7) hold 
for a family $\{v_\alpha\}_{\alpha \in \CC}$ of partial isometries and a family $\{ p_X\}_{X\in \mathcal{J}(\CC)}$ of projections.
 By the universal property of $C^*_{\rm Li}(\CC)$, it suffices to deduce that they satisfy (L4) and (L5).
For $\zeta = (\alpha_1,\beta_1, \ldots, \alpha_n,\beta_n) \in \mathcal{Z}_\CC$
 we set $v_\zeta := v_{\alpha_1}^* v_{\beta_1} \cdots v_{\alpha_n}^* v_{\beta_n}$.  It is immediate that $v_{\overline{\zeta}} = v_\zeta^*$. 
Since $\varphi_{\overline{\zeta}\zeta} = \text{id}_{A(\zeta)} = \text{id}_{A(\overline{\zeta}\zeta)}$, (L7) implies that 
\[v_{\overline{\zeta}\zeta} = p_{A(\overline{\zeta}\zeta)}= p_{A(\zeta)}.\] 
For  $\zeta, \zeta' \in \mathcal{Z}_\CC$ we also note that
\[ v_{\zeta} v_{\zeta'} = \begin{cases} v_{\zeta\zeta'} &\righttext{if } s(\zeta) = r(\zeta'), \\ \ 0 &\righttext{    otherwise.}\end{cases} \]
The first case of this equality is obvious, while the second case follows readily from the second assertion in Lemma \ref{ortho proj}.
Let now $\zeta_1, \zeta_2 \in \mathcal{Z}_\CC$. Assume first that  $s(\zeta_1) = r(\overline{\zeta_2})$. Then
\[ \varphi_{\overline{\zeta_1}\zeta_1 \overline{\zeta_2} \zeta_2} = 
\varphi_{\overline{\zeta_1}\zeta_1}\varphi_{\overline{\zeta_2} \zeta_2}=
\text{id}_{A(\zeta_1)}\text{id}_{A(\zeta_2)}=
\text{id}_{A(\zeta_1) \cap A(\zeta_2)}.\]
Thus 
\[
A(\overline{\zeta_1}\zeta_1 \overline{\zeta_2} \zeta_2) = 
\text{dom}(\varphi_{\overline{\zeta_1}\zeta_1 \overline{\zeta_2} \zeta_2})=
A(\zeta_1) \cap A(\zeta_2), 
\]
and $\varphi_{\overline{\zeta_1}\zeta_1 \overline{\zeta_2} \zeta_2} = \text{id}_{A(\overline{\zeta_1}\zeta_1 \overline{\zeta_2} \zeta_2)}$. 
Thus,  
(L7) gives that
\[
v_{\overline{\zeta_1}\zeta_1 \overline{\zeta_2} \zeta_2} = p_{A(\overline{\zeta_1}\zeta_1 \overline{\zeta_2} \zeta_2)},
\]
and we therefore get
\[
 p_{A(\zeta_1)} p_{A(\zeta_2)} 
= v_{\overline{\zeta_1}\zeta_1} v_{\overline{\zeta_2}\zeta_2} 
= v_{\overline{\zeta_1}\zeta_1\overline{\zeta_2}\zeta_2}
= p_{A(\overline{\zeta_1}\zeta_1\overline{\zeta_2}\zeta_2)}
= p_{A(\zeta_1) \cap A(\zeta_2).}
\]
Next, assume $s(\zeta_1) \neq  r(\overline{\zeta_2})$.
Then $s(\overline{\zeta_1}\zeta_1) \neq r(\overline{\zeta_2}\zeta_2)$, so 
\[p_{A(\zeta_1)} p_{A(\zeta_2)} =  v_{\overline{\zeta_1}\zeta_1} v_{\overline{\zeta_2}\zeta_2}= 0.\]
Moreover, we have $\text{id}_{A(\zeta_1) \cap A(\zeta_2)}=
 \varphi_{\overline{\zeta_1}\zeta_1} \varphi_{\overline{\zeta_2}\zeta_2}
 =\text{id}_{\varnothing},$
 so $A(\zeta_1) \cap A(\zeta_2) = \varnothing$. Hence, (L3) gives that $p_{A(\zeta_1) \cap A(\zeta_2)}=p_\varnothing =0$, and we see that $p_{A(\zeta_1)} p_{A(\zeta_2)} = p_{A(\zeta_1) \cap A(\zeta_2)}$ in this case too. 
 This shows that (L4) holds.

Similarly, let $\alpha \in \CC$ and $\zeta \in \mathcal{Z}_\CC$. Assume first that $s(\alpha) = s(\zeta)$. Then 
(L2) and (L7) give
\begin{align*}
v_\alpha p_{A(\zeta)} v_\alpha^*
&= v_{r(\alpha)} v_\alpha v_{\overline{\zeta}\zeta} v_\alpha^*v_{r(\alpha)}
= v_{(r(\alpha),\alpha)}v_{ \overline{\zeta} \zeta}v_{(\alpha,r(\alpha))}
= v_{(r(\alpha),\alpha) \overline{\zeta} \zeta (\alpha,r(\alpha))}\\
&= v_{\overline{\zeta(\alpha,r(\alpha))} \zeta(\alpha,r(\alpha))}
= p_{A(\zeta(\alpha,r(\alpha)))}.
\end{align*}
But
\[A(\zeta(\alpha,r(\alpha))) =\text{dom}(\varphi_{\zeta(\alpha,r(\alpha))})= \text{dom}(\varphi_\zeta \sigma_\alpha)
=\tau_\alpha(s(\alpha)\CC\cap A(\zeta))= \alpha A(\zeta),\]
so we get $v_\alpha p_{A(\zeta)} v_\alpha^* = p_{\alpha A(\zeta)}$. 

Next,  assume that $s(\alpha) \neq s(\zeta)$. Then $s(\alpha)\CC\cap A(\zeta) = \varnothing$ (because $r(s(\alpha)\CC) = s(\alpha)$ while $r(A(\zeta)) = s(\zeta)$). So
$ \alpha A(\zeta) =\tau_\alpha(s(\alpha)\CC\cap A(\zeta)) = \varnothing$, hence $p_{ \alpha A(\zeta)} = p_\varnothing = 0$. Moreover,
\[v_\alpha p_{A(\zeta)} v_\alpha^* = v_{(r(\alpha),\alpha)}v_{ \overline{\zeta} \zeta}v_{(\alpha,r(\alpha))} = 0\]
since $ s(\overline{\zeta} \zeta) = s(\zeta) \neq s(\alpha) = r(\alpha, r(\alpha))$. Thus we see that (L5) holds in this case too.
\end{proof}

The same argument shows the following as well.

\begin{lem}

There is a surjective homomorphism $\rho_s\: C^{*\,(\cup)}_{\rm Li}(\CC) \to C^{*\,(\cup)}_{{\rm Li},s}(\CC)$ carrying generators to generators. 

\end{lem}

Moreover we have the following result.

\begin{lem} \label{lem mu}

There is a surjective homomorphism $\mu\: C^{*\,(\cup)}_{{\rm Li},s}(\CC) \to \TT(\CC)$ carrying 
generators to generators
in the sense that each $v_\alpha$ is mapped to $t_{(r(\alpha),\alpha)}$,
and each $p_{A(\zeta)}$ is mapped to $t_{\overline{\zeta}
\zeta}$.
\end{lem}

\begin{proof}
As pointed out after Remark \ref{existenceTT}, $\TT(\CC)$ is generated by the family $\{ t_\zeta : \zeta \in \mathcal{Z}_\CC \}$
which satisfy the relations 
(S1) -- (S4), hence also (S5).
We will first show that the relations (L1) -- (L3), (L6), and (L7) hold for certain families $\{v'_\alpha\}_{\alpha \in \CC}$ and $\{ p'_X\}_{X \in \mathcal{J}^{(\cup)}(\CC)}$ in $\TT(\CC)$.

For $\alpha \in \CC$ and $\zeta \in \mathcal{Z}_\CC$ we define $v'_\alpha := t_{(r(\alpha),\alpha)} $, $p'_\varnothing := 0$,  
 and $p'_{A(\zeta)} := t_{\overline{\zeta}\zeta}$. 
Note that if $A(\zeta) = A(\zeta')$ then $\varphi_{\overline{\zeta}\zeta} = \varphi_{\overline{\zeta'}\zeta'}$, so by (S5) we have $t_{\overline{\zeta}\zeta} = t_{\overline{\zeta'}\zeta'}$, and $p'_{A(\zeta)}$ is therefore well-defined.

More generally, for $X = \bigcup_{i=1}^m A(\zeta_i)$, where $\zeta_1, \ldots, \zeta_m \in \mathcal{Z}_\CC$, we
define $p'_X := \bigvee_{i=1}^m t_{\overline{\zeta_i}\zeta_i}$.  We must show that $p'_X$ is well-defined.  
Suppose that $X = \bigcup_{i=1}^m A(\zeta_i) = \bigcup_{j=1}^n A(\xi_j)$.    Then $A(\zeta_i) = \bigcup_j A(\zeta_i) \cap A(\xi_j) = \bigcup_j A(\overline{\zeta_i}\zeta_i \overline{\xi_j}\xi_j)$, and similarly $A(\xi_j) = \bigcup_i A(\overline{\zeta_i}\zeta_i \overline{\xi_j}\xi_j)$.  By (S3), it follows that $\bigvee_{i=1}^m t_{\overline{\zeta_i}\zeta_i} = \bigvee_i(\bigvee_j t_{\overline{\zeta_i}\zeta_i \overline{\xi_j}\xi_j}) = \bigvee_j(\bigvee_i t_{\overline{\zeta_i}\zeta_i \overline{\xi_j}\xi_j}) = \bigvee_j t_{\overline{\xi_j}\xi_j}$.

Let $\alpha \in \CC$. By (S1) and (S2) we have
\[
{v'}_\alpha^* v'_\alpha = t_{(r(\alpha),\alpha)}^* t_{(r(\alpha),\alpha)}
= t_{\overline{(r(\alpha),\alpha)}(r(\alpha),\alpha)}
= p'_{A(r(\alpha),\alpha)}
= p'_{s(\alpha)\CC},
\]
verifying (L1).  

Next, let $\alpha, \beta \in \CC$. Assume that $s(\alpha) = r(\beta)$.
 By (S1) we have $v'_\alpha v'_\beta = t_{(r(\alpha),\alpha)} t_{(r(\beta),\beta)} = t_{(r(\alpha),\alpha,r(\beta),\beta)}$.  We note that $\varphi_{(r(\alpha),\alpha,r(\beta),\beta)} = \varphi_{(r(\alpha),\alpha \beta)}$.  
 Using (S5) we get 
\[v'_\alpha v'_\beta  
\, = t_{(r(\alpha),\alpha,r(\beta),\beta)}
 = t_{(r(\alpha),\alpha \beta)} = v'_{\alpha \beta}.\] 
On the other hand, if  $s(\alpha) \neq r(\beta)$, then (S1) gives
\[v'_\alpha v'_\beta = t_{(r(\alpha),\alpha)} t_{(r(\beta),\beta)} = 0.\]
Thus we have verified (L2).
(L3) is satisfied by definition of $p'_\varnothing$. Our definition of $p'_X$ implies (L6), and (L7) follows immediately from (S4). 
Hence, by the universal property of $ C^{*\,(\cup)}_{{\rm Li},s}(\CC)$, we get that there is a homomorphism $\mu: C^{*\,(\cup)}_{{\rm Li},s}(\CC) \to \TT(\CC)$ which maps each $v_\alpha$ to $ v'_\alpha $, and each $p_X$ to $p'_X$. Since each $t_\zeta$ belongs to the range of $\mu$ (as it may be written as a monomial in the $v'_\alpha$ and their adjoints), and  $\TT(\CC)$ is generated by the $t_\zeta$, we get that $\mu$ is surjective.
\end{proof}

\begin{lem} \label{lem li-zm}
There is a surjective homomorphism $g$ from $C^*_{{\rm Li},s}(\CC)$ onto $C^*({\rm ZM}(\CC))$ carrying
generators to generators in the sense that each $v_\alpha$ is mapped to $t'_{(r(\alpha),\alpha)}$ 
and each $p_{A(\zeta)}$ is mapped to $t'_{\overline{\zeta}\zeta}$.
\end{lem}
\begin{proof}
For $\alpha \in \CC$ and  $\zeta \in \mathcal{Z}_\CC$ we define $V_\alpha := t'_{(r(\alpha),\alpha)}$, $P_{A(\zeta)} := t'_{\overline{\zeta}\zeta}$ (which is independent of the choice of $\zeta$ by (S5)), and $P_\varnothing := 0$.  We will first verify (L1) -- (L3) and (L7) for $V_\alpha$ and $P_{A(\zeta)}$.  By definition of $P_\varnothing$, (L3) holds.  Suppose that $\zeta = (\alpha_1, \beta_1, \ldots, \alpha_n, \beta_n) \in \mathcal{Z}_\CC$.
Then
\begin{align*}
V_{\alpha_1}^* V_{\beta_1} \cdots V_{\alpha_n}^* V_{\beta_n}
&= {t'}_{(r(\alpha_1),\alpha_1)}^* t'_{(r(\beta_1),\beta_1)} \cdots {t'}_{(r(\alpha_n),\alpha_n)}^* t'_{(r(\beta_n),\beta_n)} \\
&= t'_{(\alpha_1, r(\alpha_1), r(\beta_1), \beta_1, \ldots, \alpha_n, r(\alpha_n), r(\beta_n), \beta_n)}, \text{ by (S1) and (S2),} \\
&= t'_\zeta, \text{ by (S5).}
\end{align*}
Thus, if $\varphi_\zeta = \text{id}_{A(\zeta)}$, that is, $\varphi_\zeta =\varphi_{\overline{\zeta}\zeta}$, then by (S5) we have  \[ t'_\zeta = t'_{\overline{\zeta}\zeta} = P_{A(\zeta)},\]
and we get $V_{\alpha_1}^* V_{\beta_1} \cdots V_{\alpha_n}^* V_{\beta_n} = t'_\zeta = P_{A(\zeta)}$
verifying (L7).  As we observed earlier (in the remarks before Lemma \ref{lem pi s surjective}), (L7) implies (L1).  Now let $\alpha, \beta \in \CC$. If $s(\alpha) = r(\beta)$ then
\begin{align*}
V_\alpha V_\beta
&= t'_{(r(\alpha),\alpha)} t'_{(r(\beta),\beta)} \\
&= t'_{(r(\alpha),\alpha,r(\beta),\beta)}, \text{ by (S1),} \\
&= t'_{(r(\alpha),\alpha\beta)}, \text{ by (S5),} \\
&= V_{\alpha\beta}.
\end{align*}
 If $s(\alpha) \neq r(\beta)$, then $s((r(\alpha),\alpha)) \neq r((r(\beta),\beta))$, so (S1) gives that $V_\alpha V_\beta = t'_{(r(\alpha),\alpha)} t'_{(r(\beta),\beta)} = 0$. Thus we have verified (L2). 
Finally we note that each $V_\alpha$ is a partial isometry (since $V_\alpha = t'_{(r(\alpha),\alpha)}$) and that each $P_{A(\zeta)}$ is a projection (since $P_{A(\zeta)} = t'_{\overline{\zeta}\zeta} = {t'}^*_\zeta t'_\zeta$ and $t'_\zeta$ is a partial isometry).    
Hence, by the universal property of $C^*_{{\rm Li},s}(\CC)$,  we get that there is a homomorphism $g : C^*_{{\rm Li},s}(\CC) \to C^*({\rm ZM}(\CC))$ such that  $g(v_\alpha) = V_\alpha$ for each $\alpha \in \CC$ and $g(p_{A(\zeta)}) = P_{A(\zeta)}$ for each $\zeta \in \mathcal{Z}_\CC$. If we set $T_\zeta = v_{\alpha_1}^* v_{\beta_1} \cdots v_{\alpha_n}^* v_{\beta_n}$ for  $\zeta = (\alpha_1, \beta_1, \ldots, \alpha_n, \beta_n) \in \mathcal{Z}_\CC$, we get that $g(T_\zeta) = t'_\zeta$. Since $C^*({\rm ZM}(\CC))$ is generated by the $t'_\zeta$, it follows that $g$ is surjective. 
\end{proof}

Using the previous lemmas, and letting $q\:C^*({\rm ZM}(\CC))\to \TT(\CC)$ denote the canonical surjective homomorphism, which maps each $t_\zeta$ to $t'_\zeta$, it is not difficult to see that we have established the commutativity of the following diagram:
\[
\begin{tikzpicture}[scale=3]
\node (0_1) at (0,1) [rectangle] {$C^*_{\rm Li}(\CC)$};
\node (1_1) at (1,1) [rectangle] {$C^*_{{\rm Li},s}(\CC)$};
\node (0_0) at (0,0) [rectangle] {$C^{*\, (\cup)}_{\rm Li}(\CC)$};
\node (1_0) at (1,0) [rectangle] {$C^{*\, (\cup)}_{{\rm Li},s}(\CC)$};
\node (2_0) at (2,0) [rectangle] {$\TT(\CC)$};
\node (2_1) at (2,1) [rectangle] {$C^*({\rm ZM}(\CC))$};
\draw[-latex,thick] (0_1) -- (1_1) node[pos=0.5, inner sep=0.5pt, anchor=south] {$\pi_s$};
\draw[-latex,thick] (0_0) -- (1_0) node[pos=0.5, inner sep=0.5pt, anchor=south] {$\rho_s$};
\draw[-latex,thick] (0_1) -- (0_0) node[pos=0.5, inner sep=0.5pt, anchor=east] {$\pi^{(\cup)}$};
\draw[-latex,thick] (1_1) -- (1_0) node[pos=0.5, inner sep=0.5pt, anchor=east] {$\rho^{(\cup)}$};
\draw[-latex,thick] (1_0) -- (2_0) node[pos=0.5, inner sep=0.5pt, anchor=south] {$\mu$};
\draw[-latex,thick] (1_1) -- (2_1) node[pos=0.5, inner sep=0.5pt, anchor=south] {$g$};
\draw[-latex,thick] (2_1) -- (2_0) node[pos=0.5, inner sep=0.5pt, anchor=east] {$q$};
\end{tikzpicture}
\]
\begin{thm} \label{thm li algebras}

\begin{enumerate}
\item \label{thm li algebras 1} 
$g$ is an isomorphism.

\item \label{thm li algebras 2} $\mu$ is an isomorphism.

\item \label{thm li algebras 3} $\rho_s$ is an isomorphism if $\CC$ is finitely aligned.

\item \label{thm li algebras 4} 
$\pi^{(\cup)}$, $\rho^{(\cup)}$ and $q$
 are not generally one-to-one, even if $\CC$ is finitely aligned.

\item \label{thm li algebras 5} $\pi_s$ is not generally one-to-one, even if $\CC$ is finitely aligned.

\end{enumerate}

\end{thm}

\begin{proof}

\eqref{thm li algebras 1}:
We will  show that $g$ has an inverse $f\: C^*({\rm ZM}(\CC)) \to C^*_{{\rm Li},s}(\CC)$.    For $\zeta = (\alpha_1, \beta_1, \ldots, \alpha_n, \beta_n) \in \mathcal{Z}_\CC$ we set $T_\zeta = v_{\alpha_1}^* v_{\beta_1} \cdots v_{\alpha_n}^* v_{\beta_n} \in C^*_{{\rm Li},s}(\CC)$. We are assuming (L1) -- (L3) and (L7) for the $v_\alpha$ and the $p_X$ with $\alpha \in \CC$ and $X \in \mathcal{J}(\CC)$, and we are going to deduce (S1),
(S2), and (S5)
 for the $T_\zeta$. It follows readily from the definition of $T_\zeta$ and Lemma \ref{ortho proj} that (S1) and (S2) hold.  If $\zeta \in  \mathcal{Z}_\CC$, then  $ \varphi_{\overline{\zeta}\zeta} = \text{id}_{A(\zeta)} =  \text{id}_{A(\overline{\zeta}\zeta)}$, so using (L7) we get 
\[ T_\zeta^*T_\zeta = T_{\overline{\zeta}\zeta}= p_{A(\overline{\zeta}\zeta)} = p_{A(\zeta)},\]
which implies that $T_\zeta$ is a partial isometry.
Next, suppose that $\zeta, \zeta' \in  \mathcal{Z}_\CC$ and $\varphi_\zeta = \varphi_{\zeta'}$. 
Then $A(\zeta) = A(\zeta')$ and $ \text{id}_{A(\zeta)} =  \text{id}_{A(\overline{\zeta}\zeta)} =  \varphi_{\overline{\zeta}} \varphi_\zeta =
 \varphi_{\overline{\zeta}} \varphi_{\zeta'} =  \varphi_{\overline{\zeta}\zeta'}$. Thus
 $A(\zeta) = A(\overline{\zeta}\zeta')$ and  $\varphi_{\overline{\zeta}\zeta'} =  \text{id}_{A(\overline{\zeta}\zeta')}$.  Using (L7) we get
 \begin{align*}
 T_\zeta^* T_{\zeta'} &= T_{\overline{\zeta}\zeta'} = p_{A(\overline{\zeta}\zeta')} \\
 &=  p_{A(\zeta)} = T_\zeta^*T_\zeta \\
 & = p_{A(\zeta')} = T_{\zeta'}^*T_{\zeta'}.
 \end{align*}
 Now, since $\varphi_{\overline{\zeta}} = \varphi_{\overline{\zeta'}}$, we get from what we just have proved that we also
 have $ T_{\overline{\zeta}}^* T_{\overline{\zeta'}} = T_{\overline{\zeta'}}^*T_{\overline{\zeta'}}$, that is, 
 $T_\zeta T_{\zeta'}^* = T_{\zeta'}T_{\zeta'}^*$.
Since $T_\zeta$ and $T_{\zeta'}$ are partial isometries, we get that 
\[T_\zeta =T_\zeta T^*_\zeta T_\zeta = T_\zeta T^*_{\zeta'} T_{\zeta'} = T_{\zeta'} T^*_{\zeta'} T_{\zeta'} = T_{\zeta'},\] verifying (S5).  Hence it follows from the universal property of $C^*({\rm ZM}(\CC))$ that there is a homomorphism $f: C^*({\rm ZM}(\CC)) \to C^*_{{\rm Li},s}(\CC)$ satisfying $f(t'_\zeta) = T_\zeta$ for every  $\zeta \in \mathcal{Z}_\CC$. Since $g(T_\zeta) = t'_\zeta$ for every  $\zeta \in \mathcal{Z}_\CC$, we get that $f$ is the inverse of $g$.
 
\noindent
\eqref{thm li algebras 2}:  
We will construct a homomorphism inverse to $\mu$.  In $C^{*\, (\cup)}_{{\rm Li},s}(\CC)$ define $\{T_\zeta : \zeta \in \mathcal{Z}_\CC\}$ as follows:  for $\zeta = (\alpha_1,\beta_1, \ldots, \alpha_n,\beta_n) \in \mathcal{Z}_\CC$ let $T_\zeta = v_{\alpha_1}^* v_{\beta_1} \cdots v_{\alpha_n}^* v_{\beta_n}$.  We are here assuming (L1) -- (L3), (L6) and (L7) for the $v_\alpha$ and the $p_X$ with $\alpha \in \CC$ and $X \in \mathcal{J}^{\cup}(\CC)$; from the universal property of $\TT(\CC)$ we see that it suffices to deduce (S1) -- (S4) for the $T_\zeta$.  It follows from our definition of $T_\zeta$ that (S1) and (S2) hold.  We next claim that for $\zeta \in \mathcal{Z}_\CC$ we have $p_{A(\zeta)} = T_\zeta^* T_\zeta$.  This follows from (L7) just as in the previous proof.  Note also that by induction, (L6) holds for finite unions.  Now if $A(\zeta) = \bigcup_{i=1}^n A(\zeta_i)$, then
\[
T_\zeta^* T_\zeta
= p_{A(\zeta)}
= p_{\bigcup_{i=1}^n A(\zeta_i)} 
= \bigvee_{i=1}^n p_{A(\zeta_i)}
= \bigvee_{i=1}^n T_{\zeta_i}^* T_{\zeta_i},
\]
establishing (S3).  Finally, (L7) is equivalent to (S4).

\noindent
\eqref{thm li algebras 3}:  Suppose that $\CC$ is finitely aligned.  Using part \eqref{thm li algebras 2}, it suffices to construct a homomorphism inverse to $\mu \circ \rho_s$.  
By the universal property of $\TT(\CC)$ in the finitely aligned case, we need only verify that the map $T:\CC \to C^{*\, (\cup)}_{\rm Li}(\CC)$ defined by $T_\alpha = v_\alpha$ is a representation of $\CC$, i.e., satisfies (1) -- (3) in  Definition \ref{spi rep}.
Lemma \ref{ortho proj} gives that $v_u$ is a projection for each $u \in \CC^0$, so
 it follows from (L1) that $v_u = v_u^* v_u = p_{u \CC}$.  Now it follows that
\[
T_\alpha^* T_\alpha = v_\alpha^* v_\alpha = p_{s(\alpha)\CC} = v_{s(\alpha)} = T_{s(\alpha)}
\]
for each $\alpha \in \CC$, proving (1).  Next, (2) follows immediately from (L2).  Finally,  let $\alpha, \beta \in \CC$. Note that by (L1) and (L5) we have
\[
v_\alpha v_\alpha^* =  v_\alpha (v_\alpha^* v_\alpha) v_\alpha^* = v_\alpha p_{s(\alpha)\CC} v_{\alpha}^* = p_{\alpha \CC}.
\]
Then
\begin{align*}
T_\alpha T_\alpha^* T_\beta T_\beta^*
&= v_\alpha v_\alpha^* v_\beta v_\beta^*
= p_{\alpha \CC} p_{\beta \CC}
= p_{\alpha \CC \cap \beta \CC}, \text{ by (L4),} \\
&= p_{\bigcup_{\gamma \in \alpha \vee \beta} \gamma \CC}
= \bigvee_{\gamma \in \alpha \vee \beta} p_{\gamma \CC}, \text{ by (L6),} \\
&= \bigvee_{\gamma \in \alpha \vee \beta} T_\gamma T_\gamma^*,
\end{align*}
verifying (3).

\noindent
\eqref{thm li algebras 4}:  Let $\CC$ be the following finitely aligned left cancellative small category (actually a 2-graph):
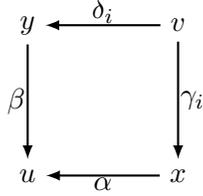
\[
\begin{tikzpicture}[scale=2]

\node (0_1) at (0,1) [rectangle] {$y$};
\node (0_0) at (0,0) [rectangle] {$u$};
\node (1_1) at (1,1) [rectangle] {$v$};
\node (1_0) at (1,0) [rectangle] {$x$};

\draw[-latex,thick] (1_0) -- (0_0) node[pos=0.5, inner sep=0.5pt, anchor=north] {$\alpha$};
\draw[-latex,thick] (0_1) -- (0_0) node[pos=0.5, inner sep=0.5pt, anchor=east] {$\beta$};
\draw[-latex,thick] (1_1) -- (1_0) node[pos=0.5, inner sep=0.5pt, anchor=west] {$\gamma_i$};
\draw[-latex,thick] (1_1) -- (0_1) node[pos=0.5, inner sep=0.5pt, anchor=south] {$\delta_i$};

\end{tikzpicture}
\]
where $i = 1$, 2, and $\alpha \gamma_i = \beta \delta_i$.  We construct a representation $\{V_\mu : \mu \in \CC \}$, $\{p_X : X \in \mathcal{J}(\CC) \}$ satisfying (L1) -- (L5), but not (L6).  We let all Hilbert spaces not defined by the following equations be isomorphic to some fixed Hilbert space (of arbitrary dimension, e.g.~dimension one).
\begin{align*}
H_x &:= H_{\gamma_1} \oplus H_{\gamma_2} \oplus H_x' \\
H_y &:= H_{\delta_1} \oplus H_{\delta_2} \oplus H_y' \\
H_u &:= H_{\alpha \gamma_1} \oplus H_{\alpha \gamma_2} \oplus H_u'.
\end{align*}
For each edge $\mu$ in $\CC$ we will let $V_\mu$ be a partial isometry with initial space $H_{s(\mu)}$.  Choose these so that
\begin{align*}
V_{\gamma_i}(H_v) &= H_{\gamma_i} \\
V_{\delta_i}(H_v) &= H_{\delta_i} \\
V_\alpha(H_{\gamma_i}) &= H_{\alpha \gamma_i} \\
V_\alpha(H_x') &= H_u' \\
V_\beta|_{H_{\delta_i}} &= V_\alpha V_{\gamma_i} V_{\delta_i}^* \\
V_\beta (H_y') &= H_u' \\
V_w &= I_{H_w},\ w \in \CC^0 \\
V_{\alpha \gamma_i} &(= V_{\beta \delta_i}) = V_\alpha V_{\gamma_i}.
\end{align*}
It is straightforward to check that $V_\mu V_\nu = V_{\mu \nu}$ if $s(\mu) = r(\nu)$, and equals 0 otherwise.  Thus (L1) and (L2) are satisfied.  In order to consider (L3) -- (L6) we must define the projections associated to constructible right ideals.  First we list all such ideals (apart from the empty set):
\begin{align*}
u \CC &= \{ u, \alpha, \beta, \alpha \gamma_1, \alpha \gamma_2 \} \\
\alpha \CC &= \{ \alpha, \alpha \gamma_1, \alpha \gamma_2 \} \\
\beta \CC &= \{ \beta, \alpha \gamma_1, \alpha \gamma_2 \} \\
\alpha \alpha^{-1} \beta \CC &= \{ \alpha \gamma_1, \alpha \gamma_2 \} \\
\alpha \gamma_i \CC &= \{ \alpha \gamma_i \} \\
x \CC &= \{ x, \gamma_1, \gamma_2 \} \\
\alpha^{-1} \beta \CC &= \{ \gamma_1, \gamma_2 \} \\
\gamma_i \CC &= \{ \gamma_i \} \\
y \CC &= \{ y, \delta_1, \delta_2 \} \\
\beta^{-1} \alpha \CC &= \{ \delta_1, \delta_2 \} \\
\delta_i \CC &= \{ \delta_i \} \\
v \CC &= \{ v \}.
\end{align*}
For $X = \mu \CC$ with $\mu \in \CC$ we let $p_X := V_\mu V_\mu^*$.  We let $p_\varnothing = 0$ (so that (L3) holds), and for the remaining ideals we set
\begin{align*}
p_{\alpha^{-1} \beta \CC} &= V_x \\
p_{\beta^{-1} \alpha \CC} &= V_y \\
p_{\alpha \alpha^{-1} \beta \CC} &= V_u.
\end{align*}
Note that $V_\alpha V_\alpha^* = V_{r(\alpha)}$ and $V_\beta V_\beta^* = V_{r(\beta)}$.  Using this it is straightforward to verify (L4) and (L5).  Also note that in this example, $\mathcal{J}(\CC) = \mathcal{J}^{(\cup)}(\CC)$.  Moreover,
\[
p_{\gamma_1\CC} + p_{\gamma_2\CC} = I_{H_{\gamma_1} \oplus H_{\gamma_2}} \not= I_{H_x} = (T_\alpha^{-1} T_\beta) (T_\alpha^{-1} T_\beta)^* = p_{\alpha^{-1}\beta \CC} = p_{\gamma_1\CC \cup \gamma_2\CC},
\]
so that (L6) does not hold.  Therefore $\pi^{(\cup)}$ is not one-to-one for this example.

It is also straightforward to verify that if $\zeta \in \mathcal{Z}_\CC$ is such that $\varphi_\zeta = \text{id}_{A(\zeta)}$, and $A(\zeta) \not= \varnothing$, then $\zeta$ must be a concatenation of zigzags from the following list:  $(\mu,\mu)$ for $\mu \in \CC$, $(\alpha,\beta,\beta,\alpha)$, $(\beta,\alpha,\alpha,\beta)$, $(\alpha,r(\alpha),r(\alpha),\alpha)$, $(r(\alpha),\alpha,\alpha,r(\alpha))$, and similarly for $\beta$.  It is easily seen that $V_\zeta = p_{A(\zeta)}$ for these, and hence (L7) holds.  Therefore $\rho^{(\cup)}$ is not one-to-one for this example. 
Since $g$ and $\mu$ are isomorphisms, this implies that $q$ is not one-to-one for this example.

\noindent
\eqref{thm li algebras 5}:  Let $\CC$ be the following finitely aligned left cancellative small category (actually a 2-graph):
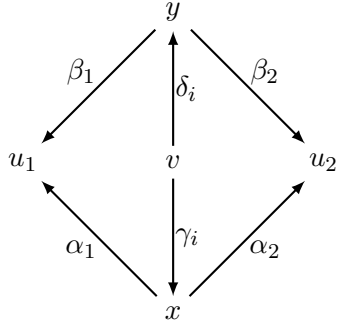
\[
\begin{tikzpicture}[scale=2]

\node (0_1) at (0,1) [rectangle] {$u_1$};
\node (1_2) at (1,2) [rectangle] {$y$};
\node (1_1) at (1,1) [rectangle] {$v$};
\node (1_0) at (1,0) [rectangle] {$x$};
\node (2_1) at (2,1) [rectangle] {$u_2$};

\draw[-latex,thick] (1_0) -- (0_1) node[pos=0.5, inner sep=0.5pt, anchor=north east] {$\alpha_1$};
\draw[-latex,thick] (1_0) -- (2_1) node[pos=0.5, inner sep=0.5pt, anchor=north west] {$\alpha_2$};
\draw[-latex,thick] (1_2) -- (0_1) node[pos=0.5, inner sep=0.5pt, anchor=south east] {$\beta_1$};
\draw[-latex,thick] (1_2) -- (2_1) node[pos=0.5, inner sep=0.5pt, anchor=south west] {$\beta_2$};
\draw[-latex,thick] (1_1) -- (1_0) node[pos=0.5, inner sep=0.5pt, anchor=west] {$\gamma_i$};
\draw[-latex,thick] (1_1) -- (1_2) node[pos=0.5, inner sep=0.5pt, anchor=west] {$\delta_i$};

\end{tikzpicture}
\]
where $1 \le i \le n$, $n > 1$, and with identifications $\alpha_j \gamma_i = \beta_j \delta_i$ for all $i$, $j$.  We construct a representation $\{V_\mu : \mu \in \CC \}$, $\{p_X : X \in \mathcal{J}(\CC) \}$ satisfying (L1) -- (L5), but not (L7).  We let all Hilbert spaces not defined by the following equations be isomorphic to some fixed Hilbert space (of arbitrary dimension, e.g.~dimension one).
\begin{align*}
H_x &:= \bigoplus_i H_{\gamma_i} \oplus H_x' \\
H_y &:= \bigoplus_i H_{\delta_i} \oplus H_y' \\
H_{u_j} &:= \bigoplus_i H_{\alpha_j \gamma_i} \oplus H_j',\ j=1,\; 2.
\end{align*}
For each edge $\mu$ in $\CC$ we will let $V_\mu$ be a partial isometry with initial space $H_{s(\mu)}$.  Choose these so that
\begin{align*}
V_{\gamma_i}(H_v) &= H_{\gamma_i} \\
V_{\delta_i}(H_v) &= H_{\delta_i} \\
V_{\alpha_j}(H_{\gamma_i}) &= H_{\alpha_j \gamma_i} \\
V_{\beta_j}|_{H_{\delta_i}} &= V_{\alpha_j} V_{\gamma_i} V_{\delta_i}^* \\
V_{\alpha_j}(H_x') &= H_j' \\
V_{\beta_1}(H_y') &= H_j' \\
V_{\beta_2}|_{H_y'} &= V_{\alpha_2} V_{\alpha_1}^* V_{\beta_1} U^*,
\end{align*}
where $U \in \mathcal{U}(H_y')$ is a nontrivial (partial) unitary operator.  Now we define $V_w := I_{H_w}$ for each vertex $w \in \CC^0$, and  $V_{\alpha_j \gamma_i} := V_{\alpha_j} V_{\gamma_i}$ and $V_{\beta_j \delta_i} := V_{\beta_j} V_{\delta_i}$.  It is straightforward to check that $V_\mu V_\nu = V_{\mu \nu}$ if $s(\mu) = r(\nu)$, and equals 0 otherwise.  Thus (L1) and (L2) are satisfied.  Also note that
\[
V_{\beta_2}^* V_{\alpha_2} V_{\alpha_1} V_{\beta_1}^*|_{H_y'}
= (U V_{\beta_1}^* V_{\alpha_1} V_{\alpha_2}^*) V_{\alpha_2} V_{\alpha_1}^* V_{\beta_1}
= U,
\]
while $\varphi_{(\beta_2,\alpha_2,\alpha_1,\beta_1)} = \text{id}_{\{ \delta_i : 1 \le i \le n \}}$.  Thus (L7) does not hold.  In order to verify (L3), (L4), and (L5) we must define the projections associated to constructible ideals.  First we list all such ideals (apart from the empty set):
\begin{align*}
u_j \CC &= \{ u_j, \alpha_j, \beta_j, \alpha_j\gamma_1, \ldots, \alpha_j\gamma_n \} \\
\alpha_j \CC &= \{ \alpha_j, \alpha_j\gamma_1, \ldots, \alpha_j\gamma_n \} \\
\beta_j \CC &= \{ \beta_j, \alpha_j\gamma_1, \ldots, \alpha_j\gamma_n \} \\
\alpha_j \alpha_j^{-1} \beta_j \CC &= \{ \alpha_j\gamma_1, \ldots, \alpha_j\gamma_n \} \\
\alpha_j\gamma_i \CC &= \{ \alpha_j \gamma_i \} \\
x \CC &= \{ x, \gamma_1, \ldots, \gamma_n \} \\
y \CC &= \{ y, \delta_1, \ldots, \delta_n \} \\
\alpha_1^{-1} \beta_1 \CC &(= \alpha_2^{-1} \beta_2 \CC) = \{ \gamma_1, \ldots, \gamma_n \} \\
\beta_1^{-1} \alpha_1 \CC &(= \beta_2^{-1} \alpha_2 \CC) = \{ \delta_1, \ldots, \delta_n \} \\
\gamma_i \CC &= \{ \gamma_i \} \\
\delta_i \CC &= \{ \delta_i \} \\
v \CC &= \{ v \}.
\end{align*}
For $X = \mu \CC$ with $\mu \in \CC$ we let $p_X := V_\mu V_\mu^*$.  We let $p_\varnothing = 0$ (so that (L3) holds), and for the remaining ideals we set
\begin{align*}
p_{\alpha_j^{-1}\beta_j \CC} &:= V_x \\
p_{\beta_j^{-1} \alpha_j \CC} &:= V_y \\
p_{\alpha_j \alpha_j^{-1} \beta_j \CC} &:= V_{u_j}.
\end{align*}
Note that $V_{\alpha_j}V_{\alpha_j}^* = V_{r(\alpha_j)}$ and $V_{\beta_j} V_{\beta_j}^* = V_{r(\beta_j)}$.  Using this it is straightforward to verify (L4) and (L5).

(We mention that if $n=1$, this example does not contradict \eqref{thm li algebras 3}.  The reason is that in that case, $\alpha_1^{-1} \beta_1 \CC = \{ \gamma_1 \} = \gamma_1 \CC$, but $p_{\alpha_1^{-1} \beta_1 \CC} \not= p_{\gamma_1 \CC}$.  Thus if $n=1$ the operators $p_X$ are not well-defined.)
\end{proof}

\begin{rem} When $\mathcal{C}$ is a submonoid of a group, Norling has shown in \cite[Proposition 3.26]{norling}
 that there is an isomomorphism from $C_{{\rm Li},s}^*(\mathcal{C})$ onto $C^*(I_\ell(\mathcal{C}))$.  Since the left inverse hull $
 I_\ell(\mathcal{C})$
 coincides with ${\rm ZM}(\mathcal{C})$ in this case, Theorem \ref{thm li algebras} (i) generalizes this result.
\end{rem}

With a bit more work we may modify 
the examples exhibited in the proofs of Theorem \ref{thm li algebras} \eqref{thm li algebras 4} and \eqref{thm li algebras 5}
 to  be submonoids of groups.  We remark that in the following, the monoids are submonoids of groups and have no inverses, and thus are categories of paths.

\begin{prop} \label{prop not one to one}

The maps $\pi^{(\cup)}$, $\rho^{(\cup)}$, $q$ and $\pi_s$ are not generally one-to-one even if $\CC$ is a finitely aligned submonoid of a group.

\end{prop}

\begin{proof}
We will modify the examples used in Theorem \ref{thm li algebras} \eqref{thm li algebras 4} and \eqref{thm li algebras 5}.  We use the amalgamation procedure from \cite[Section 11]{Spi11} (all results except the last in that section are valid for arbitrary left cancellative small categories (\cite[Section 4]{jackpaths2}); however, these examples are 2-graphs, hence within the literal scope of the reference \cite{Spi11}).  We give a treatment that applies simultaneously to both of the examples.  Let $\Lambda$ be one of the 2-graphs appearing in the proofs of Theorem \ref{thm li algebras}\eqref{thm li algebras 4} and \eqref{thm li algebras 5}.  We will make use of the representations constructed there.   Let $\widetilde{\Lambda}$ be obtained by identifying all vertices of $\Lambda$.  We will let $\widetilde{u}$ denote the unique unit of $\widetilde{\Lambda}$.  By \cite[Lemma 11.2]{Spi11} each element of $\widetilde{\Lambda}$ (other than $\widetilde{u}$) has a unique {\it normal form} $(\mu_1, \ldots, \mu_m)$ characterized by the properties that $\mu_i \in \Lambda \setminus \Lambda^0$ and $s(\mu_i) \not= r(\mu_{i+1})$.  For $\mu \in \widetilde{\Lambda}$ and $w \in \Lambda^0$ let $H^{\mu,w}$ be a copy of $H_w$:  $H^{\mu,w} = \{ \xi^{\mu,w} : \xi \in H_w \}$.  For $\mu = (\mu_1, \ldots, \mu_m)$ (in normal form), $\nu \in \widetilde{\Lambda}$, $\xi \in H_w$, let
\[
\widetilde{V}_\mu \xi^{\nu,w}
= \begin{cases}
 (V_{\mu_n} \xi)^{(\mu_1, \ldots, \mu_{n-1}),r(\mu_n)}, &\text{ if } \nu = \widetilde{u},\ \mu \not= \widetilde{u},\text{ and } s(\mu_n) = w, \\
 \xi^{\mu\nu,w}, &\text{ otherwise.}
  \end{cases}
\]
We claim that $\widetilde{V}_\mu \widetilde{V}_\nu = \widetilde{V}_{\mu\nu}$.  We may as well assume that $\mu$, $\nu \not= \widetilde{u}$.  Let $\mu = (\mu_1,\ldots,\mu_m)$ and $\nu = (\nu_1,\ldots,\nu_n)$ in normal form.  Let $\eta \in \widetilde{\Lambda}$, $w \in \Lambda^0$, and $\xi \in H_w$.  We consider several cases.

\noindent
{\it Case (i)}. Suppose that $\eta = \widetilde{u}$ and $s(\nu_n) = w$.  Then 
\[\widetilde{V}_\nu \xi^{\eta,w} = (V_{\nu_n}\xi)^{(\nu_1,\ldots,\nu_{n-1}),r(\nu_n)}.\]  There are two subcases.  First, if $n=1$ and $s(\mu_m) = r(\nu_1)$ then
\[
\widetilde{V}_\mu \widetilde{V}_nu \xi^{\eta,w}
= \widetilde{V}_\mu (V_{\nu_n}\xi)^{\widetilde{u},r(\nu_n)}
= (V_{\mu_m} V_{\nu_n} \xi)^{(\mu_1,\ldots,\mu_{m-1}),r(\mu_m)}.
\]
On the other hand $\mu \nu = (\mu_1,\ldots,\mu_{m-1},\mu_m \nu_n)$, and hence
\[
\widetilde{V}_{\mu \nu} \xi^{\eta,w}
= (V_{\mu_m \nu_n} \xi)^{(\mu_1,\ldots,\mu_{m-1}),r(\mu_m)}
= (V_{\mu_m} V_{\nu_n} \xi)^{(\mu_1,\ldots,\mu_{m-1}),r(\mu_m)}
\]
since $V$ is a representation.  Second, if $n > 1$ or $s(\mu_m) \not= r(\nu_1)$, then
\[
\widetilde{V}_\mu \widetilde{V}_\nu \xi^{\eta,w}
=\widetilde{V}_\mu (V_{\nu_n} \xi)^{(\nu_1,\ldots,\nu_{n-1}),r(\nu_n)}
= (V_{\nu_n} \xi)^{\mu(\nu_1,\ldots,\nu_{n-1}),r(\nu_n)}.
\]
On the other hand, the normal form of $\mu \nu$ has $\nu_n$ as its last piece, hence
\[
\widetilde{V}_{\mu \nu} \xi^{\eta,w}
= (V_{\nu_n} \xi)^{\mu(\nu_1,\ldots,\nu_{n-1}),r(\nu_n)}.
\]
\noindent
{\it Case (ii)}. Suppose that $\eta \not= \widetilde{u}$ or $s(\nu_n) \not= w$.  Then $\widetilde{V}_\nu \xi^{\eta,w} = \xi^{\nu \eta,w}$.  Also $\nu \eta \not= \widetilde{u}$, so $\widetilde{V}_\mu \widetilde{V}_\nu \xi^{\eta,w} = \xi^{\mu \nu \eta,w}$.  On the other hand, writing $\mu \nu = (\gamma_1, \ldots, \gamma_k)$ in normal form, we have that $s(\gamma_k) = s(\nu_n)$, and hence that $\eta \not= \widetilde{u}$ or $s(\gamma_k) \not= w$.  Therefore $\widetilde{V}_{\mu \nu} \xi^{\eta,w} = \xi^{\mu \nu \eta,w}$.

Now it is straightforward to see that the amalgamated versions exhibit the same phenomena as the versions in the proof of Theorem \ref{thm li algebras}.  Since $\widetilde{\Lambda}$ has a single unit, it is a monoid.  In fact $\widetilde{\Lambda}$ is a submonoid of a (finitely generated) free group.  In the case of the example in the proof of Theorem \ref{thm li algebras}\eqref{thm li algebras 4}, let $G$ be the free group with generators $\alpha$, $\beta$, $\gamma_1$, $\gamma_2$.  Then $\widetilde{\Lambda}$ is the submonoid generated by the four generators together with $\beta^{-1}\alpha\gamma_1$ and $\beta^{-1}\alpha \gamma_2$.  In the case of the example in the proof of Theorem \ref{thm li algebras}\eqref{thm li algebras 5}, let $G$ be the free group with generators $\alpha_1$, $\alpha_2$, $\beta_1$, $\gamma_1$, $\ldots$, $\gamma_n$.  Then $\widetilde{\Lambda}$ is the submonoid generated by the $n+3$ generators together with $\alpha_2 \alpha_1^{-1} \beta_1$, $\beta_1^{-1}\alpha_1\gamma_1$, $\ldots$, $\beta_1^{-1}\alpha_1 \gamma_n$.
\end{proof}

\begin{rem}

The example used in Theorem \ref{thm li algebras} \eqref{thm li algebras 5} can be enlarged by letting $n = \infty$ to give a nonfinitely aligned 2-graph.  An analogous argument to the one given there shows that $\rho_s$ is not generally one-to-one if $\CC$ is not finitely aligned.  Moreover, an argument analogous to the one in Proposition \ref{prop not one to one} shows that there is a nonfinitely aligned submonoid of a group for which $\rho_s$ is not one-to-one.

\end{rem}
 
\begin{rem}
When $\mathcal{C}$ is a left cancellative monoid, a certain quotient $C^*$-algebra $\mathcal{Q}(\mathcal{C})$ of $C^*_{\rm Li}(\mathcal{C})$, called the {\em boundary quotient} of $C^*_{{\rm Li}}(\mathcal{C})$, is introduced in \cite[Remark 5.5]{BRRW}. A similar quotient may be defined for a left cancellative small category $\CC$.
Recall from the beginning of Subsection \ref{subsec li algebras} that every set $A(\zeta) \in \mathcal{J}(\CC)$ has a range in $\CC^0$: $r(A(\zeta)) = s(\zeta)$.  We will write $\mathcal{J}(\CC)v = \{ E \in \mathcal{J}(\CC) : r(E) = v \}$.  For $v \in \CC^0$,
let us say that a finite nonempty subset $ F $ of $ \mathcal{J}(\CC)v$ is a {\em foundation set} if for each $Y \in  \mathcal{J}(\CC)v$
 there exists $X \in F$ with $ X \cap Y  \neq \varnothing$. 
(This coincides with the notion of {\em boundary cover} from \cite[Definition 10.6]{jackpaths2}.)
 We may then define the boundary quotient $\mathcal{Q}(\CC)$ of $C^*_{{\rm Li}}(\mathcal{C})$ as the universal $C^*$-algebra generated by partial isometries $\{v_\alpha\}_{\alpha \in \CC}$ and projections $\{e_X\}_{ X \in\mathcal{J}(\CC)}$ satisfying relations (L1) -- (L5)  and also 
\[\prod_{X\in F}(1_v-e_X)=0 \, \,  \text{   for all foundation sets  } F\, \subseteq \mathcal{J}(\CC),
\]
where $1_v := e_{v\CC}$. (Note that the above condition is equivalent to $1_v = \bigvee_{X \in F} e_X$, i.e.~to condition (5) of \cite[Definition 10.9]{jackpaths2}.)

 An interesting problem is whether $\mathcal{Q}(\CC)$ is isomorphic to $\OO(\CC)$ when $\CC$ is finitely aligned (or even singly aligned).
A result pointing towards a positive answer is provided by \cite[Theorem 3.7]{starling15}, where Starling shows that when $\CC$ is a singly aligned left cancellative monoid (i.e., is a right LCM in the terminology used in \cite{starling15}), then $\mathcal{Q}(\CC)$ is isomorphic to the tight $C^*$-algebra of the left inverse hull of $\CC$, hence to $\OO(\CC)$ by Corollary \ref{C-J}. 
\end{rem}

\begin{rem} Given a finitely aligned left cancellative small category $\mathcal{C}$, there are many other natural questions to investigate in the future. We mention a few here. What can be said about the nuclearity of any of the
$C^*$-algebras associated to $\mathcal{C}$?  What kind of conditions will ensure that (some of) the canonical maps between the $C^*$-algebras  associated to $\mathcal{C}$ are isomorphisms? When is $\mathcal{O}(\mathcal{C})$ simple? When is it simple and purely infinite? 
 For any of these questions, an answer valid only in the singly aligned case would already be interesting. 
\end{rem}


\bibliographystyle{amsplain}

\end{document}